\documentclass{article}
\usepackage{
amsmath,
amsfonts,
latexsym, 
amsrefs,
amssymb
}

\usepackage{bm}

\usepackage{tocloft}

\usepackage{cite}

\usepackage{pstricks}

\usepackage{subcaption}

\usepackage{tikz}
\usepackage{tikz,fullpage}
\usetikzlibrary{arrows,%
                petri,%
                topaths, automata}%
\usepackage{tkz-berge}

\usepackage{graphicx}
\DeclareMathOperator{\fc}{fc}

\usepackage{enumerate}
\usepackage{mathrsfs}
\usepackage{wrapfig}

\usepackage{color}

\numberwithin{equation}{section}
\newcommand{\bla}{\bm{\lambda}}
\newcommand{\boeta}{\bm{\eta}}
\newcommand{\boxi}{\bm{\xi}}

\newcommand{\boldu}{\mbox{\boldmath $u$}}

\newcommand{\pt}{\partial}
\newcommand{\rd}{{\rm d}}

\newcommand{\ba}{{\bf{a}}}

\newcommand{\beq}{\begin{equation}}
\newcommand{\bEq}{\end{equation}}
\newcommand{\del}{\partial}

\newcommand{\bu}{{\bf{u}}}
\newcommand{\bv}{{\bf{v}}}
\newcommand{\bw}{{\bf{w}}}

\newcommand{\al}{\alpha}

\newcommand{\be}{\begin{equation}}
\newcommand{\ee}{\end{equation}}

\newcommand{\e}{{\varepsilon}}

\newcommand{\la}{\lambda}

\newcommand{\om}{{\omega}}

\newcommand{\LL}{{\rm L}}

\newcommand{\cN}{{\cal N}}

\newcommand{\rU}{{\rm U}}

\usepackage{amsmath}
\usepackage{amssymb}
\usepackage{amsthm}

\def\RR{{\mathbb R}}

\renewcommand{\b}[1]{\bm{\mathrm{#1}}}

\renewcommand{\cal}{\mathcal}

\newcommand{\wt}{\widetilde}

\newcommand{\ii}{\mathrm{i}}

\newcommand{\s}{\mspace{-0.9mu}} 	
\newcommand{\col}{\mathrel{\mathop:}}

\newcommand{\deq}{\mathrel{\mathop:}=}

\renewcommand{\epsilon}{\varepsilon}
\renewcommand{\leq}{\leqslant}
\renewcommand{\geq}{\geqslant}

\renewcommand{\le}{\leq}
\renewcommand{\ge}{\geq}

\renewcommand{\P}{\mathbb{P}}
\newcommand{\E}{\mathbb{E}}
\newcommand{\R}{\mathbb{R}}
\newcommand{\C}{\mathbb{C}}
\newcommand{\N}{\mathbb{N}}
\newcommand{\Z}{\mathbb{Z}}

\newcommand{\pb}[1]{\bigl({#1}\bigr)}
\newcommand{\pB}[1]{\Bigl({#1}\Bigr)}

\DeclareMathOperator{\diag}{diag}
\DeclareMathOperator{\tr}{Tr}
\DeclareMathOperator{\var}{Var}
\DeclareMathOperator{\supp}{supp}

\DeclareMathOperator{\re}{Re}
\DeclareMathOperator{\im}{Im}

\DeclareMathOperator{\OO}{O}
\DeclareMathOperator{\oo}{o}
\DeclareMathOperator{\UU}{U}

\theoremstyle{plain} 
\newtheorem{theorem}{Theorem}[section]
\newtheorem*{theorem*}{Theorem}
\newtheorem{lemma}[theorem]{Lemma}
\newtheorem{assumption}[theorem]{Assumption}
\newtheorem*{lemma*}{Lemma}
\newtheorem{corollary}[theorem]{Corollary}
\newtheorem*{corollary*}{Corollary}
\newtheorem{proposition}[theorem]{Proposition}
\newtheorem*{proposition*}{Proposition}

\newtheorem{definition}[theorem]{Definition}
\newtheorem*{definition*}{Definition}

\newtheorem*{example*}{Example}
\newtheorem{remark}[theorem]{Remark}

\newtheorem*{remark*}{Remark}
\newtheorem*{remarks*}{Remarks}

\makeatletter
\renewcommand{\subsection}{\@startsection
{subsection}
{2}
{0mm}
{-\baselineskip}
{0 \baselineskip}
{\normalfont\bf\itshape}} 
\makeatother

\usepackage{dsfont}
\usepackage{stmaryrd}

\newcommand{\nc}{\normalcolor}

\def\bN{{\mathbb N}}

\newcommand{\bq}{{\bf q}}

\renewcommand{\b}[1]{\boldsymbol{\mathrm{#1}}}

\def\@empty{}

\def\author#1{\par
    {\centering{\authorfont#1}\par\vspace*{0.05in}}
}

\def\titlefont{\fontsize{13}{15}\bfseries\boldmath\selectfont\centering{}}
\def\authorfont{\fontsize{13}{15}}
\def\abstractfont{\fontsize{8}{10}}

\let\affiliationfont\rhfont

\def\address#1{\par
    {\centering{\affiliationfont#1\par}}\par\vspace*{11pt}
}

\def\keywords#1{\par
    \vspace*{8pt}
    {\authorfont{\leftskip18pt\rightskip\leftskip
    \noindent{\it\small{Keywords}}\/:\ #1\par}}\vskip-12pt}

\def\body{
\setcounter{footnote}{0}
\def\thefootnote{\alph{footnote}}
\def\@makefnmark{{$^{\rm \@thefnmark}$}}
}

\def\title#1{
    \thispagestyle{plain}
    \vspace*{-14pt}
    \vskip 79pt
    {\centering{\titlefont #1\par}}
    \vskip 1em
}

\renewenvironment{abstract}{\par
    \vspace*{6pt}\noindent
    \abstractfont
    \noindent\leftskip18pt\rightskip18pt
}{
  \par}

\renewcommand{\b}[1]{\boldsymbol{\mathrm{#1}}}

\usepackage{tocloft}

\makeatletter
\renewcommand{\section}{\@startsection
{section}
{1}
{0mm}
{-2\baselineskip}
{2\baselineskip}
{\normalfont\large\scshape\centering}}
\makeatother

\newcommand{\tnorm}[1]{{\left\vert\kern-0.25ex\left\vert\kern-0.25ex\left\vert #1 
    \right\vert\kern-0.25ex\right\vert\kern-0.25ex\right\vert}}

\begin{document}

\title{Random band matrices in the delocalized phase, I:\\  Quantum unique ergodicity and universality}

\vspace{1.2cm}
\noindent\begin{minipage}[b]{0.33\textwidth}

 \author{P. Bourgade}

\address{Courant Institute\\
 bourgade@cims.nyu.edu}
 \end{minipage}
\begin{minipage}[b]{0.33\textwidth}

 \author{H.-T. Yau}

\address{Harvard University\\
htyau@math.harvard.edu}

 \end{minipage}
\begin{minipage}[b]{0.3\textwidth}

 \author{J. Yin}

\address{University of California, Los Angeles\\
jyin@math.ucla.edu}

 \end{minipage}

\begin{abstract}
Consider  $N\times N$ symmetric one-dimensional random band matrices with 
general distribution of the entries and band width    $W \ge N^{3/4+\e}$ for any $\e>0$. 
 In the bulk of the spectrum and in the large $N$ limit, we obtain the following results.

\vspace{0.1cm}

{\addtolength{\leftskip}{1.5em} \setlength{\parindent}{0em}

\vspace{0.1cm}

\makebox[1.5em][l]{(i)} The semicircle law holds up to the scale $N^{-1+\e}$ for any $\e>0$.

\vspace{0.1cm}

\makebox[1.5em][l]{(ii)} The  eigenvalues locally converge to the point process given by the Gaussian orthogonal ensemble at any fixed energy.

\vspace{0.1cm}

\makebox[1.5em][l]{(iii)}   All eigenvectors are delocalized, meaning their ${\rm L}^\infty$ norms are all simultaneously bounded by $N^{-\frac{1}{2}+\e}$ (after normalization in ${\rm L}^2$)
with overwhelming probability, for any $\e>0$.

\vspace{0.1cm}

\makebox[1.5em][l]{(iv)} Quantum unique ergodicity holds, in the sense that  the  local    ${\rm L}^2$ mass of eigenvectors becomes equidistributed  with overwhelming probability.
\par 
}
\vspace{0.1cm}

\noindent 

We extend the mean-field reduction method  \cite{BouErdYauYin2017}, which required  $W=\Omega(N)$,  
to the current setting  $W \ge N^{3/4+\e}$.  Two new  ideas  are:   (1) A new estimate on the ``generalized  resolvent" of band matrices when  $W \ge N^{3/4+\e}$.  Its proof, along with an improved fluctuation average estimate,  will be presented  in parts 2 and 3 of this series  \cite {BouYanYauYin2018,YanYin2018}.  (2) A strong  (high probability) version of the  quantum unique ergodicity property of random matrices.   
For its proof,  we   construct  perfect matching observables of eigenvector 
overlaps  and show  they   satisfying   the eigenvector moment flow equation \cite{BouYau2017} under the matrix  Brownian motions.

\end{abstract}


\vspace{1cm}

\tableofcontents

{\let\thefootnote\relax\footnotetext{\noindent The work of P.B. is partially supported by the NSF grant DMS\#1513587. 
The work of H.-T. Y. is partially supported by NSF Grant  DMS-1606305 and a Simons Investigator award.
The work of  J.Y. is partially supported by the NSF grant DMS\#1552192.}}

\newpage

\section{Introduction}

\subsection{Random matrices beyond mean field.}\ 
In Wigner's vision, random matrices play the role of a mean-field model for  large  quantum systems of high complexity.   His paradigm has been confirmed with significant progress 
in understanding the  universal  behavior of many random graph and random matrix models. However,  regarding his core thesis that  
random matrix can be used to model non mean-field  systems, our understanding is much more limited.  Even for one of the simplest non mean-field models, the random Schr\"odinger
operator,  there is no result concerning  the existence of the delocalized regime, in which random matrix statistics are expected to hold.

A slightly more tractable model is  the  \emph{ random band matrix}  characterized
by the property that $H_{ij}$   becomes negligible if $\mbox{dist}(i,j)$ 
exceeds
a  parameter, $W$, called the \emph{band width}.  In general, $i,j$ are lattice points in $\Z^d$ but in this article we consider only the case $d=1$. 
Based on numerics, it was conjectured  \cite{ConJ-Ref1, ConJ-Ref2}  that the eigenvectors  of band matrices satisfy a localization-delocalization transition,  in the bulk of the spectrum, with a corresponding sharp transition for the eigenvalues distribution\cite{ConJ-Ref4}:
\begin{enumerate}[(i)]
\item for  $W \gg \sqrt N$, delocalization and  Gaussian orthogonal ensemble (GOE) spectral statistics hold;
\item  for  $W \ll \sqrt N$, eigenstates are localized and the eigenvalues converge to a Poisson point process. 
\end{enumerate}
This transition was also supported by heuristic arguments \cite{ConJ-Ref6} and a nonrigorous supersymmetry method \cite{fy}.

There have been many partial results  concerning localization-delocalization for band matrices. 
For general distribution of the matrix entries, 
localization of eigenvectors  was first shown for $W\ll N^{1/8}$ \cite{Sch2009}, and improved to $W\ll N^{1/7}$ for Gaussian entries  \cite{PelSchShaSod}.
 Delocalization was proved in some averaged sense, for
 $W\gg N^{6/7}$ in  \cite{ErdKno2013},  $W\gg N^{4/5}$  in \cite{ErdKnoYauYin2013},  $W\gg N^{7/9}$  in \cite{HeMarc2018}. 
The Green's function  was controlled
 down to the scale $\im z\gg W^{-1}$ in \cite{ErdYauYin2012Univ}, implying
a lower bound of order $W$ for the localization length of all eigenvectors.
We mention also that at the edge of the spectrum, the transition for $1d$ band matrices (with critical exponent $N^{5/6}$) was understood in \cite{Sod2010}, thanks to the method of moments.

When the entries of band matrices are Gaussian with some specific covariance profile,
one can apply  supersymmetry techniques (see \cite{Efe1997} and  \cite{Spe} for  overviews). With this method, for $d=3$, precise estimates on the
density of states  \cite{DisPinSpe2002} were first obtained.
Then,  
random matrix local spectral statistics were proved for 
 $W=\Omega(N)$ \cite{Sch2014},  and  delocalization was obtained for all
eigenvectors  when 
  $W\gg N^{6/7}$ and the first four moments of the matrix entries match the Gaussian ones \cite{BaoErd2015}
(these  results assume complex entries and hold in part of the bulk).
Still with the supersymmetry technique, a transition  around $N^{1/2}$ was proved in \cite{SchMT,Sch1}, concerning moments of characteristics polynomials.

\subsection{Mean Field reduction and  quantum unique ergodicity. }\ 
The main difficulties to analyze spectral properties  of  band matrices with general entries are two-folds.
\begin{enumerate}[(i)]
\item There is currently no effective diagrammatical method 
to estimate the Green's function when  $\im z \ll  W^{-1}$, while delocalization of eigenvectors requires estimates up to $\im z \gg  N^{-1}$.

 \item
   For the universality of local spectral statistics, 
the comparison method used for mean-field models does not apply to band matrices since the majority of matrix elements (effectively) vanish.  
\end{enumerate}

In an earlier paper \cite{BouErdYauYin2017}, we proposed a {\it mean-field reduction} method  to prove universality of local spectral statistics for  
band matrices with $W=\Omega(N)$.  
This method  relies on a notion much stronger than delocalization, the 
probabilistic  quantum unique ergodicity (QUE). Historically,  QUE was introduced by  Rudnick and Sarnak \cite{RudSar1994} 
asserting  that for negatively curved compact Riemannian manifolds, {\it all} high energy Laplacian eigenfunctions become completely flat. 
Quantum ergodicity, essentially an averaged version of QUE, had previously been proved for more general manifolds
\cite{Shn1974,Col1985,Zel1987}. 
For  $d$-regular graphs, 
the eigenvectors of the discrete Laplacian also satisfy quantum ergodicity,  under certain assumptions  on the injectivity radius and  spectral gap of the adjacency matrices \cite{AnaLeM2013}.

A probabilistic version of QUE was proposed and proved for Wigner matrices
in \cite{BouYau2017}. To state it, let  $H$ be a size $N$ random matrix with eigenvectors $\b\psi_j$ associated to eigenvalues $\la_j$.  Then, there exists $\e>0$ such that
for any  deterministic $1\leq j\leq N$ and $I\subset\llbracket 1,N\rrbracket$, for any $\delta>0$ we have
\begin{equation}\label{eqn:QUEintro}
 \P\left( \Big | \sum_{i\in I} |\psi_j(i)| ^2  - \frac{|I|}{N} \Big | \ge \delta\right)\le 
    N^{- \e}/\delta^2.  
\end{equation}

To explain the mean field reduction, 
we 
block-decompose a band matrix $H$ and its eigenvectors:
\be\label{H0}
   H=  \begin{pmatrix} A  & B^* \cr B & D  \end{pmatrix}, \quad
   \b\psi_j=   \begin{pmatrix}\b w_j \cr \b p_j \end{pmatrix},
\ee
where $A$ is a $W\times   W$ Wigner matrix. From the eigenvector equation $H \b\psi_j = \lambda_j \b\psi_j$, 
\be\label{1100}
  Q_{\lambda_j}  \bw_j  = \lambda_j  \bw_j, \quad  {\rm where}\    Q_e =  A
 - B^* \frac{1}{D-e}B.
\ee
Thus $\bw_j$  is an eigenvector to  $Q_e$  with eigenvalue $\lambda_j$ when   $e= \lambda_j$. 
The basic observation from the earlier paper \cite{BouErdYauYin2017} can be summarized as follows. 
Suppose that the probabilistic  QUE for the eigenvectors of $H$ holds. Then the eigenvalues to $H$ 
near a fixed energy $E$ can be reconstructed  from the eigenvalues of $Q_e$ near the origin with $e$ near $E$.  Thus if we can prove the spectral universality for $Q_e$,  the same statement holds  for $H$. 
On the other hand, to  establish 
QUE for the band matrix $H$,
assume first that it holds for the $W\times W$ operator $Q_e$. If we can substitute 
$e$ by $\lambda_j$, then the eigenvector ${\b\psi}_j$ is flat in the first $W$ coordinates. Clearly, we can stitch together 
the flatnesses  of ${\b\psi}_j$ in
sufficiently many  windows of size $W$ to establish the global flatness of  ${\b \psi}_j$ provided that the error in each window is sufficiently small.

To summarize, the mean field reduction method  reduces the universality and QUE  for  the band matrix $H$ to those  of 
$Q_e$. Thanks to the recent progress on these topics \cite{BouHuaYau2017,LanYau2015,LanSosYau2016}, the inputs to prove  these properties require precise 
estimates on the  Green's function $ (Q_e-z)^{-1}$  only for 
$\im z \sim  N^{-\e}$.    For probabilistic QUE, we also need to establish 
the error probability  in the sense of ``very  high probability".  In the following, we start with a discussion on the  Green's function $ (Q_e-z)^{-1}$.

\subsection{Generalized Green's Functions.}\   
It is clear that,  if we  estimate the Green's function $ (Q_e-z)^{-1}$ directly,  some  bound on 
the matrix $(D-e)^{-1}$ appearing  in $Q_e$ will be needed.   Since  $e$ is real, estimating $(D-e)^{-1}$
is clearly a much harder problem than estimating  the original Green's  function $ (H-z')^{-1}$. Fortunately, 
we only need this estimate  with $\im z \sim  N^{-\e}$. Clearly, one can  interpret 
 $(Q_e-z)^{-1}$ as the $W \times W$  corner of the {\it generalized  Green's function}  
\begin{equation}\label{eqn:general}
G(z, w )=  
\left(H- \begin{pmatrix}  z\,{\rm I}_{W}  & 0 \cr 0 & w\, {\rm I}_{N- W}  \end{pmatrix} \right)^{-1}
\end{equation}
when $w=e$. 
In \cite{BouErdYauYin2017}, we use   a somehow involved  induction argument and  
 an  uncertainty principle to estimate $G(z, e )$  for $W=\Omega(N)$.  In  this work, 
 we provide accurate estimates, Theorem \ref{LLniu},  on $G(z, e)$ for $\im z \sim  N^{-\e}$ when $W\gg N^{3/4}$. 
Our method is to derive   a self consistent equation  for  the (off-diagonal) entires of the generalized  Green's function 
(a similar equation for the standard Green's function  was called the $T$ equation  \cite{ErdKnoYauYin2013-band}).  
Notice that Ward's identity, which is instrumental in many random matrix estimations, 
is not valid for  generalized Green's functions.  More precisely, Ward's identity asserts  that for any Green's function of a Hermitian operator $H$, 
\be
\sum_j | G_{ij} (z)|^2 \le (\im z)^{-1} \im G_{ii}. 
\ee
For the generalized Green's function $G(z, w )$, the last property fails. Our strategy is to  establish an estimate on $\sum_j | G_{ij} (z)|^2$  by  successively decreasing  the imaginary part of $w$ and using repeatedly   the self-consistent  $T$ equation in each step. Besides overcoming this difficulty, we also devise   a new diagrammatic expansion in deriving  the   $T$ equation. Finally, we remark that the main   condition  $W\gg N^{3/4}$ is mainly used in   estimating $G(z, e )$. 
Besides extending  the region  of validity from $W=\Omega(N)$ to  $W\gg N^{3/4}$,  our current approach  allows the  estimate on 
$G(z, e )$ to be  completely independent  from all other arguments in this work (e.g., the mean field reduction). The proof of Theorem \ref{LLniu} will be delayed 
to  parts 2 and 3 of this series.

\subsection{Probabilistic QUE with high probability.}\  The  proof of  the quantum unique ergodicity  (\ref{eqn:QUEintro}) for  $Q_e$ in \cite{BouErdYauYin2017}  relies on two different tools.
\begin{enumerate}
\item A priori estimates on the Green's function $ (Q_e-z)^{-1}$ (for large $\im z$)  provide flatness of eigenvectors on average (quantum ergodicity). This a priori information is necessary to obtain the following.

\item The eigenvector moment flow from \cite{BouYau2017} is a random walk in a dynamic random environment whose relaxation means flatness of individual eigenvectors  (quantum unique ergodicity).

\end{enumerate}

\noindent We have just outlined our new  estimates on  the Green's function $ (Q_e-z)^{-1}$   for  $W\gg N^{3/4}$. 
The main new technique developed in this work concerns (ii): Theorem \ref{thmQUE} states that

\begin{center}
 {\it  Quantum ergodicity implies a strong  quantum unique ergodicity after adding a small GOE component.}
\end{center}
   
\noindent Compared to (\ref{eqn:QUEintro}), this new result is a strong  probabilistic QUE, as it first  allows
much more general observables of eigenvectors and is valid with probability $1-N^{-D}$ for any $D > 0$. Therefore all bulk eigenvectors are now simultaneously flat.
The proof of Theorem  \ref{thmQUE}  relies on a remarkable  combinatorial identity: the perfect matching observables defined in (\ref{feq}) satisfy the eigenvector moment flow parabolic equation, see Theorem \ref{thm:EMF}.

Thanks to this new strong, version of QUE, the eigenvectors of $Q_e$ are flat for all $e$ in a discrete subset
of size $N^C$ for any $C$ fixed.  Thus to establish flatness of $\b\psi_j$ on the first $W$ coordinates,   we only need to compare  eigenvectors of $ Q_e$ and  $ Q_{\lambda_j}$ for $|e-\lambda_j| \le N^{-C}$ with $C$ a large constant.
An eigenvector perturbation formula is enough to compute the difference between these eigenvectors, with sufficient a priori estimates given by a weak 
uncertainty principle as developed in \cite{BouErdYauYin2017}.

Therefore, our work presents an improvement from $W=\Omega(N)$ \cite{BouErdYauYin2017} to $W \gg N^{3/4}$ thanks to new results both on (i) and (ii).
As discussed in Remark \ref{64}, our hypothesis $W\gg N^{3/4}$ for delocalization comes from the generalized Green's function estimates (ii).  Heuristics for the transition at band width 
$N^{1/2}$ are given in the same remark.

\subsection{The model and  results.}\ 
All results in this paper apply to both real and complex band matrices. For the definiteness of notation,   we consider only  
the real symmetric case and we use the convention that  all eigenvectors are  real. In the following definition, 
$\Z_N$ denotes the set of residues mod $N$ so that our matrices are assumed to have periodic boundary condition.

\begin{definition}[Band matrix $H_N$ with bandwidth $W_N$]\label{jyyuan}
Let $H_N$ be a $N\times N$  matrix with real centered  entries ($H_{ij}$, $i,j\in \Z_N)$ which  are  
independent up to the  condition $H_{ij}= H_{ji}$.  We say that  $H_N$ is band matrix with bandwidth $W=W_N$ if 
\be\label{bandcw0}
s_{ij}:=\E |H _{ij}|^2{  =f(i-j)}
\ee
for some $f: \Z_N\to \R$ satisfying $\sum_{x\in\mathbb{Z}_N} f(x)=1$, and there exist
a small positive  constant $c_s $ and a large constant $C_s$ such that 
\begin{equation}\label{bandcw1}
c_s\, W^{-1} \cdot \mathds{1} _{|x| \le   W}\le f(x) \le C_s \,W^{-1} \cdot \mathds{1} _{|x| \le C_s W},\  x\in \Z_N.
\end{equation}
\end{definition}

The method in this paper also allows to treat cases with progressive decay of the variance away from the diagonal (e.g.  $f(x)\le C_s \,W^{-1} \cdot {\bf 1} _{|x| \le C_s \,W}$ instead of $f(x)\le C_s \,W^{-1} \cdot {\bf 1} _{|x| \le \,W}$), or variants with exponentially small mass
away from the band width. We work under the hypothesis (\ref{bandcw1}) for simplicity.\\

For technical  reasons we assume the following condition on the  fourth moment of the matrix entries: there is  $\e_m>0$ (here the subscript $m$ indicates  the moment condition) such that for $|i-j|\le W$,  
\be\label{cond: moment}
\min_{  |i-j|\le W}\left(\E \,\xi^4_{ij}-(\E \,\xi^3_{ij})^2-1 \right)\ge N^{-\e_m},
\ee
where $ \xi _{ij}:= H_{ij} (s_{ij})^{-1/2}$  is the normalized random variable with mean zero and variance one. 
 It is well-known that  
for any real random variable $\xi$ with mean zero and variance 1,  
$
\E \,\xi^4 -(\E \,\xi^3 )^2-1\ge 0
$
and the equality holds if and only if $\xi$ is a Bernoulli random variable (Lemma 28 of \cite{TaoVu2011}). Therefore, one simply has $\e_m=0$ when the $\xi_{ij}$'s ($|i-j|\leq W$) all have the same law, different from the Bernoulli distribution.
In the more general setting (\ref{cond: moment}), all our results are restricted to $0\leq \e_m<1/2$ because of the following condition (\ref{gwzN}).

We also assume that for some $\delta_d>0$ (subscript $d$ stands for ``decay") we have
\begin{equation}\label{eqn:subgaus}
\sup_{N,i,j}\E\left(e^{\delta_d W H_{ij}^2}\right)<\infty.
\end{equation}
This tail condition can be weakened to a finite high  moment condition.  We  assume (\ref{eqn:subgaus}) mainly for the convenience of presentation.  
The constants in the following theorems depend on the fixed parameters $c_s$, $C_s$, $\e_m$ and $\delta_d$,  in \eqref{bandcw1}, \eqref{cond: moment} and \eqref{eqn:subgaus},
but we will only  track of the dependence on $\e_m$.

\medskip

Denote the eigenvalues of $H$ by
$
\lambda_1\leq\dots\leq \lambda_N,
$
and let $(\b\psi_k)_{k=1}^N$ be the corresponding ${\rm L}^2$-normalized eigenvector, i.e., $H\b\psi_k=\lambda_k\b\psi_k$.
Thanks to the condition $\sum f(x)=1$,
it is known that the empirical spectral measure $\frac{1}{N}\sum_{k=1}^N\delta_{\lambda_k}$ converges
almost surely to the Wigner semicircle law with density
$$
\rho_{\rm sc}(x)=\frac{1}{2\pi}\sqrt{(4-x^2)_+}.
$$
The concept of localization/delocalization can be defined in many ways. For definiteness, we use the ${\rm L}^\infty$ norm. For any small constant $c>0$ and $\tau>0$, one expects that
\be\label{39m}
\P\Big(N^{-\tau}\le  \min (N, W^{2}) \|\b\psi_k\|^2_\infty \le  N^\tau\ {\rm for\,all}\ k\in \llbracket cN, (1-c)N\rrbracket \Big)=1-\oo(1),
\ee
meaning that a localization-delocalization phase transition occurs at $\log_N W=1/2$, where  $\log_N W=\log W/\log N$.
Our first result proves (\ref{39m}) in the delocalization regime $\log_N W  >3/4$.

  \begin{theorem}[Delocalization for  $\log_N W  >3/4$]\label{Main}
Let $(H_N)_{N\geq 1}$ be band matrices with band width $W_N$ satisfying  the conditions \eqref{cond: moment} and \eqref{eqn:subgaus}.  
 Recall that  $\e_m> 0$ is defined in \eqref{cond: moment}. 
Suppose that   for some constant $a>0$,  
\be\label{gwzN}
\log _N W\ge  \max\left(\frac{3}{4},\frac{1}{2}+\e_m\right)+a\nc.
\ee
For any (small) constants $\kappa , \tau>0$ and (large)  $D>0$, there exists $N_0$ such that  for all  $N\geq N_0$ we have 
\be\label{maindel}
\mathbb{P}\left(\|\b\psi_k\|^2_\infty\le N^{-1+\tau}\ \mbox{ for all } \; k\in \llbracket \kappa N, (1-\kappa)N\rrbracket \right)\geq 1-N^{-D}.
\ee
\end{theorem}
The above delocalization holds together with a local semicircle law down to the optimal scale.

  \begin{theorem}[Local semicircle law for $\log_N W  >3/4$]\label{lsc}
Under the same assumptions as Theorem \ref{Main}, there exists $\e>0$ such that 
for any (small) $\kappa$, $\tau>0$ and (large) $D>0$ there exists $N_0$  such that for any  $E_1, E_2\in [-2+\kappa, 2-\kappa]$ and any $N\geq N_0$ we have
 \be\label{SemiC}
\mathbb{P}\left( \left|\#\left\{\lambda_k\in [E_1,E_2] \right\}-N\int_{E_1}^{E_2}\rd \rho_{\rm sc}\right|<N^{\tau}+|E_1-E_2|N^{1-\e}\right)\geq 1-N^{-D}.
 \ee 
\end{theorem}

In the following fixed energy universality statement, we denote $\rho_H^{(k)}$ the $k$-point correlation function (understood in the sense of distributions) for the spectral measure of a $N\times N$ random matrix $H$.

 \begin{theorem}[Universality  for  $\log_N W  >3/4$] \label{Univ}
Under the same assumptions as Theorem \ref{Main},
for any  $\kappa >0$,  any integer  $k$ and any smooth test function $O \in \mathscr{C}^\infty ( \R^k )$ with compact support, 
there are   constants $c, C >0$ such that for any  $ |E| \le 2 - \kappa$  we have
\begin{equation}\label{e:Univ}
\left|\int_{\mathbb{R}^k} O (\ba) \rho_{H}^{(k)} \left( E + \frac{\ba}{ N \rho_{\rm sc} (E)}\right) \rd\ba
- \int_{\mathbb{R}^k} O (\ba) \rho_{\rm GOE}^{(k)} \left( E + \frac{\ba}{ N \rho_{\rm sc} (E) }   \right)\rd \ba\right| \leq  C N^{- c}.
\end{equation}
  \end{theorem}

For  the proof of Theorems \ref{Main}, \ref{lsc} and  \ref{Univ}, the first step is to show that  delocalization,
the local semicircle law, eigenvalues universality and quantum unique ergodicity  hold  under  the following additional assumption:
$H$ is a   Gaussian divisible  band matrix, i.e., there exists independent band matrices $ H_1$ and $H_2$ 
 with the same width $W$ and $c>0$ such that $H_1$ satisfies (\ref{cond: moment})  and (\ref{eqn:subgaus}), and  
 \be\label{sjui}
H= H_1 + H_2\ {\rm where}\ (H_{2})_{ij} =\mathds{1}_{|i-j|\le W}\cdot (1+\mathds{1}_{ij})^{1/2}\cdot \cal N(0, \; c \; W^{-1}N^{- \e_m }).
 \ee
Remember that $\e_m$ is defined in $(\ref{cond: moment})$. Here, $c$ is a small enough constant depending only on $\delta_d$ from (\ref{eqn:subgaus}).

  \begin{theorem}\label{lemma main3}
 Assume that $H$ is a band matrix of type \eqref{sjui}, with band width $W_N$ satisfying (\ref{gwzN}).
 \begin{enumerate}[(i)]
 \item
The eigenvectors are delocalized as in (\ref{maindel}).
 \item  
The eigenvalues
  satisfy the local semicircle law as in (\ref{SemiC}).
\item 
Fixed energy universality holds as in (\ref{e:Univ}).
\item
For any (small)  $\tau, \kappa > 0$,  and (large) $D>0$, there exists $N_0>0$  such that for any $N\geq N_0$ we have 
$$
\P\left(  \left|\frac {N} W\sum_{\alpha=\ell}^{\ell+W}
|\psi_j(\alpha)|^2 \nc-1\right|<N^{- \frac{3}{2}\nc a+\tau}\,
\mbox{for all}\ 1\le j,\ell \le N\;  \text{ such that }   \; |\lambda_j|\le 2-\kappa 
\right)\ge 1-N^{-D}, 
$$
where  $a>0$ was given in \eqref{gwzN} and all indices are defined  modulo $N$.
\end{enumerate} 
\end{theorem}

\subsection{Organization of the paper.}\
This work is essentially divided in two parts.

The first part (Sections \ref{sec:QUE1} and \ref{sec:QUE}) concerns quantum unique ergodicity for mean field blocks, and improves on the estimate (\ref{eqn:QUEintro}):
Theorem \ref{thmQUE} gives flatness of the eigenvectors with
overwhelming probability, and with optimal fluctuations scale for the ${\rm L}^2$ mass of eigenvectors on subsets of $\llbracket 1,N\rrbracket$. This result is the main technical novelty of our work.

The first aspect of the proof is algebraic (Section \ref{sec:QUE1}). A new function of the eigenvectors overlaps is defined in equation  (\ref{feq}), and it follows the eigenvector moment flow dynamics, see Theorem \ref{thm:EMF}. These dynamics of  {\it perfect matching observables}\ generalize an earlier observation from \cite{BouYau2017}. In this previous work, the eigenvectors evolution was related to a random walk in a dynamic random environment, after dimension reduction through projection on a given fixed direction. Projections can now occur on an arbitrary number of directions, see Remark \ref{rem:canon}. The proof of Theorem \ref{thm:EMF} is combinatorial and given at the end of Section \ref{sec:QUE1}.

The second aspect of the proof of Theorem \ref{thmQUE} is analytic (Section \ref{sec:QUE}). As proved by a sequence of maximum principles and approximations with short range dynamics, 
the eigenvector moment flow reaches equilibrium after some time depending on the initial condition.  This allows to identify the scale of the  perfect matching observables.
Our proof is more involved than the H\"{o}lder regularity of the eigenvector moment flow in \cite{BouYau2017}, because our observables are more general: in \cite{BouYau2017}, the  scale of observables was a priori known and the dynamics were used to identify the distribution of fluctuations.\\

The second part of the paper (Sections \ref{sec:meanfield} and \ref{sec: comp}) applies the strong form of quantum unique ergodicity to delocalization for random band matrices. First, Theorem \ref{lemma main3} is proved by the mean field reduction technique from \cite{BouErdYauYin2017}, then it is extended to more general band matrices by a moment matching argument.

The proof of Theorem \ref{lemma main3} (Section \ref{sec:meanfield}) is sketched in Subsection \ref{MR}. Subsection \ref{sec: Str} contains the first important input for the proof: the resolvent estimates for $(Q_e-z)^{-1}$. As explained after (\ref{eqn:general}), these estimates from Theorem \ref{LLniu} amount to a form of quantum ergodicity for the eigenvectors of $Q_e$. 
From this a priori estimate,  quantum unique ergodicity is deduced for the Gaussian divisible version of $Q_e$ (Subsection \ref{sec:QUEmean}). 
To access flatness of eigenvectors of our original eigenvectors $\b\psi_k$, we need to patch QUE estimates for eigenvectors of $Q_e$ when $e=\lambda_k$. By a net argument in $e$,
with mesh size $N^{-C}$ ($C$ is fixed and arbitrarily large because Theorem \ref{thmQUE} holds with overwhelming probability), we only need to control  eigenvector shifts under tiny perturbations in $e$. 
This is the role of another input for the proof of Theorem \ref{lemma main3},
the weak uncertainty principle. It is inspired by a more difficult result from \cite{BouErdYauYin2017}, and proved in Subsection \ref{sec: reg}. We refer to (\ref{ver}) for eigenvectors bounds thanks to the weak uncertainty principle.
Subsection \ref{Mfr} concludes the proof of Theorem \ref{lemma main3}.

In Section \ref{sec: comp}, delocalization, local semicircle law and universality (Theorems \ref{Main}, \ref{lsc} and \ref{Univ}) are obtained beyond the Gaussian divisible ensemble. The proof relies on moment matching, exhibiting a matrix $\wt H$ of type (\ref{sjui}) whose first 
four moments of the entries match those of $H$. This idea appeared in \cite{TaoVu2011} for the purpose of universality for Wigner matrices,
and required some a priori information on delocalization and local semicircle law. In our work, such information is only available for $\wt H$, by Theorem \ref{lemma main3}. 
It is extended to $H$ thanks to an implementation of the moment matching strategy at the level of the Green's functions \cite{ErdYauYin2012Univ}, and a self-consistent method  to obtain these estimates by continuously interpolating from $\wt H$ to $H$ \cite{aniso}.\\

Finally, although this work focuses on symmetric matrices, the method applies to the Hermitian class. The only substantial difference is the algebraic part of QUE for mean-field models: the perfect matching observables are defined in a different way for real and complex matrices, as explained in the Appendix.

\section{Quantum unique ergodicity for deformed matrices
}\label{sec:QUE1}

This and the next sections  are self-sufficient. In these sections, the size of the matrices is denoted by $n$. The main result (Theorem \ref{thmQUE}) 
will then be applied to mean-field blocks of type $Q_e$ from (\ref{1100}) (or more precisely its generalization $Q_e^g$, see (\ref{gaibian})), i.e. for $n=W$.

\subsection{Eigenvectors dynamics.}\label{subsec:defBM}\ 
In this subsection, we  first recall  the stochastic differential equation for the eigenvectors 
under the Dyson Brownian motion, as stated in \cite[Section 2]{BouYau2017}. 

The matrix Brownian motion dynamics are defined as follows, either at the matrix, eigenvalues or eigenvectors level (remember we only consider the symmetric case, the Hermitian one being detailed in the Appendix).
Let $B$ be a $n\times n$ matrix such that $B_{ij} (i<j)$ and $B_{ii}/\sqrt{2}$ are independent standard Brownian motions, and $B_{ij}=B_{ji}$. We abbreviate $Z(t)=B(t)/\sqrt{n}$.
The $n\times n$ symmetric Dyson Brownian motion $K$ with initial value $K(0)=V$ is defined as
\begin{equation}\label{eqn:Kt}
K(t)=V+Z(t).
\end{equation}

Let  $\bla_0\in\Sigma_n=\{\lambda_1<\dots<\lambda_n\}$, $\boldu_0\in \OO(n)$. The symmetric Dyson Brownian motion/vector flow  with initial condition  $(\lambda_1,\dots,\lambda_n)=\bla_0$, $(u_1,\dots,u_n)=\boldu_0$,  is defined through the dynamics
\begin{align}
\rd\la_k&=\frac{\rd B_{kk}}{\sqrt{n}}+\left(\frac{1}{n}\sum_{\ell\neq k}\frac{1}{\la_k-\la_\ell}\right)\rd t\label{eqn:eigenvaluesSymmetric},\\
\rd u_k&=\frac{1}{\sqrt{n}}\sum_{\ell\neq k}\frac{\rd B_{k\ell}}{\lambda_k-\lambda_\ell}u_\ell
-\frac{1}{2n}\sum_{\ell\neq k}\frac{\rd t}{(\la_k-\la_\ell)^2}u_k\label{eqn:eigenvectorsSymmetric}.
\end{align}
With a slight abuse of notation, we will write $\bla_t$
either for $(\lambda_1(t),\dots,\lambda_n(t))$ or for the $n\times n$ diagonal matrix with entries 
$\lambda_1(t),\dots,\lambda_n(t)$.

The link between the previously defined matrix and spectral dynamics is given as follows. See \cite{BouYau2017} for a proof, with the main ideas being due to McKean \cite{McK1969} for the existence and uniqueness of solutions, and Bru \cite{Bru1989}
for  the eigenvector dynamics, in the Wishart case.

\begin{theorem}\label{thm:PCE}
The following statements about the  Dyson Brownian motion and eigenvalue/vector flow hold.
\begin{enumerate}[(a)]
\item Existence and strong uniqueness hold for the  system of stochastic differential equations (\ref{eqn:eigenvaluesSymmetric}), (\ref{eqn:eigenvectorsSymmetric}).
Let $(\bla_t,\boldu_t)_{t\geq 0}$  be the solution. Almost surely, for any $t\geq 0$ we have $\bla_t\in\Sigma_n$ and $\boldu_t\in\OO(n)$.
\item Let $(K(t))_{t\geq 0}$ be
a symmetric Dyson Brownian motion with initial condition $K(0)=\boldu_0\bla_0\boldu_0^*$, $\bla_0\in\Sigma_n$.
Then the processes $(K(t))_{t\geq 0}$ and $(\boldu_t\bla_t\boldu_t^*)_{t\geq 0}$ have the same distribution.
\item
Existence and strong uniqueness hold for (\ref{eqn:eigenvaluesSymmetric}). For any $T>0$, let $\nu_T^{K(0)}$ be the distribution of $(\bla_t)_{0\leq t\leq T}$ with initial value the spectrum of a matrix $K(0)$.
For $0\leq T\leq T_0$ and any given continuous trajectory $\bla=(\bla_t)_{0\leq t\leq T_0}\subset\Sigma_n$, existence and strong uniqueness holds for
(\ref{eqn:eigenvectorsSymmetric}) on $[0,T]$. Let $\mu^{K(0), \bla}_T$ be the distribution of
$(\boldu_t)_{0\leq t\leq T}$ with the initial  matrix $K(0)$ and the path $\bla$  given. 

Let $F$ be continuous bounded, from the set of continuous paths (on $[0,T]$) on $n\times n$ symmetric matrices to $\RR$. 
Then for any initial matrix $K(0)$ we have
$$
\E^{K(0)}(F((K(t))_{0\leq t\leq T}))=\int \rd \nu_{T}^{K(0)}(\bla)  \int\rd\mu_T^{K(0), \bla}(\boldu)F((\boldu_t\bla_t\boldu_t^*)_{0\leq t\leq T}).
$$
\end{enumerate}
\end{theorem}

Following \cite{BouYau2017}, we introduce the notations (the dependence in $t$ will often be omitted for $c_{k\ell}$, $1\leq k<\ell\leq n$)
\begin{align}
c_{k\ell}(t)&=\frac{1}{n(\la_k(t)-\la_\ell(t))^2},\label{eqn:cij}\\
u_k\partial_{u_\ell}&=\sum_{\al=1}^nu_k(\al)\partial_{u_\ell(\al)},\label{eqn:ukd} \\
X_{k\ell}^{(s)}&=u_k\partial_{u_\ell}-u_\ell\partial_{u_k},\notag
\end{align}
We then have the following generator for the eigenvector dynamics. For a proof, see \cite{BouYau2017}.

\begin{lemma}\label{lem:generator}
For the diffusion (\ref{eqn:eigenvectorsSymmetric}) 
the generator acting on smooth functions $f:\RR^{n^2}\to\RR$ is
\begin{align}
&\LL_t^{(s)}=\sum_{1\leq k<\ell\leq n}c_{k\ell}(t)(X_{k\ell}^{(s)})^2. \notag
\end{align}
\end{lemma}
\noindent The above lemma means
$\rd\E( g (\boldu_t))/\rd t=\E(\LL_t^{(s)}g (\boldu_t))$
for the stochastic differential equation (\ref{eqn:eigenvectorsSymmetric}).

\subsection{Main result.}\ 
Let $I$ be a deterministic subset of $\llbracket 1,n\rrbracket$. We denote the eigenvectors overlaps as
\begin{align}
p_{ij}&=\sum_{\alpha\in I}u_i(\alpha)u_j(\alpha),\ \ i\neq j\in\llbracket 1,n\rrbracket\notag\\
p_{ii}&=\sum_{\alpha\in I}u_i(\alpha)^2-C_0,\ \ i\in\llbracket 1,n\rrbracket\label{eqn:pkk}
\end{align}
where $C_0$ is an arbitrary but fixed constant independent of $i$. 
We will eventually choose $C_0=|I|/n$ so that the diagonal overlaps are properly normalized, but many results in this section do not depend the actual value of $C_0$.
Moreover, these overlaps are functions of $t$ ($\bu$ satisfies the dynamics (\ref{eqn:eigenvectorsSymmetric})) but this dependence is omitted in the notation. 

Remember the notation (\ref{eqn:Kt}) and denote
$$
G(t,z)=\frac{1}{K(t)-z}.
$$
For a matrix $H$, we abbreviate the Stieltjes transform as
$$
m_H(z)= \frac{1}{n} \tr \frac{1}{H-z}.
$$

\begin{assumption}[Notations and conditions for relaxation flow]\label{XYH}
Fix a  small number $\mathfrak  a > 0$. 
A matrix $V$ is said to be bounded if the norm of $V$ is  bounded, i.e., there is a constant $C_1>0$ such that  
\be\label{clMang}
 \|V\|  := \|V\|_{\rm op}  \leq n^{C_1}.
\ee
A  deterministic matrix $V  $ is called {\bf $(\eta_*,\, \eta^*,  r)$}-regular  at $E_0$   if   $\eta_*$, $\eta^*$ and $r$ satisfy 
 \be\label{jjlzi} 
  n^{-1+\mathfrak a}\le  \eta_*, \quad    \eta_*n^ { \mathfrak a}\le r \le n^ { - \mathfrak a}  \eta^ *, \quad  \eta^ * n^ { \mathfrak a} \le 1
\ee
and there exists $C_2$ such that the imaginary part of the Stieltjes transform of $V$ is  bounded from above and below by
\begin{align}\label{e:imasup}
  C_2^{-1}\leq \Im(m_V(z))\leq C_2, \qquad  m_V(z): = \frac{1}{n} \tr (V-z)^{-1}, 
\end{align}
uniformly for any 
$$
z\in \{E+ \ii \eta: E\in[E_0-r, E_0+r], \;\eta_*\leq \eta\leq \eta^* \}.
$$
\end{assumption}

Our main result not only requires the above hypothesis about the Stieltjes transform, but also the folllowing estimates on individual diagonal resolvent entries.

\begin{assumption}\label{as:basic}
The following holds uniformly in $z\in\{E+\ii\eta: E\in[E_0-r,E_0+r],\eta_*<\eta<\eta^*\}$.
\begin{enumerate}
\item 
Diagonal entries all have the same order:
\begin{equation}\label{eqn:diaginitial}
{\rm Im} G(0,z)_{ii}\leq \frac{2}{n}{\rm Im}{\rm Tr}G(0,z).
\end{equation}
\item There exists a constant $0<\mathfrak{c}<1$ such that the averages over $I$ and $\llbracket 1,n\rrbracket$ coincide up to $n^{-\mathfrak{c}}$:
\be\label{Hop1prime}
\left|\frac{1}{|I|}\sum_{i\in I}G(0,z)_{ii}-\frac{1}{n}{\rm Tr}G(0,z)\right|\leq n^{-\mathfrak{c}}.
\ee
\end{enumerate}
\end{assumption}

In the remainder of this article, to simplify the exposition we also assume that the deterministic set $I$ from (\ref{eqn:pkk}) satisfies
\begin{equation}\label{eqn:Ibound}
|I|\geq c n
\end{equation}
for some small fixed constant $c$. This is enough for our purpose, as $|I|\sim n/2$ in the next sections.
We define, for any $r> 0$ and  $0<\kappa<1$,
\begin{equation}\label{Ir}
  I^r_\kappa (E):=  \cal I_{E, (1-\kappa) r},  \quad  
   \cal I_{E, r}= (E-r, \; E+r). 
\end{equation}
The main result of this section is the following, where we choose $C_0=|I|/N$ in (\ref{eqn:pkk}).

\begin{theorem}[Quantum unique ergodicity for deformed matrices]\label{thmQUE}
Remember the notation (\ref{eqn:pkk}) for the centered partial overlaps, take $C_0=|I|/n$ and assume (\ref{eqn:Ibound}).
Under 
Assumption \ref{XYH}  and Assumption \ref{as:basic}, 
  the following statement holds.
For any (small) $\kappa, \e>0$, (large) $D>0$ and $i,j\in \llbracket 1,n\rrbracket$, for any $t_0,t_1$ such that  $n^{\mathfrak{a}} \eta_*\leq t_0\leq t_1\leq n^{-\mathfrak{a}} r$,  we have
\be\label{Ice1}
\P\left(
\exists\, t_0<t<t_1: \mathds{1}_{\lambda_i(t),\lambda_j(t)\in  I_r^\kappa(E_0)} \left(|p_{ii}| + |p_{ij}|\right)  \geq n^{\e}\left(\frac{1}{n^\mathfrak{c}}+\frac{1}{\sqrt{n t_{{0}}}}\right)
 \right)\le n^{-D}
\ee
for large enough $N$. Here,  the constant $\mathfrak c$ is from \eqref{Hop1prime}. In other words, the errors consist of the initial error $n^{-\mathfrak c}$ and the dynamical error  $(n t_{0})^{-1/2} $.
\end{theorem}

\subsection{Perfect matching observables.}\ 
We will need the following notations.

First, as in \cite{BouYau2017}, we define
$\boeta:  \llbracket 1,n\rrbracket \to \bN$ where
$\eta_j:=\boeta(j)$ is interpreted as the number of particles at the site $j$. Thus $\boeta$ denotes the configuration space 
of particles. We denote $\cN(\boeta) = \sum_j  \eta_j=d$ the total number of particles. 
Define $\boeta^{i, j}$ to be the configuration obtained by moving one particle from $i$ to $j$. 
If there is no particle at $i$ then $\boeta^{i, j} = \boeta$. Notice that there is a direction and the particle is moved  from $i$ to $j$.

Second, for any given configuration $\boeta$, consider the set of vertices
$$
\mathcal{V}_{\boeta}=\{(i,a): 1\leq i\leq n, 1\leq a\leq 2\eta_i\}.
$$
Let $\mathcal{G}_{\boeta}$ be the set of perfect matchings of the complete graph on $\mathcal{V}_{\boeta}$, i.e. this is the set of graphs $G$ with vertices $V_{\boeta}$ and edges $\mathcal{E}(G)\subset\{\{v_1,v_2\}: v_1\in \mathcal{V}_{\boeta},v_2\in\mathcal{V}_{\boeta},v_1\neq v_2\}$ being a partition of $\mathcal{V}_{\boeta}$.

\begin{figure}[h]
\centering
\begin{subfigure}{.4\textwidth}
\centering
\vspace{1.6cm} \hspace{0cm}
\begin{tikzpicture}[scale=0.5]

\draw[fill,black] (1,1) circle [radius=0.2];
\draw[fill,black] (1,2) circle [radius=0.2];
\draw[fill,black] (4,1) circle [radius=0.2];
\draw[fill,black] (4,2) circle [radius=0.2];
\draw[fill,black] (4,3) circle [radius=0.2];
\draw[fill,black] (6,1) circle [radius=0.2];
\draw [thick,->,black] (-2,1) -- (9,1);
\draw [thick,black] (-1,0.8) -- (-1,1.2);
\draw [thick,black] (8,0.8) -- (8,1.2);

\node at (-1,0.4) {\nc 1};
\node at (1,0.4) {\nc$i_1$};
\node at (4,0.4) {\nc $i_2$};
\node at (6,0.4) {\nc $i_3$};
\node at (8,0.4) {\nc $n$};
  
\end{tikzpicture}
\vspace{0.1cm}
  \caption{A configuration $\boeta$ with $\mathcal{N}(\boeta)=6$, $\eta_{i_1}=2$, $\eta_{i_2}=3$, $\eta_{i_3}=1$.
 }
   \label{fig:sub1}
\end{subfigure}
\hspace{1cm}
\begin{subfigure}{.4\textwidth}
  \centering
\begin{tikzpicture}[scale=0.5]
  
\draw[fill,black] (1,1) circle [radius=0.2];
\draw[fill,black] (1,2) circle [radius=0.2];
\draw[fill,black] (1,3) circle [radius=0.2];
\draw[fill,black] (1,4) circle [radius=0.2];
\draw[fill,black] (4,1) circle [radius=0.2];
\draw[fill,black] (4,2) circle [radius=0.2];
\draw[fill,black] (4,3) circle [radius=0.2];
\draw[fill,black] (4,4) circle [radius=0.2];
\draw[fill,black] (4,5) circle [radius=0.2];
\draw[fill,black] (4,6) circle [radius=0.2];
\draw[fill,black] (6,1) circle [radius=0.2];
\draw[fill,black] (6,2) circle [radius=0.2];
\draw [thick,->,black] (-2,1) -- (9,1);
\draw [thick,black] (-1,0.8) -- (-1,1.2);
\draw [thick,black] (8,0.8) -- (8,1.2);

\node at (-1,0.4) {\nc 1};
\node at (1,0.4) {\nc $i_1$};
\node at (4,0.15) {\nc $i_2$};
\node at (6,0.4) {\nc $i_3$};
\node at (8,0.4) {\nc $n$};

\draw [ultra thick,black] (1,1) to[out=180,in=180] (1,3);
\draw [ultra thick,black] (1,2) to[out=-70,in=-90] (6,2);
\draw [ultra thick,black] (1,4) to[out=-70,in=120] (4,1);
\draw [ultra thick,black] (4,6) to[out=-10,in=0] (6,1);
\draw [ultra thick,black] (4,3) to[out=-10,in=0] (4,4);
\draw [ultra thick,black] (4,2) to[out=-10,in=0] (4,5);

\end{tikzpicture}
  \caption{A perfect matching $G\in\mathcal{G}_{\boeta}$. Here, $P(G)=p_{i_1i_1}p_{i_1i_2}p_{i_2i_2}^2p_{i_2i_3}p_{i_3i_1}$.}
  \label{fig:sub1}
\end{subfigure}
\end{figure}

Third, for any given edge $e=\{(i_1,a_1),(i_2,a_2)\}$, we define  
$p(e)=p_{i_1,i_2}$, 
$
P(G)=\prod_{e\in \mathcal{E}(G)}p(e)
$ and
\begin{equation}\label{feq}
f^{(s)}_{\bla, t}(\boeta)=\frac{1}{\mathcal{M}(\boeta)}\E\left(\sum_{G\in\mathcal{G}_{\boeta}} P(G)\mid \bla\right),\ \ \mathcal{M}(\boeta)=\prod_{i=1}^n (2\eta_i)!!,
\end{equation}
where $(2m)!!=\prod_{k\leq 2m,k \, {\rm odd}}k$ is the number of perfect matchings of the complete graph on $2m$ vertices.
Remarkably, the above function $f$ satisfies a parabolic partial differential equation. 

\begin{theorem}[Perfect matching observables for the eigenvector moment flow: symmetric case]\label{thm:EMF}
Suppose that $\boldu$ is the solution to the symmetric eigenvector dynamics (\ref{eqn:eigenvectorsSymmetric})
and $ f^{(s)}_{\bla, t} (\boeta)$ is given by \eqref{feq}.  Then 
$f^{(s)}_{\bla, t}$
satisfies the equation  
\begin{align}
\label{ve}
&\partial_t f^{(s)}_{\bla, t} =  \mathscr{B}^{(s)}(t)  f^{(s)}_{\bla, t},\\
\label{momentFlotSym}&\mathscr{B}^{(s)}(t)  f(\boeta) = \sum_{k \neq \ell} c_{k\ell}(t) 2 \eta_k (1+ 2 \eta_\ell) \left(f(\boeta^{k\ell})-f(\boeta)\right).
\end{align}
\end{theorem}

\begin{remark}\label{rem:reversibility}
An important property of the eigenvector moment flow is the reversibility with respect to  a simple explicit equilibrium measure:
\beq\label{eqn:weight}
\pi(\boeta) = \prod_{p=1}^n \phi(\eta_p), \ \phi(k) =\prod_{i=1}^k\left(1-\frac{1}{2i}\right).
\bEq
For any function $f$ on the configuration space, the Dirichlet form is given by 
$$
\sum_{\bm\eta}\pi(\bm\eta)f(\bm\eta)\mathscr{B}(t)f(\bm \eta)=\sum_{\boeta}  \pi(\boeta) \sum_{i \neq j}  c_{ij}  \eta_i (1+ 2 \eta_j) \left(f(\boeta^{i j}) - f(\boeta)\right)^2.
$$
\end{remark}

\begin{remark}\label{rem:canon}
The above theorem is independent of our choice of $C_0$ and of the canonical basis and, more remarkably, the projection vectors don't have to be orthogonal. More precisely, let $(\bq_\al)_{\al\in I}$ be any family of fixed vectors. Define
\begin{align*}
p_{ij}&=\sum_{\alpha\in I}\langle u_i,\bq_\alpha\rangle\langle u_j,\bq_\alpha\rangle\ \ i\neq j\in\llbracket 1,n\rrbracket,\\
p_{ii}&=\sum_{\alpha\in I}\langle u_i,\bq_\alpha\rangle^2-C_0,\ \ i\in\llbracket 1,n\rrbracket,
\end{align*}
and $f_{t,\bla}$ accordingly. Then (\ref{ve}) holds. In particular, Theorem \ref{thm:EMF} generalizes \cite[Theorem 3.1 (i)]{BouYau2017} by just choosing $|I|=1$.
\end{remark}

\subsection{Proof of Theorem \ref{thm:EMF}.}\  \label{proofofEMF}  
To start the proof of Theorem \ref{thm:EMF}, let 
\begin{equation}\label{eqn:gR}
g(\boeta)=\frac{1}{\mathcal{M}(\boeta)}\sum_{G\in\mathcal{G}_{\boeta}} P(G)
\end{equation}
and let $1\leq k<\ell\leq n$ be fixed for the rest of this subsection. We abbreviate $X=X_{k\ell}^{(s)}$. Using Lemma \ref{lem:generator},  we only need to prove 
\begin{equation}\label{eqn:application}
X^2 g(\boeta)=2\eta_k(1+2\eta_\ell) (g(\boeta^{k\ell})-g(\boeta))+2\eta_\ell(1+2\eta_k) (g(\boeta^{\ell k})-g(\boeta)).
\end{equation}
We therefore want to calculate  $X^2 P(G)$ for any $G\in\mathcal{G}_{\boeta}$. For that purpose, we first need the following definition.

\begin{definition} Let $\boeta$ and $k<\ell$ be fixed. The following notations will be useful for calculating $X^2 P(G)$.
\begin{enumerate}[(i)]
\item $\mathcal{V}_i\subset\mathcal{V}_{\boeta}$ is the set of vertices of type $(i,a)$, $1\leq a\leq 2\eta_i$. 
\item For any two vertices $v,w\in \mathcal{V}_k\cup\mathcal{V}_\ell$, we denote 
$$
\epsilon(v,w)=\left\{
\begin{array}{ll}
1&{\rm if}\ v,w\ {\rm are\ in\ the\ same}\ \mathcal{V}_i,\ i=k\ or\ \ell\\
-1&{\rm if}\ v,w \ {\rm are\ in\ different}\ \mathcal{V}_i\,'s.
\end{array}
\right.
$$
\item Let $G\in\mathcal{G}_{\boeta}$ and $v,w\in \mathcal{V}_k\cup\mathcal{V}_\ell$.

Assume $v\in \mathcal{V}_k$ and $w\in \mathcal{V}_\ell$. Then we define $S_{wv}G=S_{vw}G\in\mathcal{G}_{\boeta}$ as the perfect matching obtained by transposition of $v$ and $w$. More precisely, let $\tau_{vw}$ be the permutation of $\mathcal{V}_{\boeta}$ transposing $v$ and $w$. Then 
$$
\mathcal{E}(S_{vw}G)=\{\{\tau_{v,w}(v_1),\tau_{v,w}(v_2)\}: \{v_1,v_2\}\in\mathcal{E}(G)\}.
$$

Assume $v=(k,a)$ and $w=(k,b)$ ($a<b$) are both in $\mathcal{V}_k$. Then we define $S_{wv}G=S_{vw}G\in\mathcal{G}_{\boeta^{k\ell}}$ as the perfect matching obtained by a jump of $v$ and $w$ to $\ell$. 
More precisely, let $j_{vw}=j_{wv}$ be the following bijection from $\mathcal{V}_{\boeta}$ to $\mathcal{V}_{\boeta^{k\ell}}$: 
$j_{vw}(v)=(\ell,2\eta_\ell+1)$, $j_{vw}(w)=(\ell,2\eta_\ell+2)$, $j_{vw}((k,c))=(k,c-2)$ if $b<c$, $j_{vw}((k,c))=(k,c-1)$ if $a<c<b$ and $j_{vw}(v_1)=v_1$ in all other cases. 
Then
$$
\mathcal{E}(S_{vw}G)=\{\{j_{v,w}(v_1),j_{v,w}(v_2)\}: \{v_1,v_2\}\in\mathcal{E}(G)\}.
$$
A similar definition applies if both  $v$ and $w$ are in $\mathcal{V}_\ell$, the jump now being towards $k$.
\end{enumerate}
\end{definition}

\begin{figure}[h]
\centering
\begin{subfigure}{.45\textwidth}
\centering
\begin{tikzpicture}[scale=0.4]
  \normalcolor
\draw[fill] (1,1) circle [radius=0.2];
\draw[fill] (1,2) circle [radius=0.2];
\draw[fill] (1,3) circle [radius=0.2];
\draw[fill] (1,4) circle [radius=0.2];
\draw[fill] (4,1) circle [radius=0.2];
\draw[fill] (4,2) circle [radius=0.2];
\draw[fill] (4,3) circle [radius=0.2];
\draw(4,3) circle [radius=0.35];
\node at (3.2,3) {\normalcolor$v$};
\draw[fill] (4,4) circle [radius=0.2];
\draw[fill] (4,5) circle [radius=0.2];
\draw[fill] (4,6) circle [radius=0.2];
\draw[fill] (6,1) circle [radius=0.2];
\draw[fill] (6,2) circle [radius=0.2];
\draw(6,2) circle [radius=0.35];
\node at (6,2.8) {\normalcolor$w$};
\draw [thick,->] (-0.5,1) -- (7.5,1);

\node at (1,0.3) {\normalcolor$i$};
\node at (4,0.05) {\normalcolor$k$};
\node at (6,0.3) {\normalcolor$\ell$};

\draw [ultra thick] (1,1) to[out=180,in=180] (1,3);
\draw [ultra thick] (1,2) to[out=-70,in=-90] (6,2);
\draw [ultra thick] (1,4) to[out=-70,in=120] (4,1);
\draw [ultra thick] (4,6) to[out=-10,in=0] (6,1);
\draw [ultra thick] (4,3) to[out=-10,in=0] (4,4);
\draw [ultra thick] (4,2) to[out=-10,in=0] (4,5);

\draw[fill] (11,1) circle [radius=0.2];
\draw[fill] (11,2) circle [radius=0.2];
\draw[fill] (11,3) circle [radius=0.2];
\draw[fill] (11,4) circle [radius=0.2];
\draw[fill] (14,1) circle [radius=0.2];
\draw[fill] (14,2) circle [radius=0.2];
\draw[fill] (14,3) circle [radius=0.2];
\draw[fill] (14,4) circle [radius=0.2];
\draw[fill] (14,5) circle [radius=0.2];
\draw[fill] (14,6) circle [radius=0.2];
\draw[fill] (16,1) circle [radius=0.2];
\draw[fill] (16,2) circle [radius=0.2];

\draw [thick,->] (9.5,1) -- (17.5,1);

\draw [ultra thick] (11,1) to[out=180,in=180] (11,3);
\draw [ultra thick] (11,2) to[out=20,in=-130] (14,3);
\draw [ultra thick] (11,4) to[out=-70,in=120] (14,1);
\draw [ultra thick] (14,6) to[out=-10,in=0] (16,1);
\draw [ultra thick] (16,2) to[out=130,in=-20] (14,4);
\draw [ultra thick] (14,2) to[out=-10,in=0] (14,5);

\draw [ultra thick,->] (7,4) to[out=30,in=150] (10,4);
\node at (8.5,5.3) {\normalcolor$S_{vw}$};

\end{tikzpicture}
  \caption{The map $S_{vw}$ in case of a transposition.
 }
   \label{fig:sub3}
\end{subfigure}
\hspace{1cm}
\begin{subfigure}{.45\textwidth}
\centering
\begin{tikzpicture}[scale=0.4]
  \normalcolor
\draw[fill] (1,1) circle [radius=0.2];
\draw[fill] (1,2) circle [radius=0.2];
\draw[fill] (1,3) circle [radius=0.2];
\draw[fill] (1,4) circle [radius=0.2];
\draw[fill] (4,1) circle [radius=0.2];
\draw[fill] (4,2) circle [radius=0.2];
\draw[fill] (4,3) circle [radius=0.2];
\draw(4,2) circle [radius=0.35];
\node at (3.3,2.4) {\normalcolor$v$};
\draw(4,4) circle [radius=0.35];
\node at (3.2,4) {\normalcolor$w$};
\draw[fill] (4,4) circle [radius=0.2];
\draw[fill] (4,5) circle [radius=0.2];
\draw[fill] (4,6) circle [radius=0.2];
\draw[fill] (6,1) circle [radius=0.2];
\draw[fill] (6,2) circle [radius=0.2];
\draw [thick,->] (-0.5,1) -- (7.5,1);

\node at (1,0.3) {\normalcolor$i$};
\node at (4,0.05) {\normalcolor$k$};
\node at (6,0.3) {\normalcolor$\ell$};

\draw [ultra thick] (1,1) to[out=180,in=180] (1,3);
\draw [ultra thick] (1,2) to[out=-70,in=-90] (6,2);
\draw [ultra thick] (1,4) to[out=-70,in=120] (4,1);
\draw [ultra thick] (4,6) to[out=-10,in=0] (6,1);
\draw [ultra thick] (4,3) to[out=-10,in=0] (4,4);
\draw [ultra thick] (4,2) to[out=-10,in=0] (4,5);

\draw[fill] (11,1) circle [radius=0.2];
\draw[fill] (11,2) circle [radius=0.2];
\draw[fill] (11,3) circle [radius=0.2];
\draw[fill] (11,4) circle [radius=0.2];
\draw[fill] (14,1) circle [radius=0.2];
\draw[fill] (14,2) circle [radius=0.2];
\draw[fill] (14,3) circle [radius=0.2];
\draw[fill] (14,4) circle [radius=0.2];
\draw[fill] (16,1) circle [radius=0.2];
\draw[fill] (16,2) circle [radius=0.2];
\draw[fill] (16,3) circle [radius=0.2];
\draw[fill] (16,4) circle [radius=0.2];
\draw [thick,->] (9.5,1) -- (17.5,1);

\draw [ultra thick] (11,1) to[out=180,in=180] (11,3);
\draw [ultra thick] (11,2) to[out=-70,in=-90] (16,2);
\draw [ultra thick] (11,4) to[out=-70,in=120] (14,1);
\draw [ultra thick] (14,4) to[out=-45,in=135] (16,1);
\draw [ultra thick] (14,2) to[out=-10,in=220] (16,4);
\draw [ultra thick] (16,3) to[out=180,in=0] (14,3);

\draw [ultra thick,->] (7,4) to[out=30,in=150] (10,4);
\node at (8.5,5.3) {\normalcolor$S_{vw}$};

\end{tikzpicture}
\vspace{0.1cm}
  \caption{The map $S_{vw}$ in case of a jump.
 }
   \label{fig:sub4}
\end{subfigure}
\end{figure}

In this proof, for any set $A$ we denote $A^2_*=\{(a,b)\in A^2:a\neq b\}$.
The following result is the key step in our proof of Theorem \ref{thm:EMF}.

\begin{lemma}\label{lem:key}
For any $G\in\mathcal{G}_{\boeta}$, we have
\begin{equation}\label{sum11}
X^2P(G)=\sum_{(v,w)\in(\mathcal{V}_k\cup\mathcal{V}_\ell)^2_*}\epsilon(v,w)P(S_{vw}G)-(2\eta_k+2\eta_\ell)P(G).
\end{equation}
\end{lemma}

We postpone the proof of the above lemma and first finish the proof of Theorem \ref{thm:EMF}. Let 
$$
h(\boeta)=\sum_{G\in\mathcal{G}_{\boeta}}P(G).
$$
Note that if $v\in \mathcal{V}_k$ and $w\in \mathcal{V}_\ell$, $S_{vw}$ is a permutation of $\mathcal{G}_{\boeta}$. Moreover, if $v$ and $w$ are both in $\mathcal{V}_k$, 
 $S_{vw}$ is a bijection from $\mathcal{G}_{\boeta}$ to  $\mathcal{G}_{\boeta^{k\ell}}$. The summation of (\ref{sum11}) over all $G\in\mathcal{G}_{\boeta}$ therefore gives
\begin{align*}
X^2h(\boeta)=&\sum_{(v,w)\in(\mathcal{V}_k)^2_*}\sum_{G\in\mathcal{G}_{\boeta}}P(S_{vw}G)+\sum_{(v,w)\in(\mathcal{V}_\ell)^2_{*}}\sum_{G\in\mathcal{G}_{\boeta}}P(S_{vw}G)
-2\sum_{(v,w)\in\mathcal{V}_k\times\mathcal{V}_\ell}\sum_{G\in\mathcal{G}_{\boeta}}P(S_{vw}G)
-2(\eta_k+\eta_\ell)h(\boeta)\\
=&\sum_{(v,w)\in(\mathcal{V}_k)^2_*}h(\boeta^{k\ell})+\sum_{(v,w)\in(\mathcal{V}_\ell)^2_*}h(\boeta^{\ell k})
-2\sum_{(v,w)\in\mathcal{V}_k\times\mathcal{V}_\ell}h(\boeta)
-2(\eta_k+\eta_\ell)h(\boeta)\\
X^2h(\boeta)=&\ 2\eta_k(2\eta_k-1)h(\boeta^{k\ell})+2\eta_\ell(2\eta_\ell-1)h(\boeta^{\ell k})-(2\eta_k(2\eta_\ell+1)+2\eta_\ell(2\eta_k+1))h(\boeta).
\end{align*}
The above equation implies (\ref{eqn:application}) after renormalization by $\mathcal{M}(\boeta)$. This concludes the proof of Theorem \ref{thm:EMF}.

\begin{proof}[Proof of Lemma \ref{lem:key}]
Let  $G\in\mathcal{G}_{\boeta}$ and $1\leq k<\ell\leq n$ be fixed. The Leibniz rule applies: for any smooth functions $f,g((u_i(\al))_{1\leq i,\al\leq n}):\RR^{n^2}\to\RR$ we have
$X(fg)=fX(g)+g X(f)$, so that
\begin{equation}\label{eqn:Leibniz}
X^2 P(G)=\sum_{(e_1,e_2)\in \mathcal{E}(G)^2_*}Xp(e_1)Xp(e_2)\prod_{e\in\mathcal{E}(G)\backslash\{e_1,e_2\}}p(e)
+
\sum_{e_1\in \mathcal{E}(G)}X^2p(e_1)\prod_{e\in\mathcal{E}(G)\backslash\{e_1\}}p(e).
\end{equation}
The above sums will be decomposed depending of the following edge group (single, double or transverse):
\begin{align}
\mathcal{E}_s&=\mathcal{E}(G)\cap\{\{v,w\}:v\in \mathcal{V}_k\cup\mathcal{V}_\ell,w\not\in\mathcal{V}_k\cup\mathcal{V}_\ell\},\label{edg1}\\
\mathcal{E}_d&=\mathcal{E}(G)\cap\{\{v,w\}:(v,w)\in \mathcal{V}_k^2\cup\mathcal{V}_\ell^2\},\label{edg2}\\
\mathcal{E}_t&=\mathcal{E}(G)\cap\{\{v,w\}:v\in \mathcal{V}_k,w\in\mathcal{V}_\ell\}\label{edg3}.
\end{align}
For any $v\in\mathcal{V}_{\boeta}$, let $e_v$ be the edge containing $v$ and $v'$ be the vertex such that $e_v=\{v,v'\}$. We denote
\begin{align*}
\mathcal{V}_s&=\{v\in \mathcal{V}_k\cup\mathcal{V}_\ell:\{v,v'\}\in\mathcal{E}_s\},\\
\mathcal{V}_d&=\{v\in \mathcal{V}_k\cup\mathcal{V}_\ell:\{v,v'\}\in\mathcal{E}_d\},\\
\mathcal{V}_t&=\{v\in \mathcal{V}_k\cup\mathcal{V}_\ell:\{v,v'\}\in\mathcal{E}_t\}.
\end{align*}
Our calculations will be based on the following basic facts: if $e\not\in\mathcal{E}_s\cup\mathcal{E}_d\cup\mathcal{E}_t$ then $X_{k\ell}p(e)=0$, and
\begin{align}
&X p_{ki}=-p_{\ell i}\label{trans3},\\
&X p_{k\ell}=p_{kk}-p_{\ell\ell},\label{trans2}\\
&X p_{\ell\ell}=2 p_{k\ell}\label{trans1}.
\end{align}
From (\ref{eqn:Leibniz}) we have
$
X^2 P(G)={\rm(I)+(II)+(III)+(IV)+(V)+(VI)+(VII)+(VIII)+(IX)}
$
where all terms are defined and calculated below. First,
$$
{\rm(I)}:=\sum_{(e_1,e_2)\in (\mathcal{E}_s)^2_*}Xp(e_1)Xp(e_2)\prod_{e\in\mathcal{E}(G)\backslash\{e_1,e_2\}}p(e)
=
\sum_{(v,w)\in (\mathcal{V}_s)^2_*}Xp_{\{w,w'\}}Xp_{\{v,v'\}}\prod_{e\in\mathcal{E}(G)\backslash\{e_v,e_w\}}p(e).
$$
From (\ref{trans3}),  $Xp_{\{v,v'\}}Xp_{\{w,w'\}}=-p_{\{w,v'\}}p_{\{v,w'\}}$ if $v$ and $w$ are in distinct $\mathcal{V}_i$'s, and 
$Xp_{\{v,v'\}}Xp_{\{w,w'\}}=p_{\{j_{v,w}(v),v'\}}p_{\{j_{v,w}(w),w'\}}$ if they are both in the same $\mathcal{V}_i$. In all cases, we proved
\begin{equation}\label{eqn:I}
{\rm(I)}=\sum_{(v,w)\in (\mathcal{V}_s)^2_*} \e(v,w)P(S_{vw}G).
\end{equation}
We now consider
$$
{\rm(II)}:=\sum_{(e_1,e_2)\in \mathcal{E}_s\times\mathcal{E}_d\cup \mathcal{E}_d\times\mathcal{E}_s}Xp(e_1)Xp(e_2)\prod_{e\in\mathcal{E}(G)\backslash\{e_1,e_2\}}p(e)
=
\sum_{(v,w)\in \mathcal{V}_s\times \mathcal{V}_d}Xp_{\{v,v'\}}Xp_{\{w,w'\}}\prod_{e\in\mathcal{E}(G)\backslash\{e_v,e_w\}}p(e).
$$
For the second equality, note that vertices on a double edge need to be weighted by a factor $1/2$.
From (\ref{trans1}) and (\ref{trans3}),  $Xp_{\{v,v'\}}Xp_{\{w,w'\}}=-2p_{\{w,v'\}}Xp_{\{v,w'\}}$  if $v$ and $w$ are in distinct $\mathcal{V}_{i}$'s, and 
$2p_{\{j_{vw}(v),v'\}}p_{\{j_{vw}(w),w'\}}$ if they are in the same $\mathcal{V}_{i}$.
We therefore have
\begin{equation}\label{eqn:II}
{\rm(II)}=\sum_{(v,w)\in \mathcal{V}_s\times\mathcal{V}_d\cup\mathcal{V}_d\times\mathcal{V}_s} \e(v,w)P(S_{vw}G).
\end{equation}
For the contribution of 
$$
{\rm(III)}:=\sum_{(e_1,e_2)\in (\mathcal{E}_d)^2_*}Xp(e_1)Xp(e_2)\prod_{e\in\mathcal{E}(G)\backslash\{e_1,e_2\}}p(e)
=
\frac{1}{4}\sum_{(v,w)\in (\mathcal{V}_d)^2_*:w\neq v'}Xp_{\{v,v'\}}Xp_{\{w,w'\}}\prod_{e\in\mathcal{E}(G)\backslash\{e_v,e_w\}}p(e),
$$
from (\ref{trans1}) we have $Xp_{\{v,v'\}}Xp_{\{w,w'\}}=-4p_{\{w,v'\}}Xp_{\{v,w'\}}$ if $v$ and $w$ are in distinct $\mathcal{V}_{i}$'s, 
$2p_{\{j_{vw}(v),v'\}}p_{\{j_{vw}(w),w'\}}$ if they are in the same $\mathcal{V}_{i}$. We therefore proved
\begin{equation}\label{eqn:III}
{\rm(III)}=\sum_{(v,w)\in (\mathcal{V}_d)^2_*} \e(v,w)P(S_{vw}G)-\sum_{v\in\mathcal{V}_d} P(S_{vv'}G).
\end{equation}
We now calculate
\begin{equation}\label{eqn:IV}
{\rm(IV)}:=
\sum_{e_1\in \mathcal{E}_s}X^2p(e_1)\prod_{e\in\mathcal{E}(G)\backslash\{e_1\}}p(e)
=
\sum_{v\in \mathcal{V}_s}X^2p_{\{v,v'\}}\prod_{e\in\mathcal{E}(G)\backslash\{e_v\}}p(e)
=
-\sum_{v\in \mathcal{V}_s}P(G)
\end{equation}
where we used  (\ref{trans3}) twice to obtain $X^2p_{\{v,v'\}}=-p_{\{v,v'\}}$.
For the term
$$
{\rm(V)}:=
\sum_{e_1\in \mathcal{E}_d}X^2p(e_1)\prod_{e\in\mathcal{E}(G)\backslash\{e_1\}}p(e)
=
\frac{1}{2}\sum_{v\in \mathcal{V}_d}X^2p_{\{v,v'\}}\prod_{e\in\mathcal{E}(G)\backslash\{e_1\}}p(e),
$$
note that we have
$X^2p_{\{v,v'\}}=2p_{kk}-2p_{\ell\ell}$ if $v\in\mathcal{V}_\ell$, $2p_{\ell\ell}-2p_{kk}$ otherwise. This yields
\begin{equation}\label{eqn:V}
{\rm(V)}=
\sum_{v\in \mathcal{V}_d}(P(S_{v,v'}(G))-P(G)).
\end{equation}
We now consider cases where transverse edges appear:
$$
{\rm(VI)}:=\sum_{(e_1,e_2)\in \mathcal{E}_s\times\mathcal{E}_t\cup \mathcal{E}_t\times\mathcal{E}_s}Xp(e_1)Xp(e_2)\prod_{e\in\mathcal{E}(G)\backslash\{e_1,e_2\}}p(e)
=
2\sum_{v\in \mathcal{V}_s,\{w,w'\}\in\mathcal{E}_t}Xp_{\{v,v'\}}Xp_{\{w,w'\}}\prod_{e\in\mathcal{E}(G)\backslash\{e_v,e_w\}}p(e).
$$
Up to transposing $w$ and $w'$, we can assume that $v$ and $w$ are in the same $\mathcal{V}_i$. With  (\ref{trans3}) and (\ref{trans2}), a calculation gives 
$Xp_{\{v,v'\}}Xp_{\{w,w'\}}=p_{j_{vw}(v)v'}p_{j_{vw}(w)w'}-p_{\tau_{vw'}(v)v'}p_{\tau_{vw'}(w')w}$. This yields
\begin{equation}\label{eqn:VI}
{\rm(VI)}=
\sum_{(v,w)\in\mathcal{V}_s\times\mathcal{V}_t\cup\mathcal{V}_t\times\mathcal{V}_s}\e(v,w)P(S_{vw}(G)).
\end{equation}
We also have
$$
{\rm(VII)}:=\sum_{(e_1,e_2)\in \mathcal{E}_d\times\mathcal{E}_t\cup \mathcal{E}_t\times\mathcal{E}_d}Xp(e_1)Xp(e_2)\prod_{e\in\mathcal{E}(G)\backslash\{e_1,e_2\}}p(e)
=
\sum_{v\in \mathcal{V}_d,\{w,w'\}\in\mathcal{E}_t}Xp_{\{v,v'\}}Xp_{\{w,w'\}}\prod_{e\in\mathcal{E}(G)\backslash\{e_v,e_w\}}p(e).
$$
We can assume $v$ and $w$ are in the same $\mathcal{V}_i$. Then (\ref{trans2})  and (\ref{trans1})  give 
$Xp_{\{v,v'\}}Xp_{\{w,w'\}}=2(p_{j_{vw}(v)v'}p_{j_{vw}(w)w'}-p_{\tau_{vw'}(v)v'}p_{\tau_{vw'}(w')w})$, so that
\begin{equation}\label{eqn:VII}
{\rm(VII)}=
\sum_{(v,w)\in\mathcal{V}_d\times\mathcal{V}_t\cup\mathcal{V}_t\times\mathcal{V}_d}\e(v,w)P(S_{vw}(G)).
\end{equation}
For two transverse edges, we have
$$
{\rm(VIII)}:=\sum_{(e_1,e_2)\in (\mathcal{E}_t)^2_*}Xp(e_1)Xp(e_2)\prod_{e\in\mathcal{E}(G)\backslash\{e_1,e_2\}}p(e)
=
\frac{1}{4}\sum_{(v,w)\in (\mathcal{V}_t)^2_*,w\neq v'}Xp_{\{v,v'\}}Xp_{\{w,w'\}}\prod_{e\in\mathcal{E}(G)\backslash\{e_v,e_w\}}p(e).
$$
Without loss of generality, assume $v$ and $w$ are in the same $\mathcal{V}_i$. 
Equation (\ref{trans2}) yields $Xp_{\{v,v'\}}Xp_{\{w,w'\}}=
p_{\{j_{v,w}(v),v'\}}p_{\{j_{v,w}(w),w'\}}+p_{\{j_{v',w'}(v'),v\}}p_{\{j_{v',w'}(w'),w\}}-
p_{\{\tau_{v,w'}(v),v'\}}p_{\{\tau_{v,w'}(w'),w\}}-
p_{\{\tau_{v',w}(v),v\}}p_{\{\tau_{v',w}(w),w'\}}$. We therefore have
\begin{equation}\label{eqn:VIII}
{\rm(VIII)}=\sum_{(v,w)\in(\mathcal{V}_t)^2_*}\e(v,w)P(S_{vw}(G))+\sum_{v\in\mathcal{V}_t}P(G).
\end{equation}
Finally, from (\ref{trans2}) we have $X^2p_{k\ell}=-4p_{k\ell}$, so that
\begin{equation}\label{eqn:IX}
{\rm(IX)}:=
\sum_{e_1\in \mathcal{E}_t}X^2p(e_1)\prod_{e\in\mathcal{E}(G)\backslash\{e_1\}}p(e)
=
-2\sum_{v\in \mathcal{V}_t}P(G)
\end{equation}
By summation of all equations (\ref{eqn:I}), (\ref{eqn:II}), (\ref{eqn:III}), (\ref{eqn:IV}), (\ref{eqn:V}), (\ref{eqn:VI}), (\ref{eqn:VII}), (\ref{eqn:VIII}), (\ref{eqn:IX}), the right hand sides of (\ref{sum11}) and (\ref{eqn:Leibniz})
exactly coincide, concluding the proof of Lemma \ref{lem:key}.
\end{proof}

\section{Analysis of the eigenvector moment flow}
\label{sec:QUE}

Before getting into the details of the proof of Theorem \ref{thmQUE}, i.e. relaxation for the eigenvector moment flow (\ref{momentFlotSym}), we note substantial differences with the setting and proof from \cite{BouYau2017}. The dynamics equation (\ref{momentFlotSym}) already appeared in \cite{BouYau2017}, but
the observables associated with the equation (\ref{momentFlotSym})  
are now much more general (see Remark \ref{rem:canon}), and their natural scale (i.e., the order of the sizes of these observables) is not known a priori. 

Indeed,  in \cite{BouYau2017}, the order of magnitude of $f_t(\boeta)$ was a priori known: 
$f_t(\boeta)=\E(|\sqrt{n}\langle \bq,u_k\rangle|^d\mid\bla)\leq n^\e$ thanks to the local law.
The eigenvector moment flow  was used in \cite{BouYau2017}  to find fluctuations around this scale.

On the contrary, in the current paper, the eigenvector moment flow  (\ref{momentFlotSym}) allows to find the natural scale for a wider class of observables. For $|I|\sim c n$,  local laws only give the trivial estimate $|p_{ii}|\leq 1$ for example, although the dynamics yield Theorem  \ref{thmQUE}, i.e. $|p_{ii}|\leq n^{-1/2+\e}$ for $t$ approaching 1.

This differences  about observables and scales require the following notable novelties in the proof of Theorem \ref{thmQUE}:
\begin{enumerate}
\item The decomposition between long-range and short-range dynamics is now more intricate. In particular,  our bound on the long-range contribution improves in inductive steps (see Lemma \ref{l:Lbound} to be compared with \cite[Lemma 6.1]{BouYau2017}).
\item The maximum principle, Proposition \ref{thm:maxPrincipleQUE}, also gives stronger  results once it is used inductively,
on space-time embedded domains, while the analogue \cite[Theorem 7.4]{BouYau2017} only required one time step. 
\end{enumerate}

In summary, the error terms in  the finite speed of propagation and the maximum principle estimates depend on  the size of $f_t(\boeta)$. In this paper, the a priori bound on  
$f_t(\boeta)$ is far from its real size. Hence we need to bootstrap our estimates in a suitable way  in order to get a sharp estimate at the end of the proof.   \nc

\medskip 
We now introduce a few notations which will be useful in the statement and proof of the following Lemma \ref{lemdeloc}, and the following of this section.
For a fixed fixed and arbitrarily small $\omega>0$, we define  the control parameter
$$
\psi=n^\omega
$$
with $\omega\leq \mathfrak a/100$,
and  the following time  and spectral domains: 
\begin{equation}\label{Tdef}
\cal T_\omega(\eta_*,\eta^*, r)=\left\{t: \eta_*\psi\le t\le \psi^{-1}\, r\right\},
\end{equation}

\subsection{A priori estimates.}\ For $K(t)$ in \eqref{eqn:Kt}, we denote the initial matrix $V=U_0 \Lambda_0 U_0^*$, where $\Lambda_0=\diag\{\lambda_1(0), \cdots, \lambda_n(0)\}$, and $U_0$ is the orthogonal matrix of its eigenvectors. 
Let $m_{\fc,t}$ be the Stieltjes transform of the free convolution between the empirical spectral measure of $V$ and the Gaussian orthogonal ensemble $Z_t$.
Then  $m_{\fc,t}$ solves the equation
 \begin{equation}\label{eqn:freefunc}
m^{(n)}_{\fc,t}(z)= m_{ V}\left(z+ t\; m^{(n)}_{\fc,t}(z)\right)= \frac{1}{n}\sum_{i=1}^n g_i(t,z),\quad g_i(t,z)\deq \frac{1}{ \lambda_i (0) - z - t m_{\fc,t}(z) }.
\end{equation}  
Here  $m^{(n)}_{\fc,t}(z)$ is the Stieltjes transform of a measure with density denoted $ \rho^{(n)}_{\fc,t} $. For notational convenience we will suppress the superscript and use the notations
$m_{\fc,t}(z)$, $\rho_{\fc,t}$. 
\nc 

The typical location $\gamma_i(s)$ of the $i$th eigenvalue $\lambda_i(s)$  is defined through 
$\int _{-\infty}^{\gamma_i(s)}\rd \rho_{\fc,s}=\frac{i}{n}$.
We also recall the following stability property of the typical locations, see \cite[Lemma 3.4]{LanYau2015}: for any $0<q_1<q_2<1$ and $\omega>0$, for large enough $n$ we have,
for all $s,t\in \cal T_\omega(\eta_*,\eta^*, r)$,
\begin{equation}\label{eqn:stability}
\{i:\gamma_i(s)\in \cal I_{E_0,q_1 r}\}
\subset
\{i:\gamma_i(t)\in \cal I_{E_0,q_2 r}\}.
\end{equation}

\begin{lemma}[Delocalization for deformed matrices]\label{lemdeloc}
Let  $\tau>0$ and  let $\bu_{i,t}$ denote the normalized eigenvector of  $K(t)$ in  \eqref{eqn:Kt}, whose eigenvalues are $\lambda_{i,t}$, $1\leq i\leq n$.
We assume that $t\in \cal T_\omega(\eta_*, \eta^*, r)$,  $V$ is $(\eta_*,\eta^*, r)$-regular   at $E_0$ and bounded as in (\ref{clMang}),
and that there exists  $C>0$  such that for any $D>0$, for large enough $n$ we have 
\be\label{Hop2}
\P\left(\exists E\in \cal I_{E_0,r},  \eta_*<\eta<rn^{-\omega}:\, \im G(0,E+\ii\eta)_{ii}\ge C  \right)\le n^{-D}
\ee
for any $1\leq i\leq n$. Here $G(0, z)$ is the Green function of the initial matrix $V$. Then for any $\kappa,\tau, D>0$, provided that $n$ is sufficiently  large we have 
$$
\P\left( \mathds{1}_{|\lambda_{k,t}-E_0|\le (1-\kappa)r} \  \| u_{k,t}\|^2_\infty\ge
n^{-1+\tau}\right)\le n^{-D} 
$$
uniformly in $1\leq k\leq n$ .  
\end{lemma}

\begin{remark}\label{rem:changescales}
This lemma is essentially a restatement of \cite[Theorem 2.1]{BouHuaYau2017}, which
holds in the domain
$
\{z=E+\ii \eta: E\in I^r_\kappa (E_0), \psi^4/n\leq \eta \leq1-\kappa r \}$
under the $(\eta_*,  1, r)$-regularity  for $V$, see \cite[Assumption 1.3]{BouHuaYau2017}.

In the above lemma the assumption is weaker: we only have $(\eta_*,\eta^*, r)$-regularity for $V$. 
A simple inspection of the proof of 
\cite[Theorem 2.1]{BouHuaYau2017} shows that its conclusion remains, in the restricted domain $\psi^4/n\leq \eta \leq r n^{-\omega}$ (which will be sufficient for our purpose) under this $(\eta_*,\eta^*, r)$-regularity assumption.
\end{remark}

\begin{proof}
We bound the eigenvectors coordinates by the diagonal entries of the resolvent through
\begin{equation}\label{eqn:deloc1}
|u_{k,t}(i)|^2\leq n^{-1+\tau} \im G(t,\lambda_{k,t}+\ii n^{-1+\tau})_{ii}.
\end{equation}
If 
$|\lambda_{k,t}-E_0|\leq (1-\kappa)r$, \nc  denoting $z=\lambda_{k,t}+\ii n^{-1+\tau}$ we  have 
\begin{equation}\label{eqn:domain}
z\in\{E+\ii\eta:|E-E_0|<(1-\kappa)r, \frac{\psi^4}{n}\leq \eta\leq r\eta^* \psi^{-1}\}.
\end{equation} 
The  local law from \cite[Theorem 2.1]{BouHuaYau2017} with the domain adjustment from Remark \ref{rem:changescales}
states  that $G(t,z)$ is well approximated by $U_0\diag\{g_1(t,z), g_2(t,z),\dots, g_n(t,z)\}U_0^*$, i.e., 
for any $\eta_*\ll t\ll r$ and  any unit vector $\b q$, uniformly for any $z$ as in \eqref{eqn:domain},  the following holds with overwhelming probability:
\begin{align} 
 \left|\langle \b q, G(t,z) \b q\rangle - \sum_{i=1}^{n}\langle \b u_i(0), \b q\rangle^2g_i(t,z)\right|\leq \frac{\psi^2}
 {\sqrt{n\eta}}\im\left(\sum_{i=1}^{n}\langle \b u_i(0), \b q\rangle^2g_i(t,z)\right).
 \end{align}
Clearly, we can restate the last result as 
\begin{align}\label{e:isolaw}
 \left|\langle \b q, G(t,z) \b q\rangle -\langle \b q,  G(0,  z+ t m_{\fc,t}(z)) \b q\rangle \right|\leq \frac{\psi^2}{\sqrt{n\eta}}\im  \langle \b q, G(0,  z+ t m_{\fc,t}(z)) \b q\rangle ,
 \end{align}
where $G (0, z)$ is the Green's function of $V$. 
Since $\psi^2/\sqrt {n \eta} \le 1$, we have \nc 
\begin{equation}\label{eqn:deloc2}
\left|\im G(t,z)_{ii}-\im G(0,z+t m_{\fc, t}(z))_{ii}\right|\leq
\im G(0,z+t m_{\fc, t}(z))_{ii}. 
\end{equation}
From \cite[Proposition 2.2]{BouHuaYau2017}, for some fixed constant $C>0$ we have $C^{-1}<\im m_{\fc, t}(z)<C$
and $|\re m_{\fc, t}(z)|<C \log n$, so that $\re(z+t m_{\fc, t}(z))\in \cal I_{E_0,r}$ and $\eta_*<\im(z+t m_{\fc, t}(z))<n^{-\omega}r$.
With (\ref{Hop2}), we deduce that
\begin{equation}\label{eqn:deloc3}
\im G(0,z+t m_{\fc, t}(z))_{ii}\leq C
\end{equation}
with overwhelming probability. Equations 
(\ref{eqn:deloc1}), (\ref{eqn:deloc2}), (\ref{eqn:deloc3}) conclude the proof.
\end{proof}

Similarly to \cite{BouYau2017, BouHuaYau2017}, we split split the operator $\mathscr{B}(t)$ from (\ref{momentFlotSym}) into a short-range part and a long range part through a short range parameter $\ell$: $\mathscr{B}(t)=\mathscr{S}(t)+\mathscr{L}(t)$, with
\begin{align*}
&(\mathscr{S} f_t)(\boeta)  =   \sum_{0<|j-k|\leq \ell }  c_{jk}(t) 2 \eta_j (1+ 2 \eta_k) \left(f_t(\boeta^{j k}) - f_t(\boeta)\right),\\
&(\mathscr{L}f_t) (\boeta) =    \sum_{|j-k|>\ell}  c_{jk}(t) 2 \eta_j (1+ 2 \eta_k ) \left(f_t(\boeta^{jk}) - f_t(\boeta)\right).\notag
\end{align*}
Notice that $\mathscr{S}$ and $\mathscr{L}$ are also reversible with respect to the measure $\pi$ from \eqref{eqn:weight}. We denote by $\rU_\mathscr{B}(s,t)$ ($\rU_\mathscr{S}(s,t)$ and $\rU_\mathscr{L}(s,t)$) the semigroup associated with $\mathscr{B}$ ($\mathscr{S}$ and $\mathscr{L}$) from time $s$ to $t$, i.e.
\begin{align*}
\del_t \rU_\mathscr{B}(s,t)=\mathscr{B}(t)\rU_\mathscr{B}(s,t).
\end{align*}

For a fixed $\kappa>0$, consider the following ``distance"
on $n$ particle configurations:
\begin{equation}\label{effdis}
 d(\boeta,\boxi)=\max_{1\leq \alpha\leq d}\#\{i\in\llbracket1,n\rrbracket: \gamma_{i}(t_0)\in I_{\kappa}^{r}(E_0), i\in \llbracket x_\alpha, y_{\alpha}\rrbracket\cup \llbracket y_\alpha, x_{\alpha}\rrbracket\},
 \end{equation}
where $\boeta$: $1\leq x_1\leq x_2\leq\cdots \leq x_d\leq n$ and $\boxi$: $1\leq y_1\leq y_2\leq\cdots \leq y_d \leq n$, and the initial
an initial time  $t_0$ defined in the next lemma.  Note that we use the notation $d$ for $\tilde d$ defined in \cite[equation (3.10)]{BouHuaYau2017}.

\begin{lemma}\label{l:regpath}
Assume the initial estimates (\ref{clMang}), (\ref{jjlzi}), (\ref{e:imasup}) and (\ref{Hop1prime}) hold. We fix times $t_0$, $t_1$, and the range parameter $\ell$, such that 
\begin{equation}\label{eqn:parameters}
\psi\eta_*\leq t_0\leq t_1\leq \frac{\ell}{n\psi}\leq \frac{r}{\psi^{10}}.
\end{equation}
The matrix  Brownian motion $(K(s))_{0\leq s\leq t_1}$ defined in Subsection (\ref{subsec:defBM})) induces a measure on the space of eigenvalues and eigenvectors $(\bm \lambda(s), \bm u(s))$ for $0\leq s\leq t_1$, such that, for any $\kappa>0$, the following event $A$ holds with overwhelming probability:
\begin{enumerate}
\label{rigid}\rm\item The eigenvalue rigidity estimate holds:  $\sup_{t_0\leq s\leq t_1}|m_s(z)-m_{{\rm{fc}},s}(z)|\leq \psi (n\eta)^{-1}$  
uniformly in $z\in \cal D_\kappa$,  and $ \sup_{t_0\leq s\leq t_1}|\lambda_{i}(s)-\gamma_{i}(s)|\leq\psi n^{-1}$ uniformly for indices $i$ such that $\gamma_i(s)\in I_r^{\kappa}(E_0)$.
\rm\item When we condition on the trajectory $\bm \lambda\in A$, with overwhelming probability, the following holds:
\begin{align}\label{evcontrol}
\sup_{t_0\leq s\leq t_1}|G(s,z)_{ii}-G(0,z+s m_{\fc, s}(z))_{ii}|&\leq  C  \frac{\psi^2}{\sqrt{n\eta}},\\
\sup_{t_0\leq s\leq t_1}\left|\frac{1}{|I|}\sum_{i\in I}G(s,z)_{ii}-\frac{1}{n}{\rm Tr}G(s,z)\right|&\leq   \frac{C \psi^2}{n^{\mathfrak{c}}}+\frac{C \psi^4}{n\eta}\label{evcontrol2},
\end{align}
uniformly in $z\in \cal D_\kappa$, where $\mathfrak{b},\mathfrak{c}$ are defined in is defined in Assumption \ref{as:basic}.
\rm\item Finite speed of propagation holds: for any $d$ there exists $C_d,c_d>0$ such that uniformly, for any function $h$ on the space of $d$ particle configurations, and particle configuration $\boxi$ which is away from the support of $h$ in the sense that $  d(\boeta,\boxi)\geq \psi \ell$, we have  for any $\boeta$ in the support of $h$ that 
\begin{align}\label{finitespeed}
\sup_{t_0\leq s'\leq s\leq t_1}\rU_{\mathscr{S}}(s',s)h(\boxi)\leq C_d \|h\|_\infty n^d e^{-c_d\psi}.
\end{align}
\end{enumerate}
\end{lemma}

\begin{proof}
The statement (i) was proved in \cite[Theorem 3.3 and 3.5]{LanYau2015},   and (\ref{evcontrol}) and (iii) were given in  Theorem 2.1 and  Lemma 3.4 of  \cite{BouHuaYau2017} respectively.
For the proof of (\ref{evcontrol2}), we decompose
\begin{multline}
\frac{1}{|I|}\sum_{i\in I}G(s,z)_{ii}-\frac{1}{n}{\rm Tr}G(s,z)
=
\frac{1}{|I|}\sum_{i\in I}\left(G(s,z)_{ii}-G(0,z+s m_{\fc, s}(z))_{ii}
\right)\\
+
\left(\frac{1}{|I|}\sum_{i\in I}G(0,z+s m_{\fc, s}(z))_{ii}-
\frac{1}{n}\sum_{1\leq i\leq n}G(0,z+s m_{\fc, s}(z))_{ii}\right)
-\frac{1}{n}\sum_{1\leq i\leq n}\left(G(s,z)_{ii}-G(0,z+s m_{\fc, s}(z))_{ii}
\right) 
\end{multline}
The second sum is exactly the left hand side of (\ref{Hop1prime}), so it is bounded by $n^{-\mathfrak{c}}$. 
The third sum   is just the difference between Stieltjes transforms and it was proved in \cite[Theorem 3.3]{LanYau2015}
is of order at most $n^\e/(n\eta)$ thanks to \eqref{e:isolaw}.  Notice that we have used  $m_{\fc,s}(z) = G(0,z+s m_{\fc, s}(z))$ by definition.

The first sum of the last displayed equation  is of same type as the third one, except that the average is over not all entries but  a macroscopic fraction of them. The proof  in   \cite[Theorem 3.3]{LanYau2015},  based on a fluctuation averaging lemma,   can  be replicated to yield  that
$$
\left|\frac{1}{|I|}\sum_{i\in I}\left(G(s,z)_{ii}-G(0,z+s m_{\fc, s}(z))_{ii}
\right)\right|\leq \frac{n^\e}{n\eta}
$$
with overwhelming probability. This completes  the proof of Lemma \ref{l:regpath}.
\end{proof}

\begin{remark}
The following is an elementary consequence of the above rigidity estimate (i) together with (\ref{eqn:stability}).
For any $t_0\leq s\leq t_1$ and  and interval $I\subset I_\kappa^r(E_0)$ with $|I|\geq \psi^4/n$, we have
\begin{align}\label{e:avdeleig}
C^{-1}|I|n\leq   \#\{i: \gamma_i(s)\in I\}+ \#\{i: \lambda_i(s)\in I\} \leq C|I|n.
\end{align}
\end{remark}

\subsection{Approximation with short range dynamics.}\ 
We introduce the notation
$$
{\rm S}^{(u,v)}_I=\sup_{\boeta\subset I,u\leq s\leq v}f_s(\boeta).
$$
for the following lemma. Moreover, for $i\in \mathbb{Z}$ and $J\subset\mathbb{Z}$, let $d(i,J)=\inf_{j\in J} |i-j|$. Finally, from here, we assume that the number of particles of the eigenvector moment flow is even, i.e.
$$
d=2m.
$$

\begin{lemma}\label{l:Lbound}
Under the assumptions of Lemma  \ref{l:regpath}, consider $\bm \lambda\in A$, with $A$ defined in the same lemma. Consider the perfect matching observables $f_u$
from (\ref{feqC}). Then, for large enough $n$, for any intervals  $J_{\rm in}\subset\{i:\gamma_i(t_0)\in I_{2\kappa}^r(E)\}$ and $J_{\rm out}=\{i:d(i,J_{\rm in})\leq \psi \ell\}$, any
$d$-particle configuration $\boxi$ supported on $J_{\rm in}$, and any $t_0<u<v<t_1$ we have
$$
\left|\left((\rU_{\mathscr{B}}(u,v)-\rU_{\mathscr{S}}(u,v))f_{u}\right)(\boxi)\right|
\leq \psi^{4}\frac{n|u-v|}{\ell}
\left({\rm S}^{(u,v)}_{J_{\rm out}}
+\frac{1}{n^{\mathfrak{c}}}({\rm S}^{(u,v)}_{J_{\rm out}})^{\frac{d-1}{d}}
+\frac{1}{\ell}({\rm S}^{(u,v)}_{J_{\rm out}})^{\frac{d-2}{d}}\right).
$$
\end{lemma}

\begin{proof}
 We first define, similarly to \cite{BouYau2017, BouHuaYau2017}, the following flattening  operators on the space of functions of configurations with $d$ points: 
\begin{align*}
&({\rm Flat}_a (f))(\boeta)=\left\{
\begin{array}{cc}
f(\boeta), & \ {\rm if }\ \boeta\subset \{i:d(i,J_{\rm in})\leq a\},\\
0, &\ {\rm otherwise},
\end{array}
\right.
\end{align*}
By Duhamel's formula,
$$
\left(\left(\rU_{\mathscr{S}}(u,v)-\rU_{\mathscr{B}}(u,v)\right)f_{u}\right)(\boxi)
=\int_{u}^v \rU_{\mathscr{S}}(s,v)\mathscr{L}(s)f_{s} (\boxi)\rd s.
$$
Notice that $d (\supp(\mathscr{L}(s)f_{s}-{\rm Flat}_{\psi\ell}(\mathscr{L}(s)f_{s})), \boxi\}\geq \psi\ell$. Therefore by the finite speed of propagation \eqref{finitespeed} in Lemma \ref{l:regpath} of $\rU_{\mathscr{S}}$,  
we have
\begin{multline}
|(\rU_{\mathscr{S}}(s,v)\mathscr{L}(s)f_{s}) (\boxi)|
=|\rU_{\mathscr{S}}(s,v){\rm Flat}_{\psi\ell}(\mathscr{L}(s)f_{s}) (\boxi)|+\OO(e^{-c\psi/2})\\
\leq  \max_{\wt\boeta}|{\rm Flat}_{\psi\ell}(\mathscr{L}(s)f_{s}) (\wt\boeta)|+\OO(e^{-c\psi/2}).\label{eqn:reduction}
\end{multline}
where in the last inequality, we used that $\rU_{\mathscr{S}}$ is a contraction in ${\rm L}^\infty$.

Let $\wt\boeta$ be a configuration $ \{(i_1,  j_1),  \dots, (i_d , j_d) \}$, with support in $J_{\rm out}$. 
In view of (\ref{eqn:reduction}), we only need to prove that
\begin{equation}\label{eqn:Lbound}
|(\mathscr{L}(s)f_{s}) (\wt\boeta)|\leq \psi^{4}\frac{n}{\ell}
\left({\rm S}^{(u,v)}_{J_{\rm out}}
+\frac{1}{n^{\mathfrak{c}}}({\rm S}^{(u,v)}_{J_{\rm out}})^{\frac{d-1}{d}}
+\frac{1}{\ell}({\rm S}^{(u,v)}_{J_{\rm out}})^{\frac{d-2}{d}}\right).
\end{equation}
We have
$$
\mathscr{L}(s)f_{s}(\wt\boeta)\leq\left|\sum_{|j-k|\geq \ell}\frac{f_{s}(\wt\boeta^{jk})}{n(\lambda_j-\lambda_k)^2}\right|+ |f_{s}(\wt\boeta)|\sum_{1\leq p\leq d,|i_p-k|\geq \ell}\frac{1}{n(\lambda_{i_p}-\lambda_k)^2}.
$$
Notice that $i_p \in J_{\rm out}$, and thus $\lambda_{i_p}(s)\in I_\kappa^r(E_0)$. We denote $\eta_q=2^q\ell/n$. From the local law and a dyadic decomposition we have
$$
\sum_{k:|i_p-k|\geq \ell}\frac{1}{n(\lambda_{i_p}-\lambda_k)^2}
\leq \sum_{q=1}^{\lceil \log_2 n/\ell\rceil}\frac{1}{\eta_q}\sum_{k}\frac{\eta_q}{n((\lambda_{i_p}-\lambda_k)^2+\eta_q^2)}
\leq \frac{n}{\ell},
$$
so that the second term on the right hand side of (\ref{eqn:Lbound}) is bounded by the right hand side of (\ref{eqn:reduction}), as desired.

More subtle bounds are required for
$$
\sum_{|j-k|\geq \ell}\frac{f_{s}(\wt\boeta^{jk})}{n(\lambda_j-\lambda_k)^2}=
\sum_{|j-k|\geq \ell, \wt\eta_k=0}\frac{f_{s}(\wt\boeta^{jk})}{n(\lambda_j-\lambda_k)^2}+\OO\left(\frac{n}{\ell^2}\right)\sup_{u\leq s\leq v,\boeta\subset  J_{\rm out}}
|f_s(\boeta)|
$$
where we used that $\wt\boeta^{jk}\subset J_{\rm out}$ if $\wt\eta_k\neq 0$, and $1/(n(\lambda_j-\lambda_k)^2)\leq 1/(n(\ell/n)^2)$ for $|j-k|\geq \ell$, by rigidity, see Lemma \ref{l:regpath} (i).
For fixed $p$, we therefore want to bound
$$
\sum_{|i_p-k|\geq \ell,\wt\eta_k=0}\sum_{G\in\mathcal{G}_{\wt\boeta^{i_p k}}}\frac{\mathbb{E}(P(G)\mid\bla)}{n(\lambda_{i_p}-\lambda_k)^2}=({\rm I})+({\rm II})
$$
where $(\rm I)$ corresponds to perfect matchings such that  $\{(k,1),(k,2)\}$ is not an edge, and $(\rm II)$ corresponds to perfect matchings with an edge of type $\{(k,1),(k,2)\}$. More precisely,
\begin{align*}
({\rm I})&=\sum_{1\leq q_1,q_2\leq d} \mathbb{E}\left(P^{(q_1,q_2)}(p(e)_{e\in \mathcal{E}_{\wt\boeta}})\sum_{|k-i_{p}|>\ell,\wt\eta_k=0}\frac{p_{i_{q_1}k}p_{i_{q_2}k}}{n(\lambda_{i_p}-\lambda_k)^2}\mid\bla
\right),\\
({\rm II})&=
\mathbb{E}\left(P^{(p)}((p(e)_{e\in \mathcal{E}_{\wt\boeta}}))\sum_{|k-i_p|>\ell,\wt\eta_k=0}\frac{p_{kk}}{n(\lambda_{i_p}-\lambda_k)^2}\mid\bla
\right),
\end{align*}
with $\mathcal{E}_{\wt\boeta}$ the set of all possible edges between between vetices from $\mathcal{V}_{\wt\boeta}$,  $P^{(p,q)}$ is a finite sum of monic monomials of degree $d-2$, and , $P^{(p)}$ is a finite sum of monic monomials of degree $d-1$.

To bound $(\rm I)$, we simply write
$$
\sum_{|k-i_p|>\ell,\wt\eta_k=0}\frac{p_{i_{q_1}k}p_{i_{q_2}k}}{n(\lambda_{i_p}-\lambda_k)^2}=\OO\left(\frac{1}{n(\ell/n)^2}\sum_k(p_{i_{q_1}k}^2+p_{i_{q_2}k}^2)\right)=
\OO\left(n^{\e}\frac{|I|}{\ell^2}\right),
$$
where we slightly changed the meaning of $p_{kk}$ (only in the equation above and the equation below, $p_{kk}=\sum_{\alpha\in I}u_k(\alpha)^2$, i.e. $C_0=0$ in (\ref{eqn:pkk})) and used the elementary identity
\begin{equation}\label{eqn:elementary}
\sum_{k}p_{ik}^2=\sum_{\alpha\in I} u_i(\alpha)^2=\OO\left(n^{\e}\frac{|I|}{n}\right).
\end{equation}
The above second equality follows from Lemma \ref{lemdeloc}.
Moreover, with Lemma \ref{lem:YoungBound}, we have
$$
\mathbb{E}\left(|P^{(q_1,q_2)}(p(e)_{e\in \mathcal{E}_{\wt\boeta}})|\mid\bla\right)=\OO\left(\sup_{u\leq s\leq v,\boeta\subset  J_{\rm out}}
|f_s(\boeta)|\mathds{1}_{\mathcal{N}(\boeta)=d-2}\right)
=\OO\left(({\rm S}^{(u,v)}_{J_{\rm out}})^{\frac{d-2}{d}}\right),
$$
where we used H{\"o}lder's inequality and Lemma \ref{lem:YoungBound}.
This concludes our estimate for ({\rm I}).

The term ({\rm II}) is more complicated to bound. For fixed $p$ and $s$, let $E_1=\gamma_{i_p-\ell}$, $E_1^-=\gamma_{i_p-\ell-n^\e}$, $E_1^+=\gamma_{i_p-\ell+n^{\e}}$, $E_2=\gamma_{i_p+\ell}$, $E_2^-=\gamma_{i_p+\ell-n^\e}$, $E_2^+=\gamma_{i_p+\ell+n^\e}$. We also define the contour $\Gamma$ as the rectangle with vertices $E_1\pm\ii\frac{\ell}{n}$, $E_2\pm\ii\frac{\ell}{n}$.
Let 
\begin{align}
f(z)&=\sum_{k:\gamma_k\not\in[E_1^-,E_2^+]}\frac{p_{kk}}{n(z-\lambda_k)}\label{eqn:bd1}\\
g(z)&=\sum_{k:\gamma_k\not\in[E_1^-,E_1^+]\cup[E_2^-,E_2^+]}\frac{p_{kk}}{n(z-\lambda_k)}\notag
\end{align}
We now assume $|z-\lambda_{i_p}|\leq n^{-\e}\frac{\ell}{n}$.
By Cauchy's formula, we have
$$
f(z)=\frac{1}{2\pi \ii}\int_{\Gamma}\frac{f(\xi)}{\xi-z}\rd \xi=\frac{1}{2\pi \ii}\int_{\Gamma}\frac{g(\xi)}{\xi-z}\rd \xi,
$$
where for the second equality we used that, for any $\lambda_k$ (and $z$) inside $\Gamma$ we have
$$
\int_{\Gamma}\frac{\rd \xi}{(\xi-\lambda_k)(\xi-z)}=0,
$$
from a residue calculus. Define $\Gamma_{\rm int}=\{z=E+\ii\eta:E=E_1\ \mbox{or}\ E_2, |\eta|<n^\e/n\}$ and $\Gamma_{\rm ext}=\Gamma\slash\Gamma_{\rm int}$.
We first bound the contribution due to small $\eta$: we have
$$
\left|\int_{\Gamma_{\rm int}}\frac{g(\xi)}{\xi-z}\rd \xi\right|\leq
\frac{n}{\ell}\int_{\Gamma_{\rm int}}
\sum_{\substack{k<i_p-\ell-n^\e\\i_p+\ell+n^\e<k\\i_p-\ell+n^\e<k<i_p+\ell-n^\e}}\frac{|p_{kk}|}{n|\lambda_k-\xi|}
$$
We simply bound $|p_{kk}|$ by 1 and obtain that the corresponding integral is at most 
$$\left|\int_{\Gamma_{\rm int}}\frac{g(\xi)}{\xi-z}\rd \xi\right|\leq\frac{n}{\ell}\frac{n^\e}{n}\sum_{k\geq \ell}\frac{1}{n(k/n)}=\OO\left(\frac{n^\e}{\ell}\right).$$
We now bound the contribution from $\Gamma_{\rm ext}$. On this domain, we can afford extending the definition of $g$ to the full sum $1\leq k\leq n$, up to an error of order
$$
n^\e \frac{n}{\ell}\int_{\Gamma_{\rm Ext}}\frac{|p_{kk}|}{n|\xi-E_1|}\leq \frac{n^\e}{\ell}. 
$$
We therefore proved
\begin{equation}\label{eqn:bd2}
f(z)\leq
\frac{n}{\ell}\int_{\Gamma_{\rm Ext}}\left|\frac{1}{n}\sum_{1\leq k\leq n}\frac{p_{kk}}{\xi-\lambda_k}\right||\rd \xi|+\frac{n^\e}{\ell}
\leq
\frac{n}{\ell}\int_{\Gamma_{\rm Ext}}\left(\frac{\psi^4}{n\im \xi}+\frac{\psi}{n^\mathfrak{c}}\right)|\rd \xi|+\frac{n^\e}{\ell}=\OO(n^\e)\left(\frac{1}{\ell}+\frac{1}{n^{\mathfrak{c}}}\right),
\end{equation}
where we used (\ref{evcontrol2}) in the second inequality.
We conclude that 
$\partial_z f(\lambda_{i_p})=\OO(n^\e)\frac{n}{\ell}\left(\frac{1}{\ell}+\frac{1}{n^{\mathfrak{c}}}\right)$, so that
$$
({\rm II})=\OO\left(\frac{n}{\ell^2}+\frac{n}{\ell n^{\mathfrak{c}}}\right)({\rm S}^{(u,v)}_{J_{\rm out}})^{\frac{d-1}{d}}
$$
where we used H{\"o}lder's inequality and Lemma \ref{lem:YoungBound}.
This concludes the proof of (\ref{eqn:Lbound}) and the lemma (note that $\frac{n}{\ell^2}({\rm S}^{(u,v)}_{J_{\rm out}})^{\frac{d-1}{d}}=\OO(\frac{n}{\ell^2}({\rm S}^{(u,v)}_{J_{\rm out}})^{\frac{d-2}{d}})$).
\end{proof}

\begin{lemma}\label{lem:YoungBound}
Denote by $\boeta$  the configuration with $m$ particles at site $i$, $m$ particles at site $j$  and no particles elsewhere. Moreover, 
denote by $ \boeta^{(1)}$ ($ \boeta^{(2)}$ resp.)  the configurations with $d=2m$ particles on the site $i$ (site $j$ resp.) and no particles elsewhere. 
Then there exists $C_1,C_2,C>0$ depending only on $d$ such that 
for any $i<j$ and any time $s$ we have
$$
\mathbb{E}\left(p_{ij}(s)^d\mid \bla\right)
\leq
C_1f_{\bla, s}(\boeta^{(1)})+C_2 f_{\bla, s}(\boeta^{(2)})+C f_{\bla, s} (\boeta).
$$
\end{lemma}

\begin{proof}
From (\ref{feqC}), we have
\begin{equation}\label{eqn:mom1}
f_{\bla, s}(\boeta)=a_d \E(p_{ij}^d\mid\bla)+
\sum_{\alpha+\beta+\gamma=d,\alpha<d} b_{\alpha,\beta,\gamma}\E(p_{ij}^\alpha p_{ii}^\beta p_{jj}^\gamma\mid\bla)
\end{equation}
for some coefficients $a_d>0$, $b_{\alpha,\beta,\gamma}\geq 0$. From Young's inequality, for any $\varepsilon>0$ we have
\begin{equation}\label{eqn:mom2}
\left|\E(p_{ij}^\alpha p_{ii}^\beta p_{jj}^\gamma\mid\bla)\right|\leq
\frac{\alpha\varepsilon^2}{d}\E(p_{ij}^{d}\mid\bla)+
\frac{\beta}{d\varepsilon}\E(p_{ii}^{d}\mid\bla)+
\frac{\gamma}{d\varepsilon}\E(p_{jj}^{d}\mid\bla).
\end{equation}
Equations (\ref{eqn:mom1}) and (\ref{eqn:mom2})  imply
$$
\E(p_{ij}^{d}\mid\bla)\leq \frac{f_{\bla, s}(\boeta)}{a_d}+\sum_{\alpha+\beta+\gamma=d,\alpha<d} \frac{b_{\alpha,\beta,\gamma}}{a_d}\left(
\frac{\alpha \varepsilon^2}{d}\E(p_{ij}^{d}\mid\bla)+\frac{\beta\E(p_{ii}^{d}\mid\bla)+\gamma\E(p_{jj}^{d}\mid\bla)}{d\varepsilon}
\right).
$$
The result follows by choosing $\varepsilon=\varepsilon(d)$ small enough.
\end{proof}

\subsection{Maximum principle.}\ Iterations of the following proposition will give the main result, Theorem \ref{thmQUE}.

\begin{proposition}\label{thm:maxPrincipleQUE}
For any eigenvalue trajectory $(\bm\la(s))_{0\leq s\leq t_1}\in A$  defined in Lemma \ref{l:regpath}, let $f$ be a solution of the $d$-particle eigenvector moment flow (\ref{ve}) with initial matrix $K(0)$.
For any $C>0$, there exists $n_0$ such that for any $n\geq n_0$ the following holds. For any intervals  $J_{\rm in}\subset\{i:\gamma_i(t_0)\in I_{3\kappa}^r(E_0)\}$,  $J_{\rm out}=\{i:d(i,J_{\rm in})\leq n r/\psi\}$, and $[t,t+u]\subset[t_0,t_1]$ with $u>t/\psi$, we have
\begin{equation}\label{fest2QUE}
{\rm S}^{(t+\frac{u}{2},t+u)}_{J_{\rm in}}
\leq 
\psi^3 \left(\left(\frac{u}{r}\right)^{1/2}+\frac{1}{n t}\right){\rm S}^{(t,t+u)}_{J_{\rm out}}
+\frac{\psi^3 }{n^{\mathfrak{c}}}({\rm S}^{(t,t+u)}_{J_{\rm out}})^{\frac{d-1}{d}} 
+\frac{\psi^3 }{n t}({\rm S}^{(t,t+u)}_{J_{\rm out}})^{\frac{d-2}{d}} +n^{-C}.
\end{equation}
\end{proposition}

\begin{proof} 
For a general number of particles $d$, consider now the following modification of the eigenvector moment flow (\ref{ve}).
We only keep the short-range dynamics (depending on the short range parameter $\ell$, chosen later) and modify the initial condition to be zero when there is a particle far from $J_{\rm in}$:
\begin{align}\begin{split}\label{eqn:modMom}
&\partial_s g_{s} =  \mathscr{S}(s) g_{s},\ t\leq s\leq t+u,\\
&g_{t}(\boeta)=({\rm Av}f_{t})(\boeta),\end{split}
\end{align}
where
$$
{\rm Av}(f)=\frac{3\psi}{n r}\sum_{\frac{1}{3}\frac{nr}{\psi}<a<\frac{2}{3}\frac{nr}{\psi}}{\rm Flat}_a(f).
$$
We can write 
$$
{\rm Av}(f)(\boeta)=a_{\boeta} f(\boeta)
$$
for some coefficient $a_{\boeta}\in[0,1]$ ($a_{\boeta}=0$ if $\boeta\not\subset J_{\rm out}$, $\; a_{\boeta} =1$ if $
\boeta\subset J_{\rm in} $). We will only use the elementary property
\begin{equation}\label{eqn:propa}
|a_{\boeta}-a_{{\boxi}}|\leq\frac{\psi}{nr} d(\boeta,\boxi),
\end{equation}
where the distance is defined in \eqref{effdis}.

For any $\boeta\subset J_{\rm in}$, we have
\begin{align}\notag
|f_s(\boeta)-g_s(\boeta)|&\leq \left|\left(\rU_{\mathscr{B}}(t,s)f_{t}-\rU_{\mathscr{S}}(t,s)f_{t}\right)(\boeta)\right|
+\left|\rU_{\mathscr{S}}(t,s)(f_{t}-{\rm Av}f_{t})(\boeta)\right|\\
\label{eqn:ftog}&\leq
\psi^{4}\frac{nu}{\ell}
\left({\rm S}^{(t,t+u)}_{J_{\rm out}}
+\frac{1}{n^{\mathfrak{c}}}({\rm S}^{(t,t+u)}_{J_{\rm out}})^{\frac{d-1}{d}} 
+\frac{1}{\ell}({\rm S}^{(t,t+u)}_{J_{\rm out}})^{\frac{d-2}{d}}\right)
+e^{-c\psi/2},
\end{align}
where we bounded the first term by Lemma \ref{l:Lbound}, and the second term by finite speed of propagation \eqref{finitespeed}, since $f_{t_0}-{\rm{Av}}f_{t_0}$ vanishes for any $\boxi$ such that $\boxi\subset 
\{i:d(i,J_{\rm in})\leq nr /3\psi\}$ (note that $\psi\ell\leq nr/3\psi$). 

In the following we will prove that for large enough $n$ we have
\begin{multline}\label{eqn:g}
\sup_{\boeta\subset J_{\rm in}, t+\frac{u}{2}\leq s\leq t+u} g_s(\boeta)\leq
\psi\left(\frac{n u}{\ell}+\frac{\ell}{nr}+\frac{\psi^2}{n t}\right){\rm S}^{(t,t+u)}_{J_{\rm out}}\\
+\frac{\psi }{n^{\mathfrak{c}}}({\rm S}^{(t,t+u)}_{J_{\rm out}})^{\frac{d-1}{d}} 
+\psi\left(\frac{n u}{\ell^2}+\frac{\psi^2}{n t}\right)({\rm S}^{(t,t+u)}_{J_{\rm out}})^{\frac{d-2}{d}}+n^{-C}
\end{multline}
by a maximum principle argument.
Equations (\ref{eqn:ftog}) and (\ref{eqn:g}) together give the expected result (\ref{fest2QUE}) by choosing
$$
\ell=n\psi^2(u r)^{1/2},
$$
which satisfies (\ref{eqn:parameters})
If the left hand side of (\ref{eqn:g}) is smaller than $n^{-C}$, there is nothing to prove. It it is greater thn $n^{-C}$, by the finite speed of propagation property (\ref{finitespeed}), for any $t<s<t+u$, the configuration(s) 
$\tilde{ \bm\eta}$ such that
\begin{align*}
g_s(\tilde{ \bm\eta})=\sup_{\boeta} g_s(\bm \eta)
\end{align*} 
need to be supported in $\{i:d(i,J_{\rm in})\leq \frac{3}{4}\frac{n r}{\psi}\}$.

From the dynamics \eqref{eqn:modMom}, for any parameter $\psi^4/n\leq \eta\leq\ell/n$ to be chosen, we have
\begin{multline}
\del_t\s g_s(\tilde{\bm \eta})=\sum_{0<|j-k|\leq \ell} c_{jk} 2\tilde \eta_j (1+2\tilde\eta_k) \left(g_s(\tilde\boeta^{jk})-g_s(\tilde\boeta)\right)
\leq \frac{C}{n}\sum_{1\leq p\leq d,\atop k:0<|i_p-k|\leq \ell} \frac{g_s(\tilde \boeta^{i_pk})-g_s(\tilde{\boeta})}{(\lambda_{i_p}-\lambda_k)^2+\eta^2}\\
=\frac{1}{n\eta}\sum_{1\leq p\leq d,\atop k:0<|{i_p}-k|\leq \ell}\im \frac{g_s(\tilde{\bm \eta}^{{i_p}k})}{z_{i_p}-\lambda_k}-\frac{1}{n\eta}g_s(\tilde{\bm \eta})\sum_{1\leq p\leq d,\atop k:0<|{i_p}-k|\leq \ell}\Im \frac{1}{z_{i_p}-\lambda_k}\label{eqn:tobound}
\end{multline}
where we define  $z_{i_p}=\lambda_{i_p}+\ii \eta$. For the second term in \eqref{eqn:tobound}, note that
\begin{align*}
\sum_{1\leq p\leq d,\atop k:0<|{i_p}-k|\leq \ell}\Im \frac{1}{z_{i_p}-\lambda_k}
\geq \sum_{p=1}^d\sum_{k:0<|{i_p}-k|\leq \ell}\frac{\eta}{(\lambda_{i_p}-\lambda_k)^2+\eta^2}
\geq \sum_{p=1}^{d}\sum_{k: |\lambda_k-\lambda_{i_p}|\leq\eta} \frac{\eta}{2\eta^2}
\gtrsim n,
\end{align*}
where we used \eqref{e:avdeleig}.
For the first term in \eqref{eqn:tobound}, we claim that for any fixed $p$ we have 
\begin{multline}
\label{eqn:keybound}
\frac{1}{n}\sum_{k:0<|{i_p}-k|\leq \ell}\im \frac{g_s(\tilde\boeta^{{i_p}k})}{z_{i_p}-\lambda_k}
=
\OO\left(\psi\right)
\left(\frac{1}{n\eta}+\frac{\ell}{nr}+\frac{n u}{\ell}\right){\rm S}^{(t,t+u)}_{J_{\rm out}}\\
+\OO\left(\psi\right)
\frac{1}{n^{\mathfrak{c}}}
({\rm S}^{(t,t+u)}_{J_{\rm out}})^{\frac{d-1}{d}}
+
\OO\left(\psi\right)
\left(\frac{1}{n\eta}
+
\frac{n u}{\ell^2}
\right)
({\rm S}^{(t,t+u)}_{J_{\rm out}})^{\frac{d-2}{d}}.
\end{multline}
For this, we can bound the left hand side of \eqref{eqn:keybound} 
by $\eqref{eqn:term1}+\eqref{eqn:term2}+\eqref{eqn:term3}$ where
\begin{align}
\label{eqn:term1}&\im\sum_{k:0<|k-{i_p}|\leq \ell}\frac{1}{n}\frac{(\rU_{\mathscr{S}}(t,s){\rm Av}f_{t})(\tilde\boeta^{{i_p}k})-({\rm Av}\rU_{\mathscr{S}}(t,s)f_{t})(\tilde\boeta^{{i_p}k})}{z_{i_p}-\la_k},\\
\label{eqn:term2}&\im\sum_{k:0<|{i_p}-k|\leq \ell}\frac{1}{n}\frac{({\rm Av}\rU_{\mathscr{S}}(t,s)f_{t})(\tilde\boeta^{{i_p}k})-({\rm Av}\rU_{\mathscr{B}}(t,s)f_{t})(\tilde\boeta^{{i_p}k})}{z_{i_p}-\la_k},\\
\label{eqn:term3}&\im\sum_{k:0<|{i_p}-k|\leq \ell}\frac{1}{n}\frac{({\rm Av}\rU_{\mathscr{B}}(t,s)f_{t})(\tilde\boeta^{{i_p}k})}{z_{i_p}-\la_k}.
\end{align}
The term \eqref{eqn:term1} will be controlled by finite speed of propagation;  \eqref{eqn:term2} will be controlled by  Lemma \ref{l:Lbound}, and \eqref{eqn:term3} by  the local  law. 

To bound \eqref{eqn:term1}, we write   
$$
(\rU_{\mathscr{S}}(t,s){\rm Av}f_{t})(\tilde\boeta^{{i_p}k})-({\rm Av}\rU_{\mathscr{S}}(t,s)f_{t})(\tilde\boeta^{{i_p}k})
=\frac{2\psi}{nr}\sum_{\frac{nr}{2\psi}<a<\frac{nr}{\psi}}\left(\rU_{\mathscr{S}}(t,s){\rm Flat}_a f_{t}-{\rm Flat}_a\rU_{\mathscr{S}}(t,s) f_{t}\right)(\tilde \boeta^{{i_p}k}).
$$
For fixed $a$, let  $L_{1}\subset L_2$ be defined as $L_1=\{i:d(i,J_{\rm in})\leq a-\psi \ell\}$, $L_{2}=\{i:d(i,J_{\rm in})\leq a+\psi \ell\}$.
We consider three cases: $\tilde\boeta^{{i_p}k}\not \subset L_2$, $\tilde\boeta^{{i_p}k}\subset L_1$, or neither of them.

For $\tilde\boeta^{{i_p}k}\not \subset L_2$, by our definition, ${\rm Flat}_a\rU_{\mathscr{S}}(t,s) f_{t}(\tilde \boeta^{{i_p}k})=0$. 
By the finite speed of propagation \eqref{finitespeed}, the total mass of 
$\rU_{\mathscr{S}}(t,s){\rm Flat}_a f_{t}$ outside $L_2$ is exponentially small. In particular, $|\rU_{\mathscr{S}}(t,s){\rm Flat}_a f_{t}(\tilde\boeta^{{i_p}k})|\leq \exp(-c\psi/2)$.

For $\tilde\boeta^{{i_p}k}\subset L_1$, we have
\begin{align*}
\left|\left(\rU_{\mathscr{S}}(t,s){\rm Flat}_a f_{t}-{\rm Flat}_a\rU_{\mathscr{S}}(t,s) f_{t}\right)(\tilde \boeta^{{i_p}k})\right|
=&\left|\left(\rU_{\mathscr{S}}(t,s){\rm Flat}_a f_{t}-\rU_{\mathscr{S}}(t,s) f_{t}\right)(\tilde \boeta^{{i_p}k})\right|\\
=&|\left(\rU_{\mathscr{S}}(t,s)\left(f_{t}-{\rm Flat}_a f_{t}\right)\right)(\tilde \boeta^{{i_p}k})|\leq \exp(-c\psi/2),
\end{align*}
we used the finite speed of propagation \eqref{finitespeed} in the last inequality, since $f_{t}-{\rm Flat}_a f_{t}$ vanishes for any $\boxi$ supported in $\{i:d(i,J_{\rm in})\leq a\}$.

For the last case, we have $\tilde\boeta^{{i_p}k}\subset L_2$, and some particle(s) of $\tilde\boeta^{{i_p}k}$ is in $L_2\slash L_1$. There are at most $2n\psi \ell$ such $a$. Moreover, since $\rU_{\mathscr{S}}$ is a contraction in ${\rm L}^{\infty}$, we have
\begin{align*}
&\left|\left(\rU_{\mathscr{S}}(t,s){\rm Flat}_a f_{t}-{\rm Flat}_a\rU_{\mathscr{S}}(t,s) f_{t}\right)(\tilde \boeta^{{i_p}k})\right|\\
\leq &\left|\rU_{\mathscr{S}}(t,s){\rm Flat}_a f_{t}\right|+\left|{\rm Flat}_a\rU_{\mathscr{S}}(t,s){\rm Flat}_{a+2\psi\ell} f_{t}(\tilde \boeta^{{i_p},k})\right|+\left|{\rm Flat}_a\rU_{\mathscr{S}}(t,s)\left(f_{t}-{\rm Flat}_{a+2\psi\ell} f_{t}\right)(\tilde \boeta^{{i_p}k})\right|\\
\leq &\|{\rm Flat}_a f_{t}\|_{\infty}+\|{\rm Flat}_{a+\psi\ell} f_{t}\|_{\infty}+e^{-c\psi/2}.
\end{align*}
We bound $\|{\rm Flat}_a f_{t}\|_{\infty}, \|{\rm Flat}_{a+2\psi\ell} f_{t}\|_{\infty}\leq {\rm S}^{(t,t+u)}_{J_{\rm out}}$.
From these estimates, we have $\eqref{eqn:term1}\leq \psi^2\frac{\ell}{n r} {\rm S}^{(t_0,t_0+u)}_{J_{\rm out}}$.

We now bound \eqref{eqn:term2}. For $|k-i_p|\leq \ell$, $\wt \boeta^{i_p k}$ is supported in $\{i:\gamma_i(t_0)\in I_{2\kappa}^r(E)\}$, so that we can apply 
Lemma \ref{l:Lbound}:
\begin{multline*}
\left|({\rm Av}\rU_{\mathscr{S}}(t,s)f_{t})(\tilde\boeta^{{i_p}k})-({\rm Av}\rU_{\mathscr{B}}(t,s)f_{t})(\tilde\boeta^{{i_p}k})\right|
\leq \left|(\rU_{\mathscr{S}}(t,s)f_{t}-\rU_{\mathscr{B}}(t,s)f_{t})(\tilde\boeta^{{i_p}k})\right|\\
\leq
\psi^{4}\frac{n u}{\ell}
\left({\rm S}^{(t,t+u)}_{J_{\rm out}}
+\frac{1}{n^{\mathfrak{c}}}
({\rm S}^{(t,t+u)}_{J_{\rm out}})^{\frac{d-1}{d}}
+\frac{1}{\ell}({\rm S}^{(t,t+u)}_{J_{\rm out}})^{\frac{d-2}{d}}\right).
\end{multline*}
As a consequence, we have 
$$
\eqref{eqn:term2}\leq
\psi^{4}\frac{n u}{\ell}
\left({\rm S}^{(t,t+u)}_{J_{\rm out}}
+\frac{1}{n^{\mathfrak{c}}}
({\rm S}^{(t,t+u)}_{J_{\rm out}})^{\frac{d-1}{d}}
+\frac{1}{\ell}({\rm S}^{(t,t+u)}_{J_{\rm out}})^{\frac{d-2}{d}}\right).
$$

Finally for \eqref{eqn:term3}, note that $\tilde\boeta^{i_pk}$ is supported on $J_{\rm out}$,
so that
\begin{multline*}
\frac{1}{n}\im\sum_{k:0<|{i_p}-k|\leq \ell}\frac{({\rm Av}f_{t})(\tilde\boeta^{{i_p}k})}{z_{i_p}-\la_k}
=\frac{1}{n}\im\sum_{k:0<|{i_p}-k|\leq \ell}\frac{a_{\tilde\boeta}f_{t}(\tilde\boeta^{{i_p}k})+(a_{\tilde\boeta^{{i_p}k}}-a_{\tilde{\boeta}})f_{t}(\tilde\boeta^{{i_p}k})}{z_{i_p}-\la_k}\\
=\frac{a_{\tilde\boeta}}{n}\im\sum_{k:0<|{i_p}-k|\leq \ell}\frac{f_{t}(\tilde\boeta^{{i_p}k})}{z_{i_p}-\la_k}+\OO\left(\psi \frac{\ell}{nr}{\rm S}^{(t,t+u)}_{J_{\rm out}}\right),
\end{multline*}
where we used that $|a_{\tilde\boeta^{{i_p}k}}-a_{\tilde{\boeta}}|\leq \psi d(\tilde\boeta, \tilde\boeta^{{i_p}k})/(n r)\leq \psi \ell/(n r)$ from (\ref{eqn:propa}). 

In the above imaginary part, the contribution of all $k\in\{i_1,\dots,i_d\}$ is of order $\frac{1}{n\eta}{\rm S}^{(t,t+u)}_{J_{\rm out}}$, so that (here $k_0$ is any index not  in $\{i_1,\dots,i_d\}$)
\begin{multline*}
\frac{1}{n}\im\sum_{k:0<|{i_p}-k|\leq \ell}\frac{f_{t}(\tilde\boeta^{{i_p}k})}{z_{i_p}-\la_k}=
\frac{1}{n}\frac{1}{\mathcal{M}(\wt\boeta{i_pk_0})}\im\sum_{k:0<|{i_p}-k|\leq \ell}\sum_{G\in\mathcal{G}_{\wt\boeta^{i_p k}}}\frac{\mathbb{E}(P(G)\mid\bla)}{z_{i_p}-\lambda_k}+\OO\left(\frac{1}{n\eta}{\rm S}^{(t,t+u)}_{J_{\rm out}}\right)\\
=({\rm I})+({\rm II})+\OO\left(\frac{1}{n\eta}{\rm S}^{(t,t+u)}_{J_{\rm out}}\right)
\end{multline*}
where $(\rm I)$ corresponds to perfect matchings for which $\{(k,1),(k,2)\}$ is not an edge, and $(\rm II)$ corresponds to perfect matchings for which $\{(k,1),(k,2)\}$ is an edge. More precisely,
\begin{align*}
({\rm I})&=\im\sum_{1\leq q_1,q_2\leq d} \mathbb{E}\left(P^{(q_1,q_2)}(p(e)_{e\in \mathcal{E}_{\wt\boeta}})\sum_{k:0<|{i_p}-k|\leq \ell}\frac{p_{i_{q_1}k}p_{i_{q_2}k}}{n(z_{i_p}-\lambda_k)}\mid\bla
\right),\\
({\rm II})&=\im
\mathbb{E}\left(P^{(p)}((p(e)_{e\in \mathcal{E}_{\wt\boeta}}))\sum_{k:0<|{i_p}-k|\leq \ell}\frac{p_{kk}}{n(z_{i_p}-\lambda_k)}\mid\bla
\right),
\end{align*}
with $\mathcal{E}_{\wt\boeta}$ the set of all possible edges between between vertices from $\mathcal{V}_{\wt\boeta}$,  $P^{(p,q)}$ is a finite sum of monic monomials of degree $n-2$, and , $P^{(p)}$ is a finite sum of monic monomials of degree $n-1$.

To bound $(\rm I)$, we simply write
$$
\im\sum_{k:0<|{i_p}-k|\leq \ell}\frac{p_{i_{q_1}k}p_{i_{q_2}k}}{n(z_{i_p}-\lambda_k)}=\OO\left(\frac{1}{n\eta}\sum_k(p_{i_{q_1}k}^2+p_{i_{q_2}k}^2)\right)=
\OO\left(\frac{1}{n\eta}n^{\e}\frac{|I|}{n}\right).
$$ 
Here we slightly changed the meaning of $p_{kk}$ (in both equations above and below, $p_{kk}=\sum_{\alpha\in I}u_k(\alpha)^2$, i.e. $C_0=0$ in (\ref{eqn:pkk})) and used the elementary identity (\ref{eqn:elementary}).
The above second equality follows from Lemma \ref{lemdeloc}. 

Moreover, with Lemma \ref{lem:YoungBound}, we have
$$
\mathbb{E}\left(|P^{(q_1,q_2)}(p(e)_{e\in \mathcal{E}_{\wt\boeta}})|\mid\bla\right)=\OO\left(\sup_{t\leq s\leq t+u,\boeta\subset  J_{\rm out}}
|f_s(\boeta)|\mathds{1}_{\mathcal{N}(\boeta)=n-2}\right)
=\OO\left(({\rm S}^{(t,t+u)}_{J_{\rm out}})^{\frac{d-2}{d}}\right),
$$
where we used H{\"o}lder's inequality and Lemma \ref{lem:YoungBound}.
This concludes our bound for ({\rm I}), $\frac{1}{n\eta}({\rm S}^{(t,t+u)}_{J_{\rm out}})^{\frac{d-2}{d}}$.

More subtle bounds are required for the term ({\rm II}).
$$
\im
\sum_{k:0<|{i_p}-k|\leq \ell}\frac{p_{kk}}{n(z_{i_p}-\lambda_k)}
=
\OO\left(\frac{1}{n\eta}+\frac{1}{n^{\mathfrak{c}}}\right)-
\im
\sum_{k:|{i_p}-k|>\ell}\frac{p_{kk}}{n(z_{i_p}-\lambda_k)}
$$
where we used (\ref{evcontrol2}). This last term can be bounded exactly as between (\ref{eqn:bd1}) and (\ref{eqn:bd2}), and we obtain
$$
\im
\sum_{k:0<|{i_p}-k|\leq \ell}\frac{p_{kk}}{n(z_{i_p}-\lambda_k)}
=
\OO\left(\frac{1}{n\eta}+\frac{1}{n^{\mathfrak{c}}}\right),
$$
where we used that $\eta\leq \ell/n$. This concludes the proof of (\ref{eqn:keybound}).

We define $h(s)=\sup_{\boeta} g_s(\bm \eta)$. Equations (\ref{eqn:tobound}) and (\ref{eqn:keybound}) yield
$$
h'(s)\leq 
\frac{C\psi}{\eta}\left(
\left(\frac{1}{n\eta}+\frac{\ell}{nr}+\frac{n u}{\ell}\right){\rm S}^{(t,t+u)}_{J_{\rm out}}
+\frac{1}{n^{\mathfrak{c}}}({\rm S}^{(t,t+u)}_{J_{\rm out}})^{\frac{d-1}{d}} 
+
\left(\frac{1}{n\eta}
+
\frac{n u}{\ell^2}
\right)
({\rm S}^{(t,t+u)}_{J_{\rm out}})^{\frac{d-2}{d}}\right)
-c\frac{h(s)}{\eta}
$$
for any $t<s<t+u$.
We now chose $\eta=t/\psi^2$, so that $u/\eta>\psi$, and obtain
$$
h(s)<C\psi
\left(\frac{1}{n\eta}+\frac{\ell}{nr}+\frac{n u}{\ell}\right){\rm S}^{(t,t+u)}_{J_{\rm out}}
+C\frac{\psi }{n^{\mathfrak{c}}}({\rm S}^{(t,t+u)}_{J_{\rm out}})^{\frac{d-1}{d}} 
+
C\psi
\left(\frac{1}{n\eta}
+
\frac{n u}{\ell^2}
\right)
({\rm S}^{(t,t+u)}_{J_{\rm out}})^{\frac{d-2}{d}}
+n^{-C}
$$
for any $t+u/2<s<t+u$, which is (\ref{eqn:g}) and concludes the proof.\end{proof}

\begin{proof}[Proof of Theorem \ref{thmQUE}]
We proceed by iterating the bound from Proposition \ref{thm:maxPrincipleQUE}. We are given a small $\e$ such that $\e<\mathfrak{a}/5$ and a large $D>0$, as in the statement of Theorem \ref{thmQUE}.

We first choose $d= \lfloor 5D/\e\rfloor$, and define (implicitly, for $J_{i+1}$)
$$
\left\{
\begin{array}{l}
s_0=t_0\\
s_{i+1}=\frac{s_i+t_1}{2}
\end{array}
\right.,\ 
\left\{
\begin{array}{l}
J_0=\{i:\gamma_i(t_0)\in I_{3\kappa}^r(E_0)\}\\
J_{i}=\{i:d(i,J_{i+1})\leq \frac{nr}{\psi}\}
\end{array}
\right..
$$
A direct application of Proposition \ref{thm:maxPrincipleQUE} together with the bounds $n^{-1+\mathfrak{a}}\leq t_0\leq t_1\leq n^{-\mathfrak{a}}r$ yields
$$
{\rm S}^{(s_{i+1},t_1)}_{J_{i+1}}
\leq 
\psi^3 \left(
n^{-\mathfrak{a}/2}+2^in^{-\mathfrak{a}}
\right){\rm S}^{(s_i,t_1)}_{J_{i}}
+\frac{\psi^3 }{n^{\mathfrak{c}}}({\rm S}^{(s_i,t_1)}_{J_{\rm out}})^{\frac{d-1}{d}} 
+\frac{\psi^3 2^i}{n t_0}({\rm S}^{(s_i,t_1)}_{J_i})^{\frac{d-2}{d}} +n^{-C}.
$$
In particular, we have
$$
{\rm S}^{(s_{i+1},t_1)}_{J_{i+1}}
\leq n^{-\e/3}{\rm S}^{(s_{i},t_1)}_{J_{i}}
$$
provided that
$$
({\rm S}^{(s_{i},t_0)}_{J_{i}})^{1/d}
\geq
\frac{n^{\e/2}2^{i}}{\sqrt{n t_1}}+\frac{n^{\e/2}}{n^{\mathfrak{c}}}.
$$
This implies that for $k=\lfloor 4\e^{-1}\rfloor$
we have
$$
\left|({\rm S}^{(s_{k},t_1)}_{J_{k}})^{1/d}\right| \leq \frac{n^{3\e/4}}{\sqrt{nt_0}}+\frac{n^{3\e/4}}{n^{\mathfrak{c}}}.
$$
For each fixed $i$,  by choosing  $ \boeta$ as  the configuration with $d=2m$ particles on the site $i$  and no particles elsewhere, 
we have $|p_{ii}(s) |^d \le C_d f_s (\boeta)$. Hence 
by Markov's inequality, the last displayed equation  implies that
$$
\P\left(
\exists\, s_k<t<t_1: \mathds{1}_{\lambda_i(t)\in I_{4\kappa}^r(E_0)} |p_{ii}|\geq n^{\e}\left(\frac{1}{n^{\mathfrak{c}}}+\frac{1}{\sqrt{n t_0}}\right)\right)\le
 \left(n^{\e}\right)^{-d}\left(n^{3\e/4}\right)^d\leq n^{-D}.
$$
Here we used that $\{i:\gamma_i(t_0)\in I_{4\kappa}^r(E_0)\}\subset J_{k}$ because $k\frac{nr}{\psi}<\kappa r$ for $k=\lfloor 4\e^{-1}\rfloor$ and $n$ large enough.

Finally, by Lemma \ref{lem:YoungBound}, $p_{ij}^d$ can be estimated in terms $f_t$, $p_{ii}$ and $p_{jj}$. Hence the previous estimate also holds if we replace $\mathds{1}_{\lambda_i(t)\in I_{4\kappa}^r(E_0)} |p_{ii}|$ by $\mathds{1}_{\lambda_i(t),\la_j(t)\in I_{4\kappa}^r(E_0)} |p_{ij}|$.  \nc
This concludes the proof of the theorem,  up to redefining $t_0$ and  $\kappa$ by a constant factor.
\end{proof}

\section{Mean-field reduction} 
\label{sec:meanfield}

This section proves Theorem \ref{lemma main3}. We  actually just need to prove it
when a tiny GOE  regularization is added, as explained in the next paragraph.

\subsection{Small regularization.}\ Consider  matrices of type
\be\label{11da}
H=H_1+H_2+N^{-A}H^{\rm G}\ {\rm where}\ H^{\rm G}_{ij} \overset{\rm (d)}{=}   (1+\mathds{1}_{ij})^{1/2}\cdot \cal N(0,  N^{-1}),
\ee
where $H_1$ and $H_2$ are defined in \eqref{sjui}. Our main result in this section is the following result.

  \begin{theorem}\label{thm:afterreg}
 Let $A>10$ be any fixed constant. Assume that $H$ is a band matrix of type \eqref{11da}, with band width $W_N$ satisfying (\ref{gwzN}).
 \begin{enumerate}[(i)]
 \item
The eigenvectors are delocalized as in (\ref{maindel}).
 \item  
The eigenvalues
  satisfy the local semicircle law as in (\ref{SemiC}).
\item 
Fixed energy universality holds as in (\ref{e:Univ}).
\item
For any (small)  $\tau, \kappa > 0$,  and (large) $D>0$, there exists $N_0>0$  such that for any $N\geq N_0$ we have 
\be\label{M2}
\P\left(  \left|\frac {N} W\sum_{\alpha=\ell}^{\ell+W}
|\psi_j(\alpha)|^2 \nc-1\right|<N^{-\frac{3}{2} a+\tau}\,
\mbox{for all}\ 1\le j,\ell \le N\;  \text{ such that }   \; |\lambda_j|\le 2-\kappa 
\right)\ge 1-N^{-D}, 
\ee
where  $a>0$ was given in \eqref{gwzN} and all indices are defined  modulo $N$.
\end{enumerate} 
The same results hold for all submatrices  of $H$ of type $H^{(k)}=(H_{ij})_{i,j\in\llbracket 1,N\rrbracket\backslash\{k\}}$.
\end{theorem}

The following simple lemma shows that all properties of delocalization only need to be established for the slightly regularized matrices. It is proved by perturbative arguments.

\begin{lemma}\label{pert}
Theorem \ref{thm:afterreg} implies Theorem \ref{lemma main3}.
\end{lemma}

\begin{proof}
Let $H'=H_1+H_2$ have distribution \eqref{sjui} and $H=H'+N^{-A}H^{\rm G}$, with respective ordered eigenvalues and eigenvectors $\la_k',\b\psi_k'$, $\la_k,\b\psi_k$. 
Let $\mathcal{A}=\{\|H^{\rm G}\|_\infty\leq N^{-A-1/2+\e}\}$. By Gaussian decay of the entries of $H^{\rm G}$, for any $\e,C>0$, for large enough $N$ we have 
\begin{equation}\label{eqn:proba1}
\mathbb{P}(\mathcal{A})\geq 1-N^{-C}.
\end{equation}

The conclusions $(ii)$ and $(iii)$ of Theorem \ref{lemma main3} for $H'$ therefore follow from the Hoffman-Wieland inequality:
\begin{equation}\label{eqn:HW}
\sup_k|\la_k-\la_k'|\mathds{1}_{\mathcal{A}}\leq N^{1/2} \big(\sum_k|\la_k-\la_k'|^2\big)^{1/2}\mathds{1}_{\mathcal{A}}\leq N^{1/2}({\rm Tr}(H'-H)^2)^{1/2}\mathds{1}_{\mathcal{A}}\leq N^{-A+3}.
\end{equation}

Moreover, the conclusions  $(i)$ of Theorem \ref{lemma main3}  also holds for $H'$. Indeed, we have $\eta^{-1}|\psi'_k(i)|^2 \le \im G'_{ii}(\lambda'_k+ \ii \eta)$ and  the simple inequality
$$
\|(H'-z)^{-1}\|_\infty=\|(H-z)^{-1}\|_\infty+\OO\left(\frac{N^2}{\eta^2}\|H'-H\|_\infty\right)
$$
obtained by resolvent expansion.
From the local law and eigenvector delocalization for $H$, for any $z=E+\ii\eta$, $\eta>N^{-1+\e}$, $E\in[-2+\kappa,2-\kappa]$, for any $D>0$ we have $\mathbb{P}(\|(H-z)^{-1}\|_\infty\leq N^{\e})\geq 1-N^{-D}$ for some $C>0$, for large enough $N$. Moreover, on $\mathcal{A}$ we have $\frac{N^2}{\eta^2}\|H'-H\|_\infty\leq N^{-2}$, which concludes the proof of $(i)$ for $H'$.

The proof of $(iv)$ is more involved. We want to obtain (\ref{M2}) for $H'$, for a given large $D>0$.  Take $A=4D$ in (\ref{11da}) and denote $t=N^{-A}$.
The perturbation formula for the $\b\psi_k(s)$'s, eigenvectors of $H'+s H^{\rm G}$ associated to eigenvalues $\la_k(s)$'s, is
$$
\frac{\rd}{\rd s}{\b\psi_k}(s)=\sum_{\ell\neq k}\frac{\langle \b\psi_\ell(s),H^{\rm G}\b \psi_k(s)\rangle}{\la_k(s)-\la_\ell(s)}{\b \psi}_\ell(s).
$$
On $\mathcal{A}$, we therefore have
\begin{equation}\label{eqn:per}
\|\b \psi_k-\b \psi_k'\|_\infty\leq N^{2}\int_0^{t}\rd s\left(\frac{1}{|\la_k(s)-\la_{k+1}(s)|}+\frac{1}{|\la_k(s)-\la_{k-1}(s)|}\right).
\end{equation}
Consider eigenvalues $\la_k<\la_{k+1}$ for $H$, and $\la^{(i)}_{k}\in(\la_k,\la_{k+1})$ be an eigenvalue for the minor $H^{(i)}$, with associated normalized eigenvector $\b\psi_k^{(i)}$.
Denote
$$
\mathcal{A}^{(i)}_k=\left\{\sum_{|\alpha-i|<W}|\psi_k^{(i)}(\alpha)^2|>\frac{W}{10N}\right\}.
$$
By QUE for $H^{(i)}$, for any $C>0$, for large enough $N$ we have 
\begin{equation}\label{eqn:proba2}
\mathbb{P}(\cap_{i,k:\la_k\in[-2+\kappa,2-\kappa]}\mathcal{A}^{(i)}_k)\geq 1-N^{-C}.
\end{equation}
By a Schur complement as in 
\cite[Section 4]{NguTaoVu2017}, for any $\delta>0$ we have
$$
\mathbb{P}\left(\{|\la_{k+1}-\la_k|<\delta\}\cap \mathcal{A}^{(1)}_k\cap\dots\cap\mathcal{A}^{(N)}_k\right)\leq N\mathbb{P}(\{|\langle \wt H^{(1)},\b\psi^{(1)}_k\rangle|<\delta\sqrt{N}\}\cap\mathcal{A}^{(1)}_k)
$$
where $\wt H^{(i)}=(H_{ij})_{j\neq i}$. Take $\delta=N^{-2D}$. On $\mathcal{A}^{(1)}_k$, $\langle \wt H^{(1)},\b\psi^{(1)}_k\rangle$ is a random variable with density bounded  by $N^2$, so that
$$
\mathbb{P}\left(\{|\la_{k+1}-\la_k|<\delta\}\cap \mathcal{A}^{(1)}_k\cap\dots\cap\mathcal{A}^{(N)}_k\right)\leq N^{-2D+4}.
$$
Moreover, similarly to (\ref{eqn:HW}), we have $\sup_{0\leq s\leq t}|\la_k(t)-\la_k(s)|\mathds{1}_{\mathcal{A}}\leq N^{-A+3}$, which together with the previous equation gives
\begin{equation}\label{eqn:proba3}
\mathbb{P}\left(\{|\la_{k+1}(s)-\la_k(s)|<\delta\ {\rm for\ some}\ 0<s<t\}\cap \mathcal{A}^{(1)}_k\cap\dots\cap\mathcal{A}^{(N)}_k\cap\mathcal{A}\right)\leq N^{-2D+4}+N^{-A+3}.
\end{equation}
From equations (\ref{eqn:proba1}), (\ref{eqn:per}), (\ref{eqn:proba2}) and (\ref{eqn:proba3}), for any $C>0$ we have, for large enough $N$,
$$\mathbb{P}(\|\b\psi_k-\b\psi_k'\|_\infty<N^{-A+2+2D})\geq  1-N^{-2D+4}-N^{-A+3}-N^{-C}.
$$
This concludes the proof of QUE for $\b\psi_k'$, knowing QUE for $\b\psi_k$.
\end{proof}

\subsection{Notations.}\label{MR}\  
We now explain the ideas for the proof of 
Theorem \ref{thm:afterreg}.
We start with the following  definition which generalizes band matrices by allowing  diagonal perturbations.

 \begin{definition}[Definition of $H^{\b g}_\zeta$]\label{defHg}
 For any positive constant $\zeta$ and any  $\b g\in \R^N$, $H_{\zeta}$ and $H_{\zeta}^{\b g}$ will denote real  symmetric  $N \times N$ matrices satisfying the following properties.

The matrix $H_{\zeta}$ is centered, it has  independent entries up to the symmetry condition which satisfy (\ref{cond: moment}), (\ref{eqn:subgaus}) and is of the form 
\be\label{Hgzeta}
H_\zeta  = \begin{pmatrix} A_\zeta  & B^* \cr B & D
    \end{pmatrix}, 
\ee
where  $A_\zeta$ is a $W\times W$ matrix and 
$$\var ((H_{\zeta})_{ij})=(s_{\zeta})_{ ij}=s_{ij}-\frac{\zeta (1+\delta_{ij})}{W}\mathds{1}_{ i, j\in\llbracket 1, W\rrbracket },
$$
where $s_{ij}=f(i-j)$ and $\sum_{x\in\mathbb{Z}_N} f(x)=1$.

The matrix $H^{\b g}_\zeta$ is defined by
\be\label{Hg}
\left(H^{\b g}_\zeta \right)_{ij}:= (H_\zeta)_{ij}-  g_i \delta_{ij},\quad H^{\b g}_\zeta 
   =:  \begin{pmatrix} A^{\b g}_\zeta  & B^* \cr B & D^{\b g}
    \end{pmatrix},    \quad \b g=(  g_1,   g_2, \dots,  g_N).
\ee
We denote the eigenvalues and eigenvectors of $H^{\b g}_\zeta$ by $\lambda_k^{\b g}$ and $\b\psi_k^{\b g} = \begin{pmatrix} \bw^{\b g}_k \cr \bv^{\b g}_k  \end{pmatrix}$
where $\bw^{\b g}_k\in \R^W$. In the special case $g_j = g \mathds{1}_{j > W} $,   we will  denote 
$H^{\b g}_\zeta$ by $H^ g_\zeta$, and for $\zeta=0$ we abbreviate $H^{\b g}_\zeta$ (resp. $H^g_\zeta$) by $H^{\b g}$ (resp. $H^g$).
 \end{definition}

\noindent In fact, the matrices $H^{\b g}$ we consider will always be of type  $H^ g$, up to a translation of the basis indices mod $N$.\\

We now define some curves, illustrated in Figure \ref{Fig1} from Subsection \ref{sub:sketch}. The eigenvector equation $H^{\b g} \b\psi_k^{\b g} = \lambda_k^{\b g}\b\psi_k^{\b g}$ immediately implies that 
 $$
 (A^{\b g} -B^{\b g, *} (D^{\b g} -\lambda^{\b g} _k)^{-1}B^{\b g} ) \bw_k^{\b g} =\lambda_k^{\b g}  \bw_k^{\b g} .
 $$
Hence we will consider the eigenvector equation 
\be\label{gaibian}
Q_e^{\b g} \bu_k^{\b g} (e) = \xi_k^{\b g} (e)\bu_k^{\b g} (e); \quad Q_e^{\b g} := (A^{\b g} -B^{\b g, *} (D^{\b g} -e)^{-1}B^{\b g} ),
 \ee
where  $\xi_k^{\b g} (e)$, $\bu_k^{\b g} (e)$ are  eigenvalues and normalized eigenvectors. 
From now on, we assume that $k$ is an index in the bulk of the spectrum for $H^{\b g}$,   i.e. for some $\kappa>0 , \kappa N<k<(1-\kappa)N  $.

Since  the matrix elements 
have Gaussian components \eqref{11da}, it is easy to check that the eigenvalue flows  $g\to \lambda_k^g$ are smooth and non-intersecting with probability one.
Assuming that 
the function $g\to e= \lambda_k^g+g$ has a regular inverse (for the existence of such an inverse, see Subsection \ref{Mfr}),
for any $e$ close enough to $\la_k$, there exists a $g$ such that $e= \lambda_k^g+g$, so that we can define
$$\cal C_k(e) = \lambda^g_k.$$
The curves 
$(\cal C_k(e))_{1\leq k\leq N}$are labeled  in increasing order by  their intersections with the diagonal  $\cal C (e) = e$. 
We refer to \cite[Equation (4.16)]{BouErdYauYin2017} for a detailed discussion of the domain of $\mathcal{C}_k$.

We defined  $\xi_i(e)$ ($1 \le i \le W$), the  eigenvalue of $Q_e= Q^{\b g = 0}_e$.  A simulation of the  curves $e \to \xi (e)$ in given in Figure \ref{Fig1}.
Since $\xi_i(e)$ is also an eigenvalue of $H^{e-\xi_i(e)}$,   it is equal to $\cal C_j(e)$ for some $j$.  We follow the convention in \cite{BouErdYauYin2017} 
to denote  $k' \in\llbracket 1,W\rrbracket $ to be the index given by the relation  $
\xi_{k'}({ e}\, )=\cal C_k({e}\,)$. Here $k'= k'(e)$ depends on the energy $e$  and  $\xi_{k'(e)}(e)$ is increasing in  $k$.   
As $e$ approaches to an eigenvalue of $D$, one eigenvalue from $\b\xi(e)$  tends to $\pm\infty$. 
The other eigenvalues follow the smooth curves $\cal C_k$  
and the labels $k'(e)$ gets shifted by $\pm 1$  whenever $e$ crosses an eigenvalue of $D$.   
Since  the curve  $\cal C _k$ passes through  $(\lambda_k, \lambda_k)$, we have 
\be\label{tanxi}
H\b\psi_k=\lambda_k \b\psi_k, \quad \xi_{k'}(\lambda_k)=\lambda_k,\quad 
\b\psi_k=\begin{pmatrix} \bw_k \cr \bv_k\end{pmatrix}, \quad
 Q_{\lambda_k}\bu_{k'}=\xi_{k'}\bu_{k'}, \quad \bu_{k'}(\lambda_k)=\frac{\bw_k}{\|\bw_k\|_2}.
\ee

\subsection{Outline of proof of Theorem \ref{thm:afterreg}.}\label{sub:sketch}\  
We explain the main steps of the proof, with successively QUE for mean field blocks, QUE for  $H$ from (\ref{11da}), and its application to local law, universality and delocalization.\\

\noindent{\it First step: QUE for mean-field blocks $Q_e^{\b g}$.}\ Remember the definition of $H$ from (\ref{11da}) and denote
$
\wt H=(1+N^{-A}\frac{N+1}{N})^{-1/2}H
$.
Consider a parameter $\zeta=T=N^{-c}$ where $c>\e_m$ is defined in $\eqref{mouyzz}$. Then, thanks to the Gaussian matrix $H_2$ defining $H$,  we can write 
$$
\wt H=H_T+\sqrt{T}\begin{pmatrix} H_{W}^{\rm G}  & 0 \cr 0 & 0
    \end{pmatrix}
$$
for some $H_T$ as defined in (\ref{Hgzeta}), and $H_{W}^{\rm G}$ is a $W\times W$ GOE matrix.
To this $H_T$ we associate $H_T^{\b g}$ from formula (\ref{Hg}), and denote
$
V=A_{T}^{\b g}-B^{\b g, *} (D^{\b g} -e)^{-1}B^{\b g}.
$
Consider the flow 
\begin{equation}\label{eqn:deform}
K_{T}^{\b g, e}(t)=V+Z(t)
\end{equation}
 as in (\ref{eqn:Kt}).
 Notice that 
we have the equality in distribution
\be\label{deform1}
K_{T}^{\b g, e} (t) \overset{\rm (d)}{=}   K_{T-t}^{\b g, e}(0) = A_{T-t}^{\b g}-B^{\b g, *} (D^{\b g} -e)^{-1}B^{\b g}.
\ee
In particular, the distributions of $Q_e^{\b g}= A^{\b g} -B^{\b g, *} (D^{\b g} -e)^{-1}B^{\b g} $ from  \eqref{gaibian} is the same as  $K_{T}^{ \b g, e}(T)$.

We therefore obtain QUE for the mean field blocks $Q_e^{\b g}$ by using Theorem \ref{thmQUE}, i.e. by interpreting this matrix ensemble as the result of the flow 
$K(T)=K_{T}^{ \b g, e}(T)$. As an hypothesis for Theorem \ref{thmQUE}, some estimates on $V=K(0)$ are necessary, and given in Subsection \ref{sec: Str}.\\

\noindent{\it Second step: QUE for $H^{\b g}$.}\ 
To simplify the notations we set $\b g=0$, but QUE will be obtained similarly for any small enough $\b g$. 
 For the proof, we combine  an   $\e$-net argument with perturbations of eigenvectors.

For this, we first need to choose good points for our net.  
Let $M =N^{C}$ with $C$ a large constant which will be chosen in the rigorous proof.  
We will prove that there is another large number $C'$ such that for each $n \in \Z$ fixed such that $E_n=nN^{-C'}\in [-2+\kappa, 2-\kappa]$, then  there is a deterministic $e_n\in[E_n,E_{n+1}]$  (i.e., 
the choice of $e_n$ may depend on the law of $D$ but is independent of the random matrix elements of $D$)
$$
\inf_j  |\lambda^D_j - e_n| \ge M^{-1}   
$$
with high probability, 
where the $\lambda^D_j$'s are eigenvalues of $D$ (recall  that  $\lambda_j$ denotes  an eigenvalue of $H$).
In other words, the bulk eigenvalues of $D$ will stay away from the grid points $(e_n)_{n\in\mathbb{Z}}$ by at least $N^{-C}$: the norm of $Q_{e_n}$ is polynomially bounded,
an hypothesis necessary to prove QUE by flow methods.

 We now consider QUE for these good points $(e_n)_{n\in\mathbb{Z}}$. Let $J$ be the $W\times W$ matrix defined by  
\be\label{defJ}
(J)_{ij}=\delta_{ij}\cdot \mathds{1}_{1\le i\le W/2}.
\ee
 By QUE the mean-field blocks (see Lemma \ref{henna}),  for all $n$ and $l$  satisfying  $|\xi_l(e_n)-e_n|\le W^{-1}$ we have  
\be\label{gozet-}
\left| \left\|J\cdot \bu_l(e_n)\right\|^2_2-\frac12\right|\le   \frac{N^{1/2+\tau}}{W}+\frac{N^{\frac{\e_m}{2}+\tau}}{W^{1/2}}
 \ee
with overwhelming probability, where $\tau> 0$ an arbitrarily small positive constant and  $\e_m$ is defined in \eqref{cond: moment}.

For a given bulk index $k$, let  $\wt e=\sup_{n}\{e_n:\, e_n  <  \lambda_k \} $. Recall that  $\cal C_k(\lambda_k)=\lambda_k$ and 
 $k'\in\llbracket 1,W\rrbracket $ is the index given by the relation  $\xi_{k'}(e )=\cal C_k(e)$ for all $e$, as explained in Subsection \ref{MR}. 
By  the  eigenvector perturbation formula for the matrix $Q_e$, we have 
\be\label{eqn:perbevect}
  \frac{\rd }{\rd e}  {\b u}_{k'}(e) 
 =     \sum_{\ell \ne k}\frac{  {\b u}_{\ell'} (e)}{\cal C_{k}(e)-\cal C_\ell (e)}
\left({\b u}_{\ell'}  (e),  B^* \frac{1}{(D-e)^2}B \;  {\b u}_{k'}(e)\right).
\ee
Notice that we used the labeling associated with the curve $\cal C$ since  $\cal C_k(e)$ is  continuous, i.e.,   the label $k, \ell$ does not change 
as $e$ pass through the eigenvalues of $D$. However, the label $k'$ for the eigenvector depends on $e$. 

Our goal is to  approximate $\bu_{k'}(\lambda_k)$, which is proportional to 
the first $W$ components of the eigenvector $\b\psi_k$ of $H$, by the eigenvector   $\bu_{k'}({\wt e}\,)$ which satisfies the QUE  
by \eqref{gozet-}.
Integration gives
\be\label{zapgys-}
\| \bu_{k'}({\wt e}\,)- \bu_{k'}(\lambda_k) \| \le N^{-C'} \sup_{\lambda_k \le e \le {\wt e}}
 \sum_{\ell \ne k}\frac{ 1}{|\cal C_{k}(e)-\cal C_\ell (e)|}
\Big | \left({\b u}_{\ell'}  (e),  B^* \frac{1}{(D-e)^2}B \;  {\b u}_{k'}(e)\right) \Big |.
\ee 
We will show that  for some $C_1 > 0$ \nc the following two  estimates  hold with high probability.
\begin{enumerate}
\item  Level repulsion: for fixed $k$  we have
$$
 \min_{ e\in[\,{\wt e}, \lambda_k]}\; \min_{{  \ell:} \ell \ne  k}\; |\cal C_{ k}(e)-\cal C_{\ell}(e)|\ge  N^{- C_1/2}. 
$$
\item A consequence of the weak uncertainty principle from Section \ref{sec: reg},
$$
\sup_{\lambda_k \le e \le {\wt e}}
\Big | \left({\b u}_{\ell'}  (e),  B^* \frac{1}{(D-e)^2}B \;  {\b u}_{k'}(e)\right) \Big |\le   N^{C_1/2}.
$$
\end{enumerate}
If these two bounds hold then \eqref{zapgys-} gives stability of the eigenvector under perturbation in $e$, provided that $C' \gg C_1$. 
Delocalization and QUE  of $\bu_{k'}({\wt e})$ therefore imply the same properties for  $\bu_{k'}(\lambda_k)$.

Thus, denoting $\e_N$ the right hand side of (\ref{gozet-}) and $X_n=\sum_{1 \le i \le W/2} |\psi_k(i + n W/2)|^2$,  we have 
\begin{equation}\label{eqn:ratio}
\frac{X_1}{X_{2}}=1+\OO\left(\e_N\right)
\end{equation}
with overwhelming probability.
Now we can shift the indices by $W/2$ and repeat the same argument, so that for any $1\leq \ell<m\leq 2N/W$, we have
$$ 
\frac{X_{\ell}}{X_{m}}=\left(1+\OO(\e_N)\right)^{m-\ell}=1+\OO\left(\frac{N}{W}\e_N\right).
$$
provided that $\frac{N}{W}\e_N=\oo(1)$.
Summing over $\ell$ for fixed $m$ gives, with overwhelming probability, 
\begin{equation}\label{eqn:QUEsketch}
\frac{N}{W}X_m=\frac{1}{2}+\OO\left(\frac{N}{W}\e_N\right).
\end{equation}
This concludes the outline that QUE  
for the eigenvector $\b\psi_k$ holds, when $\frac{N}{W}\e_N=\oo(1)$.\\

\noindent{\it Third step: applications of QUE.}\ 
We successively outline the proofs of delocalization, universality and local law for $H$ from (\ref{11da}).

Delocalization for the mean field blocks $Q_e^{\b g}$ holds thanks to a priori resolvent estimates from  subsection \ref{sec: Str}, and regularization of the resolvent by Dyson Brownian motion, as in (\ref{evcontrol}).
By stability as in (\ref{zapgys-}), this delocalization is extended to $\bu_{k'}(\la_k)$. As $((u_{k'}(\la_k))(i))_{1\leq i\leq W}=(\psi_k(i))_{1\leq i\leq W}/\|\b\psi_k\|_{{\rm L}^2(\llbracket 1,W\rrbracket)}$, 
delocalization for $\b\psi_k$ follows from both delocalization of $\bu_{k'}(\la_k)$ to the QUE estimate (\ref{eqn:QUEsketch}) about $\|\b\psi_k\|_{{\rm L}^2(\llbracket 1,W\rrbracket)}$.

For universality, remember that for any  $e$,   $\mathcal{C}_k(e)=\xi_{k'} (e)$ denotes  the eigenvalues of $Q_e$, and that
the intersection points of the  curves $e \to \xi_k(e)$ with the diagonal $e= \xi$ 
are  eigenvalues for $H$  (see Figure \ref{Fig1}).
Thus $\lambda_j$  can be determined by the spectrum $\b\xi(e)$ for a fixed $e$, 
and  the slope of the curves $e \to \xi_{k'}(e)$. 
On the one hand, $\b\xi(e)$  follows GOE statistics as a consequence of \cite{LanSosYau2016}.
On the other hand, a simple calculation yields 
$$
\partial_e  \mathcal{C}_k(e)=1-\frac{1}{\sum_{i=1}^W \left| \psi_k^g(i) \right|^2},
$$
where $\b\psi^g$ is the corresponding eigenvector of $ H^g$ with $g$ the solution to $e=\la_k^g+g$.
From the  QUE  \eqref{eqn:QUEsketch} for $H^g$, all slopes are equal at leading order, so that the statistics of 
$\lambda_j$ will be given by those of $\xi_k$ up to some  trivial scaling.  
In the same way, the local law for $H$ follows from a local law for $Q_e$ by parallel projection.

\vspace{0.5cm}

\begin{figure}[h]
\centering
\vspace{0.2cm}
\begin{subfigure}{.4\textwidth}
  \centering
\begin{tikzpicture}
\node[anchor=south west,inner sep=0] (x) at (0,0) {\includegraphics[width=7cm]{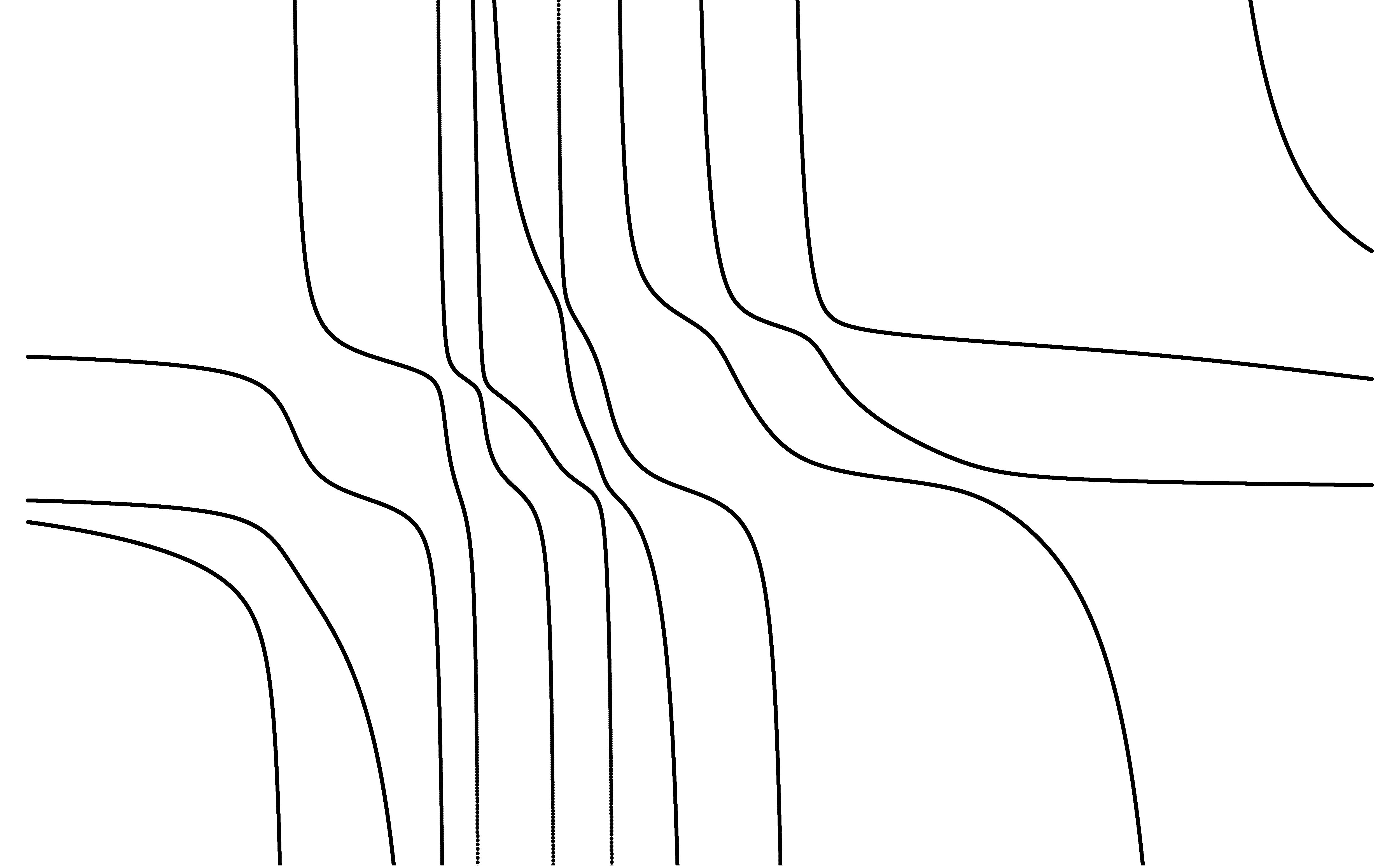}};
\begin{scope}[x={(x.south east)},y={(x.north west)}]
\draw[black,thick,rounded corners] (0.5,0.55) rectangle (0.64,0.75);
\draw [black,->] (0,0.5) -- (1.02,0.5);
\draw [black,->] (0.47,0) -- (0.47,1);
\draw [black,dashed] (0.47,0.5) -- (0.81,1);
\draw [black,dashed] (0.47,0.5) -- (0.13,0);
\fill[black] (0.198,0.1)  circle[black,radius=1.5pt];
\fill[black] (0.26,0.19)  circle[black,radius=1.5pt];
\fill[black] (0.313,0.265)  circle[black,radius=1.5pt];
\fill[black] (0.339,0.305)  circle[black,radius=1.5pt];
\fill[black] (0.388,0.378)  circle[black,radius=1.5pt];
\fill[black] (0.423,0.426)  circle[black,radius=1.5pt];
\fill[black] (0.435,0.445)  circle[black,radius=1.5pt];
\fill[black] (0.452,0.468)  circle[black,radius=1.5pt];
\fill[black] (0.52,0.58)  circle[black,radius=1.5pt];
\fill[black] (0.555,0.622)  circle[black,radius=1.5pt];
\fill[black] (0.585,0.665)  circle[black,radius=1.5pt];
\fill[black] (0.198,0.1)  circle[black,radius=1pt];
\node[black] at (1.05,0.5) {$e$};
\end{scope}
\end{tikzpicture}
  \caption{A simulation of  eigenvalues of $Q_e=A-B^*(D-e)^{-1}B$, i.e. functions $e\mapsto \xi_j(e)$. Here $N=12$ and $W=3$. The $\lambda_i$'s are the abscissa of the intersections with the diagonal.}
\end{subfigure}
\begin{subfigure}{.1\textwidth}
  \centering
\begin{tikzpicture}
\draw [black,dashed] (0,0) -- (0,0);
\end{tikzpicture}
\end{subfigure}
\begin{subfigure}{.4\textwidth}
\centering
\vspace{-1.1cm} \hspace{-2cm}
\begin{tikzpicture}
\node[anchor=south west,inner sep=0] (y) at (0,0) {\includegraphics[width=8.5cm,height=5cm,trim={3cm 3cm 1.5cm 1.5cm},clip]{C.jpg}};
\begin{scope}[x={(y.south east)},y={(y.north west)}]
\draw[white,ultra thick,fill=white]  (0,0) rectangle (1,0.8);
\draw [black,-,thick,dashed] (0.21,0.16) -- (0.65,0.8);
\draw [black,-,thick,dashed] (0.37,0.16) -- (0.37,0.8);
\draw [black,-,thick] (0.1,0.7) -- (0.8,0.3);
\draw [black,-,thick] (0.1,0.8) -- (0.8,0.4);
\draw [black,-,thick] (0.1,0.93) -- (0.8,0.53);
\draw [black,-,thick] (0.1,1) -- (0.8,0.6);
\draw [black,-,thick] (0.1,0.64) -- (0.8,0.24);
\draw [black,-,thick] (0.1,0.5) -- (0.8,0.1);
\draw [black,-,thick] (0.1,0.4) -- (0.8,0.02);
\draw [black,-,thick] (0.1,0.32) -- (0.63,0.04);
\fill[black] (0.26,0.235)  circle[black,radius=1.8pt];
\fill[black] (0.3,0.29)  circle[black,radius=1.8pt];
\fill[black] (0.35,0.36)  circle[black,radius=1.8pt];
\fill[black] (0.415,0.46)  circle[black,radius=1.8pt];
\fill[black] (0.445,0.5)  circle[black,radius=1.8pt];
\fill[black] (0.495,0.573)  circle[black,radius=1.8pt];
\fill[black] (0.56,0.665)  circle[black,radius=1.8pt];
\fill[black] (0.595,0.715)  circle[black,radius=1.8pt];
\draw[black,fill=black] (0.37,0.775)  +(-1.3pt,-1.3pt) rectangle +(1.3pt,1.3pt) ;
\draw[black,fill=black] (0.37,0.645)  +(-1.3pt,-1.3pt) rectangle +(1.3pt,1.3pt) ;
\draw[black,fill=black] (0.37,0.545)  +(-1.3pt,-1.3pt) rectangle +(1.3pt,1.3pt) ;
\draw[black,fill=black] (0.37,0.485)  +(-1.3pt,-1.3pt) rectangle +(1.3pt,1.3pt) ;
\draw[black,fill=black] (0.37,0.345)  +(-1.3pt,-1.3pt) rectangle +(1.3pt,1.3pt) ;
\draw[black,fill=black] (0.37,0.255)  +(-1.3pt,-1.3pt) rectangle +(1.3pt,1.3pt) ;
\draw[black,fill=black] (0.37,0.18)  +(-1.3pt,-1.3pt) rectangle +(1.3pt,1.3pt) ;
\draw [black,->](0.37,0.845)-- (0.495,0.774);
\filldraw[white,fill=white]
(0,0) -- (0,1) -- (1,1) -- (1,0) -- cycle
(0.2,0.15) -- (0.65,0.15) -- (0.65,0.8) -- (0.2,0.8) -- cycle;
\draw[black,thick,rounded corners] (0.2,0.15) rectangle (0.65,0.8);
\draw [black,->,thick] (0.37,0.05) -- (0.37,0.13);
\node[black] at (0.37,0.02) {$e$};

 \draw [black,->,thick] (0.31,0.05) -- (0.31,0.13);
\node[black] at (0.31,0.02) {$\lambda'$};

\draw [black,->,thick] (0.26,0.05) -- (0.26,0.13);
\node[black] at (0.26,0.02) {$\lambda$};

\draw [black,->](0.37,0.775)-- (0.475,0.715);
\draw [black,->] (0.37,0.645) -- (0.44,0.605);
\draw [black,->] (0.37,0.545)-- (0.415,0.52);
\draw [black,->] (0.37,0.49)-- (0.398,0.468);
\draw [black,->] (0.37,0.255) -- (0.33,0.274);
\draw [black,->] (0.37,0.18) -- (0.31,0.209);
\end{scope}
\end{tikzpicture}
  \caption{Framed region of Figure (a) for large $N,W$: the curves $\xi_j$ are almost parallel, with slope $1-N/W$.  The eigenvalues of  $Q_e$ and  $H$ are related by a
  projection to the diagonal.
 }
\end{subfigure}

\caption{The idea of mean-field reduction, from \cite{BouErdYauYin2017}: universality of gaps between eigenvalues for fixed $e$ implies universality on the diagonal through parallel projection.
For $e$ fixed, we label the curves by $\xi_k(e)$.  }
\label{Fig1}
\end{figure}

\subsection{Generalized resolvent estimates.}\label{sec: Str}\ In this subsection, we do not need to assume (\ref{cond: moment}).

Recall that we have added a GOE regularization of size $N^{-A}$ in \eqref{11da}. Since $N^{-A}$ is  tiny, all resolvent  estimates cited in this paper for matrices are valid  after adding  this small  regularizing GOE.   A formal proof can be obtained  by 
the standard resolvent identity $
(B-C)^{-1}=B^{-1}+B^{-1}CB^{-1}+\cdots$, which we will skip.  
In this section, all results will  be proved without this regularization so as to  simplify the notations.

Our first goal is to show that $K_{T, t}^{\b g, e}$ \eqref{deform1}   is $(\eta_*, \eta^*,  r)$-regular  at $e=E_0$, in the sense of Assumption \ref{XYH},
for some range of $t$. The precise choice of the parameters $r, T, \eta_*, \eta^*$ will be given in \eqref{mouyzz}.
Recall the matrix  $H^{\b g}_{T, t} $ is defined by 
$$
H^{\b g}_{T, t} = \left(\begin{array}{cc}   A_T^{\b g}+Z_t    & B^* \\
B  & D^{\b g}
\end{array}\right).
$$
As in (\ref{eqn:general}), define the ``generalized  resolvent" of $H^{\b g}_{T, t}$  by  
$$
G^{\b g}_{T, t}(z, e )=  
\left(H^{ \b g}_{T, t}- \begin{pmatrix}  z {\rm I}_{W}  & 0 \cr 0 & e {\rm I}_{N- W}  \end{pmatrix} \right)^{-1}.
$$
The distribution of $H^{\b g}_{T, t}$ is the same as $H^{\b g}_{T-t}$ defined in \eqref{Hg}, 
so we will also denote  $G^{\b g}_{T, t}(z, e )$  by $G^{\b g}_{T-t}(z, e )$.

Clearly, the $W\times W$ component of  $ G^{\b g}_{T, t}(z, e )$ is  exactly the resolvent $ ( K_{T-t}^{ \b g, e}-z)^{-1}$.
We will state estimates on this generalized resolvent in Theorem \ref{LLniu}, an important input for our mean-field reduction method. The 
proof appears in the companion papers \cite{BouYanYauYin2018,YanYin2018}. 
On the one hand, the absence of imaginary part on most of the diagonal elements 
of the generalized resolvent  is a major obstacle to estimate it.  On the other hand, Theorem \ref{LLniu} assumes   $\eta={\rm Im}\, z$ is large (almost of order $1$), which is a sufficient 
input to apply Theorem \ref{thmQUE} and obtain quantum unique ergodicity.

Define $M^{\zeta, \b g}_i(  z, \wt z)$ as the solution of the self consistent equation
$$
(M^{\zeta, \b g}_i)^{-1}(  z, \wt z)=-( \wt z -z)\mathds{1}_{i>W}-z-{g}_i- \sum_{j} (s_{\zeta})_{ ij}M  _j ^{\zeta, \b g}(z,\wt z ), \quad z,\; \wt z\in \C^+\cup \R
$$
with the constraint that $$M^{0,\b 0}_i( \wt z, \wt z\,)=m_{\rm sc} (\wt z+ \ii 0^+ ),$$  the Stieltjes transform of the semicircle law. 
For simplicity of notations, we denote by  $M^{\zeta, \b g}( z, \wt z)$ the matrix with entries 
$$M^{\zeta, \b g}_{ij}: =M^{\zeta, \b g}_i\delta_{ij}.$$ We will show that $M^{\zeta, \b g}( z, \wt z)$  is the limit of the generalized resolvent 
$G_{\zeta}^{\b g}(  z,\wt z\, )$. For this purpose, we first collect  basic  properties of  $M$ in the following lemma, which is proved 
in \cite{BouYanYauYin2018}.

\begin{lemma}\label{UE}
 Assume $|\re \wt z\, |\le 2-\kappa$  for some $\kappa>0$. There exist $c, C>0$ such that the following holds.
 \begin{enumerate}[(i)]
\item  (Existence and Lipschitz  continuity) If  
\be\label{heiz}
 \zeta+\| \b g\|_\infty+ |z-\wt z|\le c, 
\ee
   then  $M^{\zeta, \b g}_i(  z, \wt z)$  exists and 
$$
 \max_i\left|M^{\zeta, \b g}_i(  z, \wt z)-m_{\rm sc}(\wt z+  \ii 0^+)\right|\le C \left(\zeta+\| \b g\|_\infty+ |z-\wt z|\,\right).
$$
If, in addition, we assume $\zeta'+\| \b g'\|_\infty+ |z'-\wt z|\le c $, then 
$$
\max_i\left|M^{\zeta', \b g'}_i(  z', \wt z)-M^{\zeta, \b g}_i(  z, \wt z)\right|\le C\left(\| \b g-\b g'\|_\infty +|z'-z|+|\zeta'-\zeta|\right).
$$
\item (Uniqueness)  The vector $M^{\zeta, \b g}_i(  z, \wt z)$, $(1\le i\le N)$ is unique under the constraints  \eqref{heiz}  and  
$$
  \max_i\left|M^{\zeta, \b g}_i(  z, \wt z)-m_{\rm sc}(\wt z+ \ii  0^+ )\right|\le c.
$$
\end{enumerate}
\end{lemma}

Since $s_{ij}$ from \eqref{bandcw0} is a periodic function of $i-j$,  by the uniqueness of the previous lemma, we have
$
 M_i^{\zeta, \b 0}(z, \wt z)= M_{W-i}^{\zeta, \b 0}(z, \wt z), 
$
so that
\be\label{aaaa}
\sum_{i=1}^{W/2}M_i^{\zeta, \b 0}(z, \wt z)=\frac12\sum_{i=1}^{W}M_i^{\zeta, \b 0}(z, \wt z).
\ee
This equation will be necessary to obtain the quantum ergodicity estimate for $Q^{\b g}_e$ in (\ref{Hop3}).   
Our main results on the generalized  resolvent of $H_{\zeta}^{\b g}$ is the following, proved in a companion paper.

 \begin{theorem}[Generalized resolvent estimate]\label{LLniu}
 Recall $\eta_*, \eta^*$ and $r$ from Assumptions \ref{XYH} and \ref{as:basic}.  Suppose these parameters  are of the form 
\be\label{mouyzz} 
\eta_*=N^{-\e_*},\quad \eta^*=N^{-\e^*} \quad r=N^{- \e_*+3\e^*},\quad  T=N^{- \e_*+\e^* }, \quad  0<\e^*\le \e_* /20,
\ee 
where $T$ is a new parameter  used in the following equation (\ref{mouyzz2}).  Assume that
\begin{equation}\label{eqn:constraint}
\log _N W>  \max\left(\frac{3}{4}+\e^*,\frac{1}{2}+\e_*+\e^*\right).
\end{equation}
For any small $\tau,
\kappa>0$ and large $D$, uniformly in $|e|<2-\kappa$, for large enough $N$ the following holds.
For any  deterministic $z$, $\zeta$ and $\b g$ satisfying 
  \be\label{mouyzz2} 
   |\re z-e|\le r, \quad \eta_* \le  \im z\le \eta^* , \quad 0\le \zeta\le T, \quad \|\b g\|_\infty\le W^{-3/4},
\ee
we have (we denote $\|A\|_{\max}=\max_{i,j}|A_{ij}|$)
\be\label{jxw}
\P\left( \|{ G^{ \b g}_\zeta(z, e)-M ^{\zeta, \b g }( z, e)}\| _{\max}\ge   N^\tau \left(\frac{N^{1/2}}{W}+\frac{1}{\sqrt{W\, {\rm Im}z}}\right)\nc\right)\le N^{-D}.
 \ee 
 \end{theorem}

The following corollary is an immediate consequence of the above generalized resolvent estimate, the deterministic  Lemma \ref{UE}, and (\ref{aaaa}). In the statement, we use the notation $\cal I_{e, r}=(e-r,e+r)$ as in (\ref{Ir}).

\begin{corollary}
\label{lemLQ}  
We  follow the  assumptions and conventions  of Theorem \ref{LLniu}. 
Then for any $z = E + \ii \eta$ with $E \in  \cal I_{e,r}$ and  $\eta_*\leq \eta\leq \eta^*$,  any $t$ satisfying $0\le t\le T$ and any fixed (large)  number  $D>0$ the following statements hold for $N$ large enough:
\begin{align}
&\label{Hop4}
 \P\left(\exists E\in \cal I_{e,r},\; \left|\im \big(K_{T-t}^{\b g, e}-z \big)^{-1}_{kk}\right|\ge \frac{2}{W}\im \tr \big(K_{T-t}^{\b g, e}-z\big)^{-1}  \right)\le W^{-D},\\
&\label{jiangjiulh}
 \P\left(  \left|  \frac1W\tr\big(K_{T-t}^{\b g, e} -z \big)^{-1} - m_{\rm sc}(z) \right|\ge N^{-\e^*/2}\right)\le W^{-D},\\
&
\label{Hop3}
 \P\left( \max_{E\in \cal I_{e,r}}\left|  \frac{1}{W}\nc\sum_{1\le k\le W/2}\big( K_{T-t}^{\b g, e} -z \big)^{-1}_{kk}-\frac{1}{2  W\nc}\tr \big(K_{T- t}^{\b g, e}-z\big)^{-1}\right|\ge  N^\tau \left(\frac{N^{1/2}}{W}+\frac{1}{\sqrt{W\, {\rm Im}z}}\right)\nc\right)\le W^{-D}.
\end{align}
In particular, $K_{T-t}^{\b g, e}$  satisfies the regularity assumptions  (\ref{e:imasup}), (\ref{eqn:diaginitial}), \ref{Hop1prime})
in the  range $0\le t\le T$.
 \end{corollary}

\begin{remark} \label{64} 
This corollary gives control of the error in QUE for mean field blocks and therefore controls the range of  $W$ for which delocalization can be proved.

  More precisely, assume $\e_*=0$ to simplify. The error $N^{-\frak c}$ in  Assumption \ref{as:basic}, which governs the error in  Theorem \ref{thmQUE},  
is of order  $N^{-\frak c} \sim  \frac{ N^{1/2}}{W}$, form (\ref{Hop3}).
In order to patch this estimate to get QUE for the band matrix $H$, we will need $\frac{N}{W}\cdot \frac{N^{1/2}}{W}   \ll 1$. This explains our condition 
$ W \gg N^{3/4}$.  

However,  the  error $\frac{\sqrt N}{W}$ in \eqref{Hop3} 
is taken  from \eqref{jxw}; this error in \eqref{Hop3}   usually can be improved by taking into account the average of the index $k$. \nc We believe that the key error term in Theorem \ref{thmQUE} comes from the last term in \eqref{Ice1}. If we take $t_0$ close to 1 and replace $n$ by $W$, this error is of order
 $W^{-1/2}$. We therefore expect that for
$\frac{N}{W}\cdot \frac 1 {W^{1/2}}   \ll 1$, i.e. $ W \gg N^{2/3}$,  the QUE for band matrices holds. 
If  we additionally assume that these errors associated to different blocks are centered and asymptotically independent, then the total error for the QUE of  the band matrix $H$ would be 
$(\frac{N}{W})^{1/2}\cdot \frac 1 {W^{1/2}}$, which is much smaller than $1$ when $W \gg N^{1/2}$. 
\end{remark}

 \subsection{Eigenvectors and eigenvalues estimates for mean-field blocks.}\ \label{sec:QUEmean}
The  following lemma concerning  the QUE and related properties of 
the $W\times W$ matrix $Q^{\b g}_e$  from \eqref{gaibian}. It is an important  building block for the proof of Theorem \ref{thm:afterreg}.

For the statement,  recall the notations from Subsection \ref{MR}. In particular, the matrix $Q^{\b g}_e$  and its eigenvalues and eigenvectors  $\xi^{\b g}_k(e)$ and $\bu^{\b g}_k(e)$
are defined in \eqref{gaibian}. 
  
 \begin{lemma}\label{henna} Let $H$ satisfy the assumptions  in Theorem \ref{thm:afterreg} and $\kappa,\tau>0$ be small constants. 
For  $e\in \R, 1 \le k \le W$ and $C> 0$,  
denote by $\chi$ the set   
\be\label{71}
\chi(k,e,   C,\b g):= 
\left \{ |D^{\b g}-e|\ge N^{-C} \right\} \cap \left\{ |\xi^{\b g}_k(e)-e|\le  W^{-1 } \right\}.
 \ee
Uniformly in deterministic $ |e|\le 2-\kappa$,    the following statements hold. 
 \begin{enumerate}[(i)]
\item (Delocalization)  For any  $C, D>0 $,  for $N$ large enough, we have
\be\label{sza}
\max_{\|\b g\|_\infty\le W^{-3/4}}\P\left( \{   \|\bu^{\b g}_k(e)\|^2_\infty\ge W^{-1+\tau} \}\cap  \chi(k,e,  C, \b g)  \right)\le N^{-D}.
 \ee
 \item (Level repulsion) 
For any  $C, D>0 $,  there exists $N_0=N_0(C,D)$ such that for $N\geq N_0(C,D)$, for any $x>0$ we have
\be\label{sza3}
\max_k \max_{\|\b g\|_\infty\le W^{-3/4}}\P\left(
  \big \{ \left|\xi^{\b g}_{k\pm 1}(e)-\xi^{\b g}_{k }(e) 
  \right| \le \frac xW   \big \} \cap \chi(k,e,  C, \b g)  \right)\le   W^\tau x^{2-\tau} +N^{-D}.
 \ee 
\item (QUE) Recall that $\e_m$ is defined in \eqref{gwzN} and $J$ in \eqref{defJ}. For any  $C, D>0 $,  for $N$ large enough,
\be\label{sza2}
\max_{\|\b g\|_\infty\le W^{-3/4}}\P\left(   \left  \{  \Big |  \|J\cdot \bu^{\b g}_k(e)  \|^2_2-\frac12   \Big |\ge  \frac{N^{1/2+\tau}}{W}+\frac{N^{\frac{\e_m}{2}+\tau}}{W^{1/2}}\nc  
  \right  \} \cap \chi(k,e,  C, \b g)  \right) \le N^{-D} .
 \ee
\item(Local law)  Take $\b g=0$. There exists $\e >0$ which does not depend on $\tau$ such that for any  $C,D>0$, for  sufficiently large $N$ we have 
\be\label{zyaa}
\P\left(\left\{\sup_{  0 \le e'-e \le  W^{  -1+\e } }\left|\#\Big\{k: \xi_k(e)\in [e, e']\Big\}-W\int_{e}^{e'} \rho_{\rm sc}(x)\rd x\right|\ge  W^{\tau}  \right\}\cap \{|D-e|\geq N^{-C}\} \right)\le N^{-D}.
\ee
Notice that  for  $\tau < \e$, we have $W^{\tau} < W\int_{e}^{e+ W^{-1+ \e} } \rho_{\rm sc}(x) \rd x$.  
\end{enumerate} 
\end{lemma}

\begin{remark}\label{large constant}
The constraint  $|\xi^{\b g}_k(e)-e|\le  W^{-1 } $ in  \eqref{71} can be replaced by $|\xi^{\b g}_k(e)-e|\le  W^{-1 + \e }$ for some $\e>0$, with little change in the proof.  
In the application of this lemma in our paper, we only need to use $A W^{-1}$ for any large fixed constant $A$. 
\end{remark}

\begin{proof} 
Recall the operator $K^{\b g, e}_{T, t}$ in \eqref{deform1} and   $\eta_*$, $\eta^*$,  $r$ and $T$  in \eqref{mouyzz}.
Denote the eigenvalues and eigenvectors of $K^{\b g, e}_{T}(t)$  by  $\lambda_{k}^T(t)$  and $\bu_{k}^T(t)$.
Hence   the distributions of the eigenvalue $\xi^{\b g}_k(e)$ and eigenvector $\bu^{\b g}_k(e)$ of $Q^{\b g}_e$ are given by 
$$
\xi^{\b g}_k(e)\overset{\rm (d)}{=}  \lambda_{k, T}(t), \quad \bu^{\b g}_k(e)\overset{\rm (d)}{=} {\bu_{k,T}(t)}.
$$
By definition of $K^{\b g, e}_{T}(t)$,  it is trivial to prove that for any $C_0>0$ there exists $C_1$ such that 
\be\label{13}
  \mathds{1}_{|D^{\b g}-e|\ge N^{-C_0}} \| K^{\b g, e}_{T}(t) \|\le W^{C_1}
\ee
 holds with a very high probability for any $0 \le t \le T$.  
 Corollary \ref{lemLQ} and \eqref{13}  imply that $K^{\b g, e}_{T}(t)$ is $(\eta_*, \eta^*,  r)$-regular  (Assumption \ref{XYH})  at $E_0=e$
for any   $0 \le t \le T$ (under the condition 
$\mathds{1}_{|D^{\b g}-e|\ge N^{-C_0 }}$). In addition, the conditions in Assumption \ref{as:basic}  are guaranteed by \eqref{Hop4} and \eqref{Hop3}. 
By Theorem  \ref{thmQUE}, for any small $\e>0$,  with overwhelming probability we have
$$
\Big |  \|J\cdot \bu^{\b g}_k(e)  \|^2_2-\frac12   \Big |  \le 
W^\e \left( \frac{N^{1/2}}{W}\nc 
+ (W N^{- \e_*+\e^* })^{-1/2}\right), 
$$
where we have used $T$ defined in \eqref{mouyzz}. 
We now choose $\e_*=\e_m+\e$ and $\e^*=\e$. For small enough $\e$,
thanks to (\ref{gwzN}), the constraint  (\ref{eqn:constraint}) is satisfied. Together with  the above equation, this proves 
proving \eqref{sza2}. 
Moreover, by \eqref{evcontrol},  $(K^{\b g, e}_{T, t}-z)^{-1}_{ii}$ is uniformly bounded for $z \in \cal D_\kappa$ and  this  implies 
\eqref{sza}.  

To prove (\ref{sza3}),
we need a level repulsion result from \cite{LanYau2015}.

 \begin{lemma}\label{shang} (Theorem 3.5 and 3.6 of \cite{LanYau2015})
Let $\lambda_{i,t}$ denote the eigenvalues of  $K(t) $  \eqref{eqn:Kt} with $V$  $(\eta_*, \eta^*, r)$-regular at $E_0$ and bounded as in Definition  \ref{XYH}.
Assume that there exists $ c<1$ such that   
\be\label{lu}
|\log \eta^*|\le c  |\log \eta_*|.
\ee
Let $\tau>0$. Then for large enough $N$, 
for any  $x>0$ we have 
$$
\max_{ 
t\in \cal T_\omega
} \P\left(\left\{|\lambda_{i,t}-E_0|\le W^{-1}\right\}\cap\left\{ |\lambda_{i,t}-\lambda_{i\pm 1,t}|\le  W^{-1}x\right\}\right)\le W^{\tau}x^{2-\tau},
$$ 
where $\cal T_\omega$ is defined in \eqref{Tdef}. 
\end{lemma} 

We now apply Lemma \ref{shang}  to the flow \eqref{eqn:deform}. The condition \eqref{lu} is trivially verified by the choice of $\e_*, \e^*$ 
in  Lemma \ref{LLniu}. Hence Lemma \ref{shang} implies the level repulsion estimate \eqref{sza3}.

It remains to prove  \eqref{zyaa}. From Lemma \ref{l:regpath} (i) applied to $K^{\b 0, e}_{T}(t)$ at time $t=T$, we have
$$
\P\left(\left\{\sup_{  0 \le e'-e \le  W^{  -1+\e } }\left|\#\Big\{k: \xi_k(e)\in [e, e']\Big\}-W\int_{e}^{e'} \rho_{{\rm fc}, T}(x)\rd x\right|\ge  W^{\tau} 
\right\}
 \cap \{|D-e|\geq N^{-C}\}
 \right)\le N^{-D}.
$$
We therefore just need to prove
\begin{equation}\label{eqn:expect}
\left|\int_{e}^{e'} \rho_{{\rm fc}, T}(x)\rd x-\int_{e}^{e'} \rho_{{\rm sc}}(x)\rd x\right|\leq W^{-1+\tau}.
\end{equation}
Recall the following relation between $m_{\fc,t } $ and $V$: 
\be\label{jslajf;}
 m _{\fc,t}(z)= m_{ V}\left(z+ t\; m _{\fc,t}(z)\right)=\frac1W\tr\left(V-z-t\; m _{\fc,t}(z)\right)^{-1} 
\ee
where   $V= A_{T}^{\b g}-B^{\b g, *} (D^{\b g} -e)^{-1}B^{\b g} = K^{\b g, e}_{T} $. 
 For $z=E+\ii\eta$ with  $ |E-e|\le r$ and $ \eta_* \le \eta \le \eta^*  $,  \eqref{jiangjiulh} implies  that 
\be
m _{\fc,0}(z) -m_{\rm sc}(z)=  \frac1W\tr\left(V-z \right)^{-1}-m_{\rm sc}(z)=\OO(N^{-\e^*/2}) 
\ee
holds with high probability. Similarly, by \eqref{jiangjiulh} and \eqref{jslajf;},  for any  $t \ge 0$ we have 
\be\label{66-}
 m _{\fc,t}(z) -m_{\rm sc}(z+ t\; m _{\fc,t}(z))= \frac1W\tr\left(V-z-t\; m _{\fc,t}(z)\right)^{-1}-m_{\rm sc}(z+ t\; m _{\fc,t}(z))=\OO(N^{-\e^*/2})
\ee
provided that 
\be\label{66}
 \eta_*\leq  \im (z+ t\; m _{\fc,t}(z)) = \eta +  t \im  m _{\fc,t}(z)\le  \eta^*, \quad  |  \re (z+ t\; m _{\fc,t}(z)) - e| \le r.
\ee 
For $t=T$  as  defined in Lemma \ref{LLniu}, we have
\be\label{77}
\eta_*\ll T \ll \eta^*,  \quad T \ll  r/2,  \quad |E-e|\le r/2,   \quad  0 \le \eta \le \eta^* /2.
\ee
Moreover, as proved in \cite[Lemma 7.2]{LanYau2015}, for any $0<\eta<\eta^*$, we have
\be\label{88}
 c \le  \im  m _{\fc,T}(z) \le C',\ |m _{\fc,T}(z)|< C'\log N,
\ee
for some positive constants $c, C'$. Equations (\ref{77}) and (\ref{88}) show that the assumption \eqref{66} holds for $t=T$, 
and concludes the proof of (\ref{eqn:expect}) by taking $\eta=0^+$ in (\ref{66-}).\end{proof}

\subsection{Regularity and weak   uncertainty principle.}\label{sec: reg} \ 
The GOE  component  in \eqref{11da} implies the following  regularity property and weak  uncertainty principle.  This lemma does not require 
the decomposition \eqref{sjui}, i.e., the Gaussian divisibility for the band matrix elements, we state it under this assumption for simplicity.
 
\begin{lemma}\label{Hide} Let $H$ be as in Theorem \ref{thm:afterreg} for some fixed $A>10$.
Let $\b \phi \in \R^N $  be   defined by 
$$
\phi_{i}=\mathds{1}_{W\le i\le N}.
$$ 
Recall the notatons from Definition \ref{defHg}.
Then there exists a (large)  constant  $C_r=C_r(A)$ (the  subscript $r$ is used to indicate that the constant is related to the  regularity)  
such that for any fixed $D>0$
\begin{align}\label{reg1}
& \max_{\|\b g\|_\infty\le W^{0.9}/N}\P\left(\exists \,t: |t|\le20, \;  k\in \Z_N \; \text{ such that } \;  \left|\lambda_k^{\b g + t \b \phi }\right|\le 20, \; \left\|\bw_k^{\b g + t\b \phi }\right\|_2^2\le N^{-C_r}\right)\le N^{-D}, \\
&\max_{\|\b g\|_\infty\le W^{0.9}/N}\P\left(\exists \,e: |e|\le10, \quad  B^*\frac1{(D^{\,\b g}-e)^2}B\ge N^{C_r}\left(B^*\frac1{(D^{\,\b g}-e) }B\right)^2+N^{C_r} \right)\le N^{-D}.
\label{reg2}
\end{align}
\end{lemma}

The proof of this lemma follows the one for  \cite[Proposition 3.1]{BouErdYauYin2017}. 
Lemma \ref{Hide} is weaker in the sense that  the error $N^{\pm C_r}$ was originally given by  order one quantities in \cite{BouErdYauYin2017}. In addition,  \cite[Proposition 3.1]{BouErdYauYin2017} applies  to any approximate eigenvector without assuming  the small  GOE  regularization $N^{-A}H^{\rm G}$. On the other  hand, Lemma \ref{Hide} works for $W \ge N^{3/4+c}$  (in fact,  $W \ge N^{1/2+c}$ is enough), while \cite{BouErdYauYin2017} required $W=\Omega(N)$.

\begin{proof}[Proof of Lemma \ref{Hide}]
We first note that  $0.9$ in (\ref{reg1}) and (\ref{reg2}) can be replaced by any fixed number less than 1 in the following arguments.

We will first prove the following form of an uncertainty principle: approximate eigenvectors for $D^{\b g}$  have some weight on the first $W$ coordinates, in the sense that there exists $C>0$ such that for any fixed  $D>0$, 
\be\label{15z2}
\max_{\|\b g\|_\infty\le W^{0.9}/N}\P\left( \exists \bv\in \R^{N-W}   \mbox{ with } \| \bv\|_1=1  \nc, \; e\in \R,  \;  \|B^*\bv\|_2+ \|(D^{\b g}-e)\bv \|_2\le N^{-C}\right)\le N^{-D}.
\ee

We first consider $B^*\bv$.
Thanks to the  component  $N^{-A}H^{\rm G}$  in \eqref{11da},  for any fixed $\bv$ and $1\le n\le W$, 
there is an $a_0$ independent of $H^{\rm G}$ such that 
$(B^*\bv)_n=    a_0 + N^{-A}  \xi_n \cdot \bv$ with  $\b \xi_1, \ldots,  \b \xi_W$ having independent Gaussian entries of variance order $1/N$.   
Thus there exists $C>0$ such  that for any $\|\bv\|_2=1$, we have
$$
\P\left(\left|(B^*\bv)_n\right|\le N^{-C}\right)\le 1/2, 
$$
for all $1\le n\le W $. 
Taking intersection of these independent events,   we have proved that 
there exist $C >0,  c> 0 $ such that for any $\bv$ as above, 
\be\label{yanghuo}
\P\left(\|B^*\bv\|_2\le N^{-C } \right)\le e^{- c W}.
\ee

The matrix $D$ in $H$ is itself a band matrix of size  $N-W$  and band width $W$. 
 Denote by  $\lambda_k^D$ be the eigenvalues of $D$. 
The  local law  in \cite[Theorem 2.1]{ErdYauYin2012Univ} was established up to  the scale $W^{-1+\tau}$ for any constant $\tau>0$, strictly speaking 
for random band matrices satisfying $\sum_j s_{ij} =1$. For $D$, this assumption is violated for $i$ in a set of size at most $2W$, 
but \cite[Theorem 2.1]{ErdYauYin2012Univ} still holds by elementary adjustments left to the reader. 
This theorem  implies in particular that  with  probability $ 1- N^{-D}$ 
\be\label{qzih}
  \max_{e\in\R}\frac{\#\{k,\; \lambda^D_k \in [e, e+W^{-1+\tau}]\}}{NW^{-1+\tau}}\le 10.
\ee
 Denote by $\lambda^{D^{\b g}}_k$ and $\b\psi^{D^{\b g}}_k$, $1\le k\le N-W$, the eigenvalues and  eigenvectors of ${D^{\b g}}$. 
A trivial bound on the  eigenvalue perturbation  gives
$
 | \lambda_k^{D^{\b g}}-\lambda_k^D|= \| {\b g}\|_\infty \le W^{0.9}/N,
$
so that  $ \lambda^{D^{\b g}}_k\in [e, e+N^{-1}]$ implies $ \lambda^{D}_k\in [e-W^{0.9}/N, e+ 2 W^{0.9}/N]$.  
Hence with high probability we have 
\begin{align*}
& \big | \{k: \lambda^{D^{\b g}}_k\in [e, e+N^{-1}]\}\big |
\le \big | \{j: \lambda^{D}_j\in  [e-W^{0.9}/N, e+ 2 W^{0.9}/N]\}\big |  
 \le  10 W^{0.9 },
\end{align*} 
where we have used \eqref{qzih} and  $W^{0.9 }/N  \ge W^{-1+\tau}$.
Hence for any  $D>0$, for large enough $N$ we have  
$$
\P\left(\exists e\in \R,  \; \big | \{k: \lambda^{D^{\b g}}_k\in [e, e+N^{-1}]\}\big |\ge 30 W^{0.9}\right)\le N^{-D}.
$$
As a consequence, if we define 
$S_e={\rm span}\{\b\psi_k^{D^{\b g}}:\la_k^{D^{\b g}}\in[e,e+N^{-1}]\}$, then
$\dim (S_e)=\OO(W^{0.9})$ with high probability. 
For such an  $S_e$ of   dimension $\OO(W^{0.9})$, 
we can choose a finite set  $\cal N$ in the unit sphere of  $S_e$ \nc with  
$|\cal N|= N^{\OO( W^{0.9})}$ such that for any $\bv$ in the unit sphere  there is a  $p\in \cal N$ such that $|\bv - p| \le  N^{-C-1}$ with $C$ 
being the constant in \eqref{yanghuo}. 
 Together with   \eqref{yanghuo} and the fact that $\|B\|\le N$ holds with very high probability,  we obtain that 
 there exists $C>0$ such that for any fixed  $D>0$,
$$
\P\left( \exists \bv\in S_e,\;  \|B^*\bv\|_2\le N^{-C}\right)\le N^{-D},
$$
 where we have used $e^{-c W} N^{\OO( W^{0.9})}  \le e^{-c' W} $ for some $c'> 0$. 
Therefore  there exists  $C>0$ such that  for any fixed $e$ and $D>0$, for large enough $N$ we have $\mathbb{P}(A_e)\leq N^{-D}$ where
\begin{equation}\label{eqn:A_e}
A_e=\{ \exists \bv\in \R^{N-W}   \mbox{ with } \| \bv \|=1,  \nc \; \|B^*\bv\|_2+ \|({D^{\b g}}-e)\bv \|_2\le N^{-C}\}.
\end{equation}
By union bound, we also have 
$\mathbb{P}(\cap_{e\in N^{-2C}\mathbb{Z},|e|<N^C}A_e)\leq N^{-D}$.
Moreover, $\|D^{\b g}\|\leq N^{C}$  with high probability, so that  (\ref{15z2}) follows easily.

We now  show how \eqref{reg1} follows from  (\ref{15z2}).  By definition $D^{\b g + t\phi}=D^{\b g}-t$ and $A^{\b g + t\phi}=A^{\b g}$, so 
the eigenvector equation is
\begin{align*} 
 A^{\b g}\bw_k^{\b g + t\b \phi}+B^*\bv_k^{\b g + t\b \phi} & =\lambda_k^{\b g + t\b \phi}\bw_k^{\b g + t\b \phi},\\
 B\bw_k^{\b g + t\b \phi}+(D^{\b g}-t)\bv_k^{\b g + t\b \phi}& =\lambda_k^{\b g + t\b \phi}\bv_k^{\b g + t\b \phi}.
\end{align*} 
If $\|\bw_k^{\b g + t\b \phi}\|_2 \le N^{-C}$ for some $C>0$, then  $\| A^{\b g}\bw_k^{\b g + t\b \phi}\| + \| B\bw_k^{\b g + t\b \phi} \| \le N^{-C/2}$ with very high probability. 
Hence $ \bv_k^{\b g + t\b \phi} $, after normalization,  realizes the condition (\ref{eqn:A_e}) with $e=t+\lambda_k^{\b g + t\b \phi}$. 
Therefore,    \eqref{reg1} follows from \eqref{15z2}.

We now  prove \eqref{reg2}. The event in  \eqref{reg2} means that  for some normalized $\bv\in \R^{N-W}$ and $|e|<10$,
\be\label{ting} 
 \big \|\frac1{(D^{\b g}-e) }B\bv \big \| _2\ge N^{C_r} \left(\big \|B^*\frac1{(D^{\b g}-e) }B\bv\big \|_2+1\right).
\ee
Denoting
 $\wt \bv =(D^{\b g}-e) ^{-1}B\bv / \|(D^{\b g}-e) ^{-1}B\bv\|_2$, it follows from \eqref{ting} that     
\begin{align*}
\|(D^{\b g}-e)\wt \bv\|_2 & =\frac{ \|B\bv\|_2}{\|(D^{\b g}-e) ^{-1}B\bv\|_2}  \le    \frac{ \|B\bv\|_2} {   N^{C_r} \left(\big \|B^*\frac1{(D^{\b g}-e) }B\bv\big \|_2+1\right)}  \le N^{-C_r}\|B\bv\|_2 , \\
 \|  B^*   \wt \bv\|_2 &=\frac{ \|   B^*   (D^{\b g}-e) ^{-1}B\bv\|_2}{\|(D^{\b g}-e) ^{-1}B\bv\|_2}\le N^{-C_r}.
\end{align*}
Since $\|B\|_{\rm op}\le N$ with  high probability, $\wt v$ realizes the event $(\ref{eqn:A_e})$, so that 
\eqref{15z2} implies  \eqref{reg2}. 
 \end{proof}

\subsection{Proof of Theorem \ref{thm:afterreg}.}\label{Mfr} \   
We make rigorous our proof sketch from Subsection \ref{sub:sketch}.
We consider the full band matrix $H$, the proof for the minors $H^{(k)}$ being the same up to trivial adjustments.

Recall the notations from Subsection \ref{MR}. 
There, we assumed that the map 
$g\to e= \lambda_k^g+g$ has a regular inverse which enable us to  define the curve $\cal C_k(e) = \lambda^g_k$. 
To prove  this,  a simple perturbation calculation yields that 
$\partial (\lambda_k^g+g)/\partial g=\sum_{i=1}^{W} \left|\psi_k^g( i) \right|^2$
By \eqref{reg1}, 
 $\left|\psi_k^g( i) \right|^2 \ge N^{- C_r}$  for all $|g|<20$ for some constant $C_r > 0$, with high probability.  Thus the invertibility is proved with high probability
 and from now on we shall restrict ourselves to this case. 
By  differentiating w.r.t.  $g$ in the identity $\mathcal{C}_k(\lambda_k^g+g)=\lambda_k^g$,
 we have  
\begin{equation}\label{eqn:slope}
\Big | \frac{\partial}{\partial e}\mathcal{C}_k(e) \Big | = \Big | 1-\Big(\sum_{i=1}^{W} \left|\psi_k^g( i) \right|^2\Big)^{-1}\Big | \le N^{C_r} .
\end{equation}

We now complete the proof of Theorem \ref{thm:afterreg}, successively considering QUE for some small mean-field matrices, then QUE and delocalization  for band matrices,  the semicircle law, and universality.\\
 
\noindent 
{\it Part  1A:  QUE for a small matrix.}  We will prove that a slightly more general form of \eqref {maindel} holds  for   eigenvectors $\b\psi^{\b g}$ of $H^{\b g}$
with any $\|{\b g}\|_\infty\le W^{-3/4}$. 
 But for simplicity, we present the proof for $\b g=\b 0$, and point out the  modification    for the general  case   whenever it is  needed.

We will prove  the following   delocalization and QUE  for the eigenvector  $\bu_{k'}(\lambda_k)$ of $Q^{\b g}_{\lambda_k}$ defined in \eqref{gaibian}:
\begin{align}
&
\label{jbg1}
\P\left(\exists k\in \llbracket 1, W\rrbracket , j \in \Z_N \; : \; |\lambda_j|\le 2-\kappa, \; \xi_k(\lambda_j)=\lambda_j, \; \|\bu_k(\lambda_j)\|^2_\infty\ge W^{-1+\tau} \right) \le N^{-D},\\
\label{jbg2}
&\P\Big ( \exists k\in \llbracket 1, W\rrbracket , j \in \Z_N\; :\, |\lambda_j|\le 2-\kappa, \; \xi_k(\lambda_j)=\lambda_j, \; 
\left|\| J\cdot \bu_k(\lambda_j)\|^2_2-\frac12\right|  \ge \frac{N^{\frac{1}{2}+\tau}}{W}+\frac{N^{\frac{\e_m}{2}+\tau}}{W^{1/2}}\nc   \Big ) \le N^{-D}.
\end{align}

The difference between \eqref{jbg1}--\eqref{jbg2} and \eqref{sza}--\eqref{sza2} are the randomness of their arguments, i.e.,  $e$ in  \eqref{sza}--\eqref{sza2}
is replaced by $\lambda_j$ in \eqref{jbg1}--\eqref{jbg2}.   To prove \eqref{jbg1} and \eqref{jbg2}, our basic strategy is combining  an   $\e$-net argument with 
a perturbation theory of eigenvectors. 
Let $M =N^{2C_1+2D}$;   here $D$ is the    constant appeared in  \eqref{jbg1} and \eqref{jbg2} (not to confuse with the notation that   $D$ was also used to denote a matrix)
and  $C_1 = D+6C_r$ where $C_r$ is the constant in \eqref{eqn:slope}.

We denote  $E_n =  n N^{-C_1}$, and claim that for each   fixed $n$ satisfying $[E_n, E_{n+1}]\subset [-2+\kappa, 2-\kappa]$, there is a deterministic $e_n$ with  
$
E_n< e_n < E_{n+1}
$
such that  
\be\label{wanjyz}
\P\left(\inf_{j, n}  |\lambda^D_j-e_n|\le M^{-1}  \right)\le N^{-D }
\ee
where the $\lambda^D_j$'s are the eigenvalues of $D$.
To see this, note that for any $\eta>0$
$$
 \int_{E_n }^{E_{n+1}}  \E \im [D - E- \ii \eta ]^{-1} \rd E \le  \E \int_\R \im [D - E- \ii \eta ]^{-1} \rd E \le N.
$$ 
Hence  there is an  $e_n \in  [E_n,  E_{n+1}]$ such that,  with $\eta = M^{-1}$,
$
   \E \im [D - e_n- \ii \eta ]^{-1}   \le C N^{C_1+1}.
$
By the Markov inequality, 
$
\P \Big ( \inf_j  |\lambda^D_j - e_n| \le M^{-1}  \Big  ) \le N^{C_1+1} M^{-1}$ and (\ref{wanjyz}) holds,
so that  
we can restrict our consideration to the set   $|D-e| \ge N^{-C}$ for any $C \ge 2C_1+2D$.  
 By Lemma \ref{henna},  \eqref{sza}, \eqref{sza3} and \eqref{sza2} hold. In particular, \eqref{sza3} holds with $x = N^{-C_1/2}$. 
Hence  for all $n$ and $l$  satisfying  $|\xi_l(e_n)-e_n|\le W^{-1}$ we have  
\be\label{gozet}
 \|\bu_l(e_n)\|^2_\infty\le W^{-1+\tau}  ,\quad \left| \left\|J\cdot \bu_l(e_n)\right\|^2_2-\frac12\right|\le  \frac{N^{1/2+\tau}}{W}+\frac{N^{\frac{\e_m}{2}+\tau}}{W^{1/2}}\nc,\quad
 |\xi_{l  \pm1}(e_n)-\xi_l(e_n)|\ge N^{ -C_1/2}
\ee
with probability larger than $ (1-N^{-D}-N^{ -C_1(1- \tau) } W^{2+ \tau}) \ge 1-2 N^{-D} $. Here we have used the choice $C_1 = D+6C_r$ and
 $\tau$ can be  an arbitrarily  small number. 
 
We define
\be\label{mk}
m(\lambda_k)  = \sup_n \{n: e_n  <  \lambda_k \}
\ee
For simplicity we denote $
{\wt e}=e_{m(\lambda_k)}$. 
Recall that $\cal C_k(\lambda_k)=\lambda_k$, and thus \eqref{eqn:slope} 
and   \eqref{reg1} assert that $\left|\partial  {\cal C}_k(e)/\partial e\right|\le N^{C_r}$
holds with high probability. 
Since  $e_{n+1} - e_n \le 2 N^{-C_1}$,   \eqref{mk} implies $|\wt e-\lambda_k|\le 2N^{-C_1}$.  
Since $C_1\ge 6 C_r $ so that  
$N^{-C_1}N^{C_r}\ll N^{-0.8C_1}$, we have 
\be\label{23}
  \left|\cal C_{k }({\wt e\,})-{\wt e\,}\right|\le   \left|\cal C_{k }({\wt e\,})-\cal C_{k }( \lambda_k) \right| + \left| \lambda_k-{\wt e\,}\right|  \le N^{C_r} \left| \lambda_k-{\wt e\,}\right| 
\le 2N^{-C_1}N^{C_r}  \le N^{-0.8C_1} \ll     W^{-1  } 
\ee
with   probability larger than $1-N^{-D}$. 

Recall that  $k'\in\llbracket 1,W\rrbracket $ is the index given by the relation  $
\xi_{k'}(e )=\cal C_k(e)$ for all $e$. 
Applying \eqref{gozet} 
with $e_n$ set to be $\tilde e$  and using $C_1 \gg D$,   
we obtain   the level repulsion bound 
$$
 \P\left(|\xi_{k'  \pm1}({\wt e}\,)-\xi_{k'}({\wt e}\,)|\ge N^{- C_1/2}\right)\ge 1-N^{-D}.
$$ 
Together with  the continuity argument used in \eqref{23},  the level-repulsion  holds between $\cal C_k$ and $\cal C_{k\pm1}$ in the interval $[\,\wt e, \lambda_k\,]$,   i.e.,  
\be\label{ljih1}
\P\left(\exists \,e\in[\,{\wt e}, \lambda_k], \;s.t. \;|\cal C_{k  \pm1}(e)-\cal C_{k}(e)|\le \frac12N^{-  C_1/2 }\right)\le N^{-D}.
\ee
Integrating the perturbation formula (\ref{eqn:perbevect}), we get
\be\label{zapgys}
\bu_{k'}(\lambda_k)=\bu_{k'}({\wt e}\,)-\int_{\lambda_k}^{{\wt e}}
 \sum_{\ell \ne k}\frac{  {\b u}_{\ell'} (e)}{\cal C_{k}(e)-\cal C_\ell (e)}
\left({\b u}_{\ell'}  (e),  B^* \frac{1}{(D-e)^2}B \;  {\b u}_{k'}(e)\right)\rd e
\ee
Inserting (\ref{ljih1}) into \eqref{zapgys} and using  $|\wt e-\lambda_k|\le 2N^{-C_1}$, we obtain
\be\label{ver}
\left\|\bu_{k'}(\lambda_k)-\bu_{k'}({\wt e}\,)
\right\|_\infty\le CN^{- C_1/2}
\max_{ e\in[\,{\wt e}, \lambda_k]}
\max_{ \ell\ne  k}
 \frac{\left|\left(\bu_{ \ell'    }, B^* (D-e)^{-2 }B \bu_{ k'}\right) \right|}{|\cal C_{ \ell}(e) |+1}.
\ee
The numerator of the last term can be bounded by using \eqref{reg2} so that  
\be\label{wanshang}
 \left|\left(\bu_{ \ell'   }, B^* (D-e)^{-2 }B \bu_{ k'}\right) \right|\le 
N^{C_r} \left( \left\|B^* (D-e)^{-1 }B\bu_{ \ell'   }\right\|^2_2+1\right)^{1/2}\left( \left\|B^* (D-e)^{-1 }B\bu_{ k'  }\right\|^2_2+1\right)^{1/2}.
 \ee
Inserting the identity
$B^* (D-e)^{-1 }B\bu_{ \ell'   }=Q_e\bu_{ \ell'   }-A\bu_{ \ell'   } =\xi_{ \ell'}(e)\bu_{ \ell'   }-A\bu_{ \ell'   }$
into the right hand side  of \eqref{wanshang}, we obtain 
$$
 \left|\left(\bu_{ \ell'   }, B^* (D-e)^{-2 }B \bu_{ k'}\right) \right|\le N^{C_r} 
 \left( \xi_{ \ell'}(e)+\|A\|_{op}+1\right)  \left( \xi_{ k'}(e)+\|A\|_{op}+1\right) .
$$
It is easy to prove that  $\|A\|_{op}=\OO(N)$ with high probability.  
Together with the fact that  
 $\xi_{ k'}(e)\in [\lambda_j, \xi_{k'}(\wt e\,)]$ for $e\in[\,{\wt e}, \lambda_k]$,   which follows from $\partial_e\cal C_j(e)<0$, 
we have proved 
$$
\left|\left(\bu_{ \ell'   }, B^* (D-e)^{-2 }B \bu_{ k'}\right) \right|\le N^{C_r+2}\left(|\xi_{ \ell'}(e) |+1\right).
$$
Inserting this bound  into  \eqref{ver}, using that  $\xi_{ \ell'}(e)  = \cal C_{ \ell}(e)$ in the denominator and   the choice  $C_1 = D+6C_r$,  
 we have proved  \eqref{jbg1} and \eqref{jbg2}.
 
Notice that the constant $C_r$ is associated with the uncertainty principle in \eqref{reg2} and $N^{-C_1}$ is the grid size. 
We can make the grid size  small by 
choosing large $C_1$; the price to pay is that the initial data in the stochastic flow argument becomes large, i.e., the constant $C_1$ in \eqref{clMang} is large. 
However, the results we use on the stochastic flow (e.g. \ref{shang}) are  insensitive to this  constant, which is the main reason we can choose 
$C_1$ large.\\

\noindent  
{\it  Part 1B: Delocalization and QUE for the band matrices.}   
By definition \eqref{tanxi},  
$ \bu_{k'}(\lambda_k)=\bw_k/\|\bw_k\|_2 $.  
Equations \eqref{jbg1} and \eqref{jbg2} can be written  in the following form: for any fixed large $D>0$ and small $\tau>0$, 
 \begin{align}\label{jbg3}
& \P\left(\exists k \in \Z_N,\;  : \; |\lambda_k|\le 2-\kappa, \;  \;\frac{\max_{1\le i\le W} |\psi_k(i)|^2}{\sum_{i=1}^W |\psi_k(i)|^2}\ge W^{-1+\tau} \right) \le N^{-D},\\
&
\P\left(\exists k \in \Z_N\, : \,   |\lambda_k|\le 2-\kappa, \; \left|\frac{\sum_{i=1}^{ W/2} |\psi_k(i)|^2}{\sum_{i=1}^W |\psi_k(i)|^2}-\frac12 \right|\ge  \frac{N^{1/2+\tau}}{W}+\frac{N^{\frac{\e_m}{2}+\tau}}{W^{1/2}}\nc   \right) \le N^{-D}. 
\notag
\end{align} 
Clearly we can  shift the indices so that  
$$
\P\left(\exists k \in \Z_N\, :  \; |\lambda_k|\le 2-\kappa, \; \max_{n\in \Z_N}\left|\frac{\sum_{i=1}^{W/2} |\psi_k(n+i)|^2}
{\sum_{i=1}^W |\psi_k(n+i)|^2}-\frac12 \right|\ge  \frac{N^{1/2+\tau}}{W}+\frac{N^{\frac{\e_m}{2}+\tau}}{W^{1/2}}\nc   \right) \le N^{-D}, 
$$
and a similar shifted version of \eqref{jbg3} holds. 
Since $W,N$ and $\e_m$ are related  by  \eqref{gwzN}, we  have $ \frac{N^{1/2+\tau}}{W}+\frac{N^{\frac{\e_m}{2}+\tau}}{W^{1/2}}\nc =\oo(W/N)$, so that exactly as in (\ref{eqn:ratio})--(\ref{eqn:QUEsketch}) we obtain
$$
\P\Big (\exists k \in \Z_N,\; :  \; |\lambda_k|\le 2-\kappa, \;   \max_{n\in \Z_N} \Big |\frac{N}{W} \sum_{i=1}^{W/2} |\psi_k(i+n )|^2-\frac12 \Big |\ge 
  \frac{N}{W}\left(\frac{N^{1/2+\tau}}{W}+\frac{N^{\frac{\e_m}{2}+\tau}}{W^{1/2}}\right)\nc 
 \Big) \le N^{-D}.
$$
when $N/W$ is an integer.
If $N/W$ is not an integer,  the delocalization estimate \eqref{jbg3} can be used to lead to the same conclusion. 
 Moreover, from \eqref{gwzN} we have $\frac{N}{W}\cdot\frac{N^{1/2+\tau}}{W}<N^{-2a}$ and $\frac{N}{W}\cdot\frac{N^{\frac{\e_m}{2}+\tau}}{W^{1/2}}<N^{-\frac{3}{2}a}$ \nc with $a>0$ given in \eqref{gwzN}.
We  have thus proved  the QUE part of  Theorem \ref{thm:afterreg}.  Finally, note that the above QUE for length interval $W/2$ obviously implies the same estimate for length $W$.\nc

Finally, the proof of Theorem \ref{thm:afterreg} just given above holds for all $\|\b g\|\le W^{-3/4}$ since all lemmas were proved under this assumption.    We  have thus proved  that for any fixed $\tau, D>0$, for large enough $N$ we have
\be\label{wzajj2}
\min_{\|\b g\|\le W^{-3/4}}\P\Big (\exists j \in \Z_N,\;  :  \; |\lambda^{\b g}_j|\le 2-\kappa, \;   \frac NW\sum_{i=1}^{ W} |\psi^{\b g}_j(i )|^2 =1+\OO(N^{-  \frac{3}{2}\nc a+\tau})\Big)\ge 1-N^{-D} 
\ee

\noindent 
{\it Part 2: the semicircle law.} 
Following  the  mean-field reduction method,  we first prove the following lemma.

 \begin{lemma}\label{zybb} Recall the definition of the constant $a$  in \eqref{gwzN}. 
Under the assumption of Theorem \ref{thm:afterreg},   for any fixed $e_0$ with $|e_0|\le 2-\kappa$ we have  
 \be\label{axing}
 \max_j\P\left(|\lambda_j-e_0|\le N^{-1+\frac{a}{2}}  \mbox { and }  \Big |\left(\cal C_j(e_0)-e_0\right)-\frac NW\left(\lambda_j-e_0\right)\Big |\ge W^{-1-{\frac{a}{10}}} \right)\le N^{-D}.
 \ee
 \end{lemma}
  
\begin{proof}
Recall the definition of the matrix $H^g$ from  \eqref{Hg}
and the relation $ \frac{\pt (g+\lambda^g_j )}{\pt g}  =  \sum_{i=1}^{W} \left|\psi_j^g( i) \right|^2 $.
Integrating this relation from 
$g_0$ 
to $0$ with $g_0$ defined by  the equation $g_0+\lambda^{g_0}_j=e_0$, we have 
\be\label{zee}
 \int_{g_0}^0\sum_{i=1}^{W} \left|\psi_j^g( i) \right|^2\rd g=\lambda_j-e_0.
\ee
By  \eqref{wzajj2},  for each $|g|\le W^{-3/4}$ fixed,   
we have $\sum_{i=1}^{W} \left|\psi_j^g( i) \right|^2=W/N\left(1+\OO(N^{-a})\right)$ with  high probability. Hence the left side of \eqref{zee} is equal to $-g_0 
W/N\left(1+\OO(N^{-a})\right)$ with high probability.  By definition,  $\cal C_j(e_0) =\lambda^{g_0}_j =e_0-g_0 $. Inserting this relation into \eqref{zee}, 
we have proved \eqref{axing}. 
\end{proof}

We now  prove the local semicircle law by using  \eqref{zyaa}. 
 For $\e>0$ small enough, we consider $E_2>  E_1 $  with  $\Delta:=E_2-E_1\le N^{-1+\e}$.
 Clearly, we can assume
 $\Delta\ge 1/N$. We apply  Lemma \ref{zybb} with the choice $e_0=E_1$: for any $D>0$, for large enough $N$ we have
$$
   \#\left\{k: \cal C_k(e_0)\in \Big [e_0, \; e_0+ \Delta\frac{N}{W}-\frac{1}{W} \Big]\right\}   \le\#\left\{k: \lambda_k\in [E_1, \; E_2] \right\}  \\
\le \#\left\{k: \cal C_k(e_0)\in \Big[e_0, \;  e_0+\Delta\frac{N}{W}+\frac{1}{W}\Big]\right\}
$$
with probability $1-N^{-D}$. 
Since $\cal C_k = \xi_{k'}$ represents the same curve,    we can apply   \eqref{zyaa} with the choices 
 $
e=E_1$, $e'-e=\Delta N /W-1/W $, or $\Delta N/W+1/W. 
 $
Hence the estimate \eqref{zyaa} implies  the local semicircle law  and we have completed the proof of Theorem \ref{thm:afterreg}.\\

\noindent {\it Part 3: Eigenvalues local statistics.}\ 
We rely on fixed energy universality result  for a matrix flow, from \cite{LanSosYau2016} (note that the constraint $\om_0>2/3$ below is probably not optimal but sufficient in our setting).

\begin{theorem}[Fixed energy universality for the Dyson Brownian motion \cite{LanSosYau2016}]
Let $V$ be a $n\times n$ deterministic matrix and $Z$ be a $n\times n$ standard GOE matrix.  Consider  $H = V + \sqrt{t_0} Z$ with $t_0 = n^{\om_0}/n$.  
Assume that $V$ is ($n^{-\delta_1}t_0$, $n^{-\delta_2}$, $n^{\delta_3} t_0)$ regular at E (see Assumption \ref{XYH}) with  ($c$ is a universal small constant)
$$
\frac{2}{3}< \om_0<1,\ \delta_2 < \min\left(\frac{1-\om_0}{4},\delta_3,c\right).
$$
Remember the notation
$
\rho^{(n)}_{\fc,t_0}
$ 
for the density corresponding to 
 the Stieltjes transform $m^{(n)}_{\fc,t_0}$  defined in (\ref{eqn:freefunc}).

For any smooth test function $O \in \mathscr{C}^\infty ( \R^k )$ with compact support, 
there are   constants $c, C >0$ such that
$$
\bigg | \int_{\mathbb{R}^k} O (\ba) p_{H}^{(k)} \left( E + \frac{\ba}{ N \rho^{(n)}_{\fc,t_0}(E) } \right) {\rm d \ba}
- \int_{\mathbb{R}^k} O (\ba) p_{\rm GOE}^{(k)} \left( E + \frac{\ba}{ N \rho^{(n)}_{\fc,t_0}(E) } \right) {\rm d \ba} \bigg | \leq  C n^{- c}.
$$
\end{theorem}

We apply this result to the $W\times W$ matrix flow $t \to K_{T-t}^{\b g, e}$
at $t = T$ with the initial data  $V= K_{T}^{\b g, e}$, i.e., $n= W$.
By Corollary \ref{lemLQ}, $V$  is $(\eta_*, \eta^*, r)$ regular with the parameters defined in Theorem \ref{LLniu}. 
With $\eta_*, \eta^*, r, T$ defined in \eqref{mouyzz}, we have the following identifications
$$
\delta_3 = 2 \e^*  \log_W N , \; \delta_1= \e^* \log_W N , \; \delta_2 =  \e^*  \log_W N , \; 1- \omega_0 =  \log_W N(\e_*-\e^* ).
$$
The above theorem with $e=E$ gives (we consider the case $k=2$ to simplify the presentation)
$$
\bigg | \int_{\mathbb{R}^2} O (\ba) p_{Q_E}^{(2)} \left( E + \frac{\ba}{ N \rho_{\rm sc}(E) } \right) {\rm d \ba}
- \int_{\mathbb{R}^2} O (\ba) p_{\rm GOE}^{(2)} \left( E + \frac{\ba}{ N \rho_{\rm sc}(E) } \right) {\rm d \ba} \bigg | \leq  C N^{- c}.
$$
where  we replaced $\rho^{(n)}_{\fc,t_0}$ with $\rho_{\rm sc}$  by taking $\eta=0^+$ in (\ref{66-}).
We can write
\begin{multline*} 
 \int_{\mathbb{R}^2} O (\ba) p_{Q_E}^{(2)} \left( E + \frac{\ba}{ N \rho_{\rm sc}(E) } \right) {\rm d \ba}
=   
\frac 12  \sum_{k'  \not =  j'} \E  \, O \Big (W \rho_{\rm sc} (E) (\xi_{k'}-E) , W \rho_{\rm sc} (E) (\xi_{j'}-E) \Big )\\
=   
\frac 12  \sum_{k \not =  j} \E  \, O \Big (W \rho_{\rm sc} (E) (\cal C_k (E)-E) , W \rho_{\rm sc} (E) (\cal C_j (E)-E) \Big). 
\end{multline*} 
Recall that  there is a shift of indices $k \to k'$ (depending  on the randomness)  so that $ \cal C_k (E) = \xi_{k'}$. 
In the   expression above, we have summed  over all indices and thus this shift is irrelevant  for our purpose. 

Applying   \eqref{axing} with  $e_0=E$, we can substitute $W \rho_{\rm sc} (E) (\cal C_k (E)-E) $ with 
 $N \rho_{\rm sc} (E) (\lambda_k-E) + \OO(W^{-\frac{a}{10}}) $ in the above equation.
Note that \eqref{axing} holds only for eigenvalues in a small neighborhood of $E$. Since $O$ is compactly supported, this restriction does not affect the usage of  
\eqref{axing}  in the last equation. 
Finally, the error term $\OO(W^{-a/10})$ is negligible, which concludes the  the proof.

\section{A comparison method} \label{sec: comp}

In this section, we prove the theorems \ref{Main}, \ref{lsc} and \ref{Univ}. The basic idea follows  the Green function comparison method in \cite{ErdYauYin2012Univ}, interpolating between resolvents of two matrices $H$ and $\wt H$. 
However, contrary to the setting from  \cite{ErdYauYin2012Univ}, we only have  a priori  estimates on the Green's function for $\wt H$, but not for $H$.
A self-consistent Green function comparison method for band matrices was developed in \cite{BaoErd2015}, which only requires estimates on the Green's function of one of both matrices. Our a priori estimates on the Green's function are different so that we proceed with another self-consistent method from \cite{aniso}.

\subsection{Elementary facts.}\ 
Recall the resolvent of a matrix $H$ can be written as 
\be\label{imG}
G_{ij}(z)=\sum_{k}\frac{\psi_k(i)\psi_k(j)}{\lambda_k-z}, \quad \im G_{ii} (E+\ii \eta) =
\sum_k \frac{\eta |\psi_k(i)|^2} {(E-\lambda_k)^2+\eta^2}
\ee
where $\b\psi_k$  is the $k$-th eigenvector with eigenvalue $\lambda_k$. 
The following lemma is a classical fact connecting the Green function with delocalization of eigenvectors and local laws.

\begin{lemma}\label{Equ} Let $(H_N)_{N\geq 1}$ be a sequence of  $N\times N$ random symmetric matrices. Suppose  that for any 
 $c>0$ there exists a constant $\kappa_c>0\in \R$  such that for any  $D>0$
\be\label{rough bout}
\inf_{j\in \llbracket cN, (1-c)N\rrbracket} \P\left(|\lambda_j |\le 2-\kappa_c \right)\ge 1-N^{-D}
\ee
provided that $N$ is large enough. 
Consider the following assertion:
for any  small $\kappa, \tau>0$ and $D>0$ 
\begin{align}
&\label{buchi2}
\sup_{|E|\le 2-\kappa,N^{-1}\le \eta\le 1}\max_{i}\left(\im G_{ii}(z)\right)=\OO(N^{\tau}),\\
&\label{buchi3}
\sup_{|E|\le 2-\kappa,N^{-1}\le \eta\le 1}\max_{i,j}|G_{ij}(z)|=\OO(N^{\tau}),
\end{align}
hold with probability larger than $1-N^{-D}$. 
 Then 
 \begin{enumerate}
\item \eqref{buchi2}  implies  \eqref{maindel}. 
\item \eqref{maindel}  and   \eqref{SemiC} imply   \eqref{buchi3}.
\end{enumerate}

  \end{lemma}

\begin{proof} 
For any $k\in \llbracket cN, (1-c)N\rrbracket$, by \eqref{rough bout},  we can assume 
$ |\lambda_k |\le 2-\kappa_c  $
for some  $\kappa_c> 0$.  
Then  \eqref{buchi2}  implies that, with high probability,
$$
\eta^{-1}|\psi_k(i)|^2 \le \im G_{ii}(\lambda_k+ \ii \eta )=\OO(N^\tau), \quad \eta=N^{-1},
$$
which is \eqref{maindel}.   
On the other hand, the bound  \eqref{maindel}  on $\psi_k(i)$ and the eigenvalue distribution estimate  \eqref{SemiC} 
inserted in \eqref{imG} yield  \eqref{buchi3} by a simple dyadic decomposition.
\end{proof}  

\subsection{Proof of Theorem \ref{Main}.}\   By Theorem \ref{lemma main3},   \eqref{maindel} of Theorem \ref{Main}  is just  a corollary of the following lemma. 
In the remainder of this section,  we will prove Lemma \ref{shay}.  

\begin{lemma}\label{shay} If the statement   \eqref{buchi2} holds for all $H$  in \eqref{sjui},  then  \eqref{buchi2}  holds for any  $H$ in Theorem \ref{Main}. 
\end{lemma}

To prove  Lemma \ref{shay}, note that  for each $H$ in Theorem \ref{Main}, 
there is $\wt H$  of type  \eqref{sjui} such that the  first four moments of the entries of $H$
and   $\wt H$ coincide.  A precise statement is the following  lemma, about a single random variable. The proof is easily adapted from \cite[Corollary 30]{TaoVu2011} and it is left to the reader.

\begin{lemma}\label{4M} Let $H$ be a band matrix satisfying the conditions in  Theorem \ref{Main}. 
Then there  exists a matrix ensemble   $\wt H$ of the form \eqref{sjui} satisfying the assumptions of Theorem \ref{lemma main3}
such that 
$$
 \E (H_{ij})^n=\E (\wt H_{ij})^n, \quad |i-j|\le W, \quad n=1,2,3,4.
$$
\end{lemma}

\begin{proof} [Proof of Lemma \ref{shay}.]  Let $H $ be the matrix in Theorem \ref{Main} and  $\wt H$ the one given in Lemma \ref{4M}. 
Denote by $\wt G (z)=(\wt H-z)^{-1}$ the Green function of $\tilde H$.  By  Theorem \ref{lemma main3}, \eqref{maindel} and  \eqref{SemiC} hold
for the eigenvalues and eigenvectors of $\tilde H$.
Together with  Lemma \ref{Equ},  we get
\be\label{zhizl}
\sup_{|E|\le 2-\kappa, N^{-1}\le \eta\le 1}\|\wt G (z)\|_{\max}=
\OO( N^\tau), z=E+\ii\eta
\ee
with probability  $1-N^{-D}$.  We now prove that the same estimate holds for $G$, i.e. 
\be\label{zhizlttx}
\sup_{|E|\le 2-\kappa, N^{-1}\le \eta\le 1}\|G (z)\|_{\max}=
\OO( N^\tau).
\ee

We follow the self-consistent  comparison method from \cite{aniso}.  
We start with a very weak estimate  $\|G(z)\|_{\max}\leq\eta^{-1}$, i.e.,   \eqref{zhizlttx} holds  for $\eta\sim 1$.  For $\eta< 1$, let $\e_0>0$ be a small parameter and define
$$
\eta_m=N^{-m\e_0}, \quad z_m=E+\ii\eta_m, \quad 1\le m\le \e_0.
$$
Our goal is to prove  by induction  that    for $z=z_m, 1\le m\le \e_0^{-1}$,  \eqref{zhizlttx} holds, which 
implies \eqref{zhizl}. \nc  Thus it remains to prove that if \eqref{zhizlttx} holds for $z=z_m'$, $0\le m'\le m$, then \eqref{zhizlttx} holds for $z=z_{m+1}$, $1\le m+1\le \e_0^{-1}$.
  
As in \cite{aniso}, we define the symmetric interpolation matrix $H^\theta$ by  
\be\label{Ht}
 (H^\theta)_{ij}=(1-\chi^\theta_{ij})H^0_{ij}+ \chi^\theta_{ij} H^1 _{ij},
\quad H^1=H, \quad H^0=\wt H
\ee
where for $i\le j$,  $\chi^\theta_{ij}$  are i.i.d. Bernoulli random variable such that $\P(\chi_{ij}^\theta=1)=\theta$.   
Denote $G^\theta(z)=(H^\theta-z)^{-1}$. We can now recast the induction as follows: if for any (small) $\tau$ and (large) $D$, and  $|E|< 2-\kappa$,  we have  
$$
\max_{0\le \theta\le 1} \max_{  m'\le  m}\| G^\theta (z_{m'})\|_{\max}=\OO( N^\tau), \quad z_{m'}=E+\ii\eta_{m'}  
$$
then for any $\tau$ and $D$, and $|E|< 2-\kappa$,  we have
\be\label{zhizlttx3}
\max_{0\le \theta\le 1}  \| G^\theta (z_{m+1})\|_{\max}=\OO(N^\tau), \quad z_{m+1}=E+\ii\eta_{m+1} . 
\ee
We know that  \eqref{zhizlttx3} holds  for $\theta=0$ and all $m\le \e_0^{-1}$.
Our aim is to prove \eqref{zhizlttx3} for $0< \theta\le 1$.

From 
\cite[Lemma 10.2]{BenKno2018},   
we have 
$\|  G(E+\ii\eta/r)\|_{\max}\leq r\| G(E+\ii\eta)\|_{\max}$ for any  $r>1$. As a consequence, 
$
\max_{0\le n\le m}\|  G^\theta (z_n)\|_{\max}=\OO( N^\tau)$ implies   $\|G^\theta (z_{m+1})\|_{\max}= \OO(N^{\tau+\e_0})$.
 Thus it  remains to show that, under the assumption \eqref{zhizl}, if  we have     
  \be\label{jjdu} 
\max_{0\le \theta\le 1}  \| G^\theta (z_{m+1})\|_{\max}=\OO(N^{\tau+\e_0})
\ee
 then (\ref{zhizlttx3}) also holds.
 
 By \eqref{zhizl}, for any  $p\in \N$ and 
for any $\tau>0$, we have 
\be\label{zhizlwe}
 \max_{kl}\E|G^{\theta=0} _{kl}  (z_{m+1})|^{2p}=\OO( N^{2p\tau}).
\ee
We will use the following Lemma  from \cite{aniso}  to extend the above bound  to general $\theta\in [0,1]$:
\be\label{zhizlwe2}
 \max_\theta\max_{kl}\E|G^{\theta } _{kl}  (z_{m+1})|^{2p}=\OO( N^{2p\tau}),
\ee
which completes the proof  of \eqref{zhizlttx3} by Markov's inequality.  

\begin{lemma} \label{lem:interpol}
For fixed $i, j\in Z_N$, and $\lambda \in \R$, 
we define the matrix
\begin{equation*}
\pb{H^{\theta,\lambda}_{(ij)}}_{ab}=
\begin{cases}
\lambda & \hbox{if } \{a,b\} =\{i,j\},
\\
H^\theta_{ab} & \hbox{if } \{a,b\} \ne \{i,j\}\,.
\end{cases}
\end{equation*}
 For bounded and smooth $F \col \R^{N\times N} \to \C$ we have
 $$
\partial_\theta\E F(H^\theta)=\sum_{i\le j} \left( \E F (H^{\theta,H^1_{ij}}_{(i j)}) - \E F (H^{\theta,H^0_{ij}}_{(ij)})\right).
$$
\end{lemma}

We now return to prove \eqref{zhizlwe2}. 
Choose the function $F$ as follows:
 $$F(X) =F_{kl, p,z}(X)=\left|((X-z)^{-1})_{kl}\right|^{2p}.$$
By \eqref{zhizlwe},   for any $\tau> 0$ and  $p\in\mathbb{N}$, 
$
\E F(H^0)=\OO(N^{2\tau p}).
$
Thus, \eqref{zhizlwe2} for $H^\theta, 0\le \theta\le 1$, follows from  Gronwall's inequality and the following inequality, to be proved in the remainder of this paragraph (here and below, $z=z_{m+1}$): for any $p\geq 100$, there exists $c>0$ such that
\begin{equation}\label{LGZ}
  \left|\partial_\theta\E F(H^\theta) \right|  = \Big|\sum_{i,j}\left( \E F\pB{H^{\theta,H^1_{ij}}_{(i j)}} - \E F\pB{H^{\theta,H^0_{ij}}_{(ij)}}\right)\Big| 
\le  N^{-c}\left(1+   \E F\pB{H^{\theta}}\right)
\end{equation}
 for  any $0\le \theta\le 1$. Note that  the above equality is Lemma  \ref{lem:interpol}.

The matrices $H^{\theta,H^1_{ij}}_{(i j)}$ and $H^{\theta,H^0_{ij}}_{(i j)}$ are identical except for the entries $(i,j)$ and $(j,i)$ when $|i-j|\le W$, so we now compare them by a perturbative argument. 
We fix $i,j$ and define 
$$
f(\lambda)=f_{ij,kl,p,z,\theta}(\lambda):=F_{kl,p,z} \pB{H^{\theta,\lambda}_{(i j)}}.$$ 
By definition,  
$ f(H^\theta_{ij})
=F\pB{H^{\theta } }$
with $H^{\theta } $ as in \eqref{Ht}.
The $n$-th derivative of $f$, $f^{(n)}$,  is a sum of products of some $2p+n$ matrix entries of the resolvent and its conjugate.   From  \eqref{jjdu},   we therefore have  
$$
f^{(n)}(H^\theta_{ij})=\OO(N^{(\tau+\e_0)(2p+n)})
$$
with high probability.
By standard iterated resolvent identities, the same bound holds  for any $y=\OO(W^{-1/2+\tau})$:
$$
f^{(n)}(y)=\OO(N^{(\tau+\e_0) (2p+n)}).
$$
with overwhelming probability. Hence, by Taylor's expansion  with respect to $y=0$, for any $m\geq 1$,  we have
$$\E F\pB{H^{\theta,H^1_{ij}}_{(i j)}} - \E F\pB{H^{\theta,H^0_{ij}}_{(ij)}}=
\sum_{5\le n\le  m}\frac{\E(f^{(n)}(0))}{n!}\left(\E ((H^1_{ij})^n)-\E ((H^0_{ij})^n)\right)
+\OO
\left(N^{(\tau+\e_0)(2p+m+1)}
(W^{-\frac{1}{2}-\tau})^{m+1}
\right) 
$$
where we used that the first four moments of $H^1_{ij}$ and $H^0_{ij}$ are the same.   On the one hand, we choose $m=p$ so that the
above error term is at most $\OO(N^{-\frac{p}{10}})$ when $\tau+\e_0<1/100$.
On the other hand,
$f^{(n)}(0)$ is a sum of products of some $2p+n$ matrix entries of the resolvent and its conjugate,
 among which  at least $2p-n$ are either 
$
\left(H^{\theta,0}_{(ij)}-z\right)^{-1}_{kl}
$
or
its conjugate. 
With the resolvent identity, these two quantities  are easily  bounded by 
$ 
|G^\theta_{kl}|+W^{-1/2+\e}
$ for any $\e>2(\tau+\e_0)$, with high probability.  
The remaining $2n$ resolvent entries are bounded using (\ref{jjdu}).
Therefore, for $n\le p$,  
$$
\left|\E f^{(n)}(0)\right| \le C_p\, N^{2n(\tau+\e_0)} \left(\E(|G^\theta_{kl}|^{2p-n})+W^{(-1/2+\e)(2p-n)}
\right)
\leq 
C_p\, N^{2n(\tau+\e_0)} \left(1+\E(F(H^\theta))\right).
$$ 
The above estimates together give
$$
  \left|\partial_\theta\E F(H^\theta) \right|
\le C_p NW^{1-\frac{5}{2}+10(\tau+\e_0)}(1+\E F(H^{\theta}))
+
C_p N^{1-\frac{p}{10}} W.
$$
As $W \ge N^{3/4}$ and $p\geq 100$, this concludes the proof of the inequality in \eqref{LGZ} (and Lemma \ref{shay}).
\end{proof}

\subsection{Proof of Theorem \ref{lsc}}.
We keep the notations from the proof of  Lemma \ref{shay} for $H$ and $\wt H$. On the one hand, from the local law in Theorem \ref{lemma main3}, for any $\kappa, \tau>0$ there exist $\e>0$ such that
for any $z=E+N^{-1+\tau}$ with $-2+\kappa<E<2-\kappa$,
$\im N^{-1}{\rm Tr}\, \wt G(z)-\im m_{\rm sc}(z)=\OO(N^{-\e})$.
On the other hand, by repeating exactly the proof of Lemma \ref{shay}, the estimate (\ref{LGZ}) also holds for 
$
F(X)=\left|\im N^{-1}{\rm Tr}(X-z)^{-1}-\im m_{\rm sc}(z)\right|^{2p},
$
so that $\im N^{-1}{\rm Tr}\, G(z)-\im m_{\rm sc}(z)=\OO(N^{-c})$ for some $c>0$.  In turn, this implies the local law for $H$.

\subsection{Proof of Theorem \ref{Univ}}.
Again, we follow the notations from the proof of  Lemma \ref{shay} for $H$ and $\wt H$. 
 Theorem \ref{lemma main3} gives universality for $\wt H$, so that  
Theorem \ref{Univ} follows  by applying the Green's functions comparison theorem from \cite{ErdYauYin2012Univ}.
The input for this theorem is the four moment matching of the matrix entries, given by construction of $\wt H$, and resolvent bounds as proved for our band matrices in (\ref{zhizlttx}).

\setcounter{equation}{0}
\setcounter{theorem}{0}
\renewcommand{\theequation}{A.\arabic{equation}}
\renewcommand{\thetheorem}{A.\arabic{theorem}}
\appendix
\setcounter{secnumdepth}{0}
\section{Appendix  \ \ Perfect matching observables for Hermitian matrices}\label{complex}

Although the main reults of this paper are stated for symmetric matrices, they can be adapted to the Hermitian class. The only major modification concerns the definition of the perfect matching observables. We explain below the Hermitian counterpart of Section \ref{sec:QUE1}.

\subsection{Eigenvector dynamics.}\ Let $B^{(h)}$ be a $n\times n$ matrix such that $\Re(B^{(h)}_{ij}),\Im(B^{(h)}_{ij})  (i<j)$ and $B^{(h)}_{ii}/\sqrt{2}$ are independent standard Brownian motions, and $B^{(h)}_{ji}=
(B^{(h)}_{ij})^*$. The $n\times n$ Hermitian Dyson Brownian motion $K^{(h)}$ with initial value $K^{(h)}(0)$
is
$
K^{(h)}(t)=K^{(h)}(0)+\frac{1}{\sqrt{2n}}B^{(h)}(t).
$

Let  $\bla_0\in\Sigma_n$, $\boldu_0\in \UU(n)$. The Hermitian Dyson Brownian motion/vector flow with initial condition  $(\lambda_1,\dots,\lambda_n)=\bla_0$, $(u_1,\dots,u_n)=\boldu_0$,  is 
\begin{align}
\rd\la_k&=\frac{\rd B^{(h)}_{kk}}{\sqrt{2n}}+\left(\frac{1}{n}\sum_{\ell\neq k}\frac{1}{\la_k-\la_\ell}\right)\rd t,\notag\\
\rd u_k&=\frac{1}{\sqrt{2n}}\sum_{\ell\neq k}\frac{\rd B^{(h)}_{k\ell}}{\lambda_k-\lambda_\ell}u_\ell
-\frac{1}{2n}\sum_{\ell\neq k}\frac{\rd t}{(\la_k-\la_\ell)^2}u_k.\label{eqn:eigenvectorsHermitian}
\end{align}
With the above definitions, the strict analogue of Theorem \ref{thm:PCE} holds in this Hermitian setting.

In addition to (\ref{eqn:cij}) and (\ref{eqn:ukd}), we define
\begin{align}
&u_k\partial_{\overline{u}_\ell}=\sum_{\al=1}^nu_k(\al)\partial_{\overline{u}_\ell(\al)},\notag\\
&X_{k\ell}^{(h)}=u_k\partial_{u_\ell}-\overline{u}_\ell\partial_{\overline{u}_k},\notag\
\overline{X}_{k\ell}^{(h)}=\overline{u}_k\partial_{\overline{u}_\ell}-u_\ell\partial_{u_k}.\notag
\end{align}
Here  $\partial_{u_\ell}$ and $\partial_{\overline{u}_\ell}$ are defined by considering 
$u_\ell$ as a complex number, i.e., if we write $u_\ell= x+ \ii y$ then 
$\partial_{{u}_\ell}= \frac 1 2 (\partial_x - \ii \partial_y) $. The analogue of Lemma \ref{lem:generator} for the generator is then (see \cite{BouYau2017})
\begin{equation}
\LL_t^{(h)}=\frac{1}{2}\sum_{1\leq k<\ell\leq n}c_{k\ell}(t)\left(X_{k\ell}^{(h)}\overline{X}_{k\ell}^{(h)}+\overline{X}_{k\ell}^{(h)}X_{k\ell}^{(h)}\right), \label{LH}
\end{equation}
meaning that
$\rd\E( g (\boldu_t))/\rd t=\E(\LL_t^{(h)} g(\boldu_t))$)
for the stochastic differential equation  (\ref{eqn:eigenvectorsHermitian}).

\subsection{The observables.}\ 
As in Section \ref{sec:QUE1}, let $I$ be a fixed subset of $\llbracket 1,n\rrbracket$, and denote the eigenvector overlaps 
\begin{align*}
p_{ij}&=\sum_{\alpha\in I}u_i(\alpha)\bar u_j(\alpha),\ \ i\neq j\in\llbracket 1,n\rrbracket\\
p_{ii}&=\sum_{\alpha\in I}u_i(\alpha)\bar u_i(\alpha)-C_0,\ \ i\in\llbracket 1,n\rrbracket
\end{align*}
where $C_0$ is an arbitrary but fixed constant independent of $i$. Note that contrary to the real case, we now have $p_{ij}\neq p_{ji}$ for $i\neq j$.

With this definition, Theorem \ref{thmQUE} still holds.
For the proof, we keep the same definition for our configuration space as in the real case: $\boeta:  \llbracket 1,n\rrbracket \to \bN$ where
$\eta_j:=\boeta(j)$ is interpreted as the number of particles at the site $j$.
For any given configuration $\boeta$, consider the set of vertices
$$
\mathcal{V}_{\boeta}=\{(i,a,\e): 1\leq i\leq n, 1\leq a\leq \eta_i,\e\in\{b,w\}\}.
$$
We  represent vertices corresponding to $\e=b$ (resp. $\e=w$) by a black (resp. white) disk.
Let $\mathcal{A}_{\boeta}$ be the graph with vertices $\mathcal{V}_{\boeta}$ and with edges all possible $\{v_1,v_2\}$ with $\e_1\neq\e_2$, where $v_1=(i_1,a_1,\e_1)$, $v_2=(i_2,a_2,\e_2)$. In words, $\mathcal{A}_{\boeta}$ is the complete graph on $\mathcal{V}_{\boeta}$ except that edges between vertices of the same color are forbidden.
Let $\mathcal{G}_{\boeta}$ be the set of perfect matchings of $\mathcal{A}_{\boeta}$. Let  $\mathcal{E}(G)$ be the set of edges of a graph $G\in\mathcal{G}_{\boeta}$.

\begin{figure}[h]
\centering
\begin{subfigure}{.4\textwidth}
\centering
\vspace{0.6cm} \hspace{0cm}
\begin{tikzpicture}[scale=0.5]
\draw[fill,black] (1,1) circle [radius=0.2];
\draw[fill,black] (1,2) circle [radius=0.2];
\draw[fill,black] (4,1) circle [radius=0.2];
\draw[fill,black] (4,2) circle [radius=0.2];
\draw[fill,black] (4,3) circle [radius=0.2];
\draw[fill,black] (6,1) circle [radius=0.2];
\draw [thick,->,black] (-2,1) -- (9,1);
\draw [thick,black] (-1,0.8) -- (-1,1.2);
\draw [thick,black] (8,0.8) -- (8,1.2);

\node at (-1,0.4) {\nc 1};
\node at (1,0.4) {\nc $i_1$};
\node at (4,0.4) {\nc $i_2$};
\node at (6,0.4) {\nc $i_3$};
\node at (8,0.4) {\nc $n$};
  
\end{tikzpicture}
\vspace{0.1cm}
  \caption{A configuration $\boeta$ with $\mathcal{N}(\boeta)=6$, $\eta_{i_1}=2$, $\eta_{i_2}=3$, $\eta_{i_3}=1$.
 }
   \label{fig:sub1}
\end{subfigure}
\hspace{1cm}
\begin{subfigure}{.4\textwidth}
  \centering
\begin{tikzpicture}[scale=0.5]

\draw [ultra thick,black] (0.6,1) to[out=45,in=225] (1.4,2);
\draw [ultra thick,black] (0.6,2) to[out=55,in=120] (4.4,3);
\draw [ultra thick,black] (1.4,1) to[out=35,in=220] (3.6,3);
\draw [ultra thick,black] (4.4,1) to[out=100,in=80] (3.6,1);
\draw [ultra thick,black] (4.4,2) to[out=0,in=120] (5.6,1);
\draw [ultra thick,black] (3.6,2) to[out=65,in=100] (6.4,1);

\draw [dashed,->,black] (-2,1) -- (9,1);
\draw [thick,black] (-1,0.8) -- (-1,1.2);
\draw [thick,black] (8,0.8) -- (8,1.2);
\draw[fill,black] (0.6,1) circle [radius=0.2];
\draw[fill,black] (0.6,2) circle [radius=0.2];
\draw[black,fill=white] (1.4,1) circle [radius=0.2];
\draw[black,fill=white] (1.4,2) circle [radius=0.2];
\draw[fill,black] (3.6,1) circle [radius=0.2];
\draw[fill,black] (3.6,2) circle [radius=0.2];
\draw[fill,black] (3.6,3) circle [radius=0.2];
\draw[black,fill=white] (4.4,1) circle [radius=0.2];
\draw[black,fill=white] (4.4,2) circle [radius=0.2];
\draw[black,fill=white] (4.4,3) circle [radius=0.2];
\draw[fill,black] (5.6,1) circle [radius=0.2];
\draw[black,fill=white] (6.4,1) circle [radius=0.2];

\node at (-1,0.4) {\nc 1};
\node at (1,0.4) {\nc $i_1$};
\node at (4,0.15) {\nc $i_2$};
\node at (6,0.4) {\nc $i_3$};
\node at (8,0.4) {\nc $n$};
\end{tikzpicture}
  \caption{A perfect matching $G\in\mathcal{G}_{\boeta}$. Here, $P(G)=p_{i_1i_1}p_{i_1i_2}p_{i_2i_1}p_{i_2i_2}p_{i_2i_3}p_{i_3i_2}$.}
  \label{fig:sub1}
\end{subfigure}
\end{figure}

Moreover, for any given edge $e=\{(i_1,a_1,\e_1),(i_2,a_2,\e_2)\}$, we define  
$p(e)=p_{i_1,i_2}$ if $\e_1=b$, $p(e)=p_{i_2,i_1}$ if $\e_2=b$. Let
$
P(G)=\prod_{e\in \mathcal{E}(G)}p(e)
$ and
\begin{equation}\label{feqC}
f^{(h)}_{\bla, t}(\boeta)=\frac{1}{\mathcal{L}(\boeta)}\E\left(\sum_{G\in\mathcal{G}_{\boeta}} P(G)\mid \bla\right),\ \ \mathcal{L}(\boeta)=\prod_{i=1}^n \eta_i!.
\end{equation}
We have the following complex analogue of Theorem \ref{thm:EMF}.

\begin{theorem}[Perfect matching observables for the eigenvector moment flow: Hermitian case]\label{thm:EMFC}
Suppose that $\boldu$ is the solution to the Hermitian eigenvector dynamics (\ref{eqn:eigenvectorsHermitian}),
and $ f^{(h)}_{\bla, t}$ is given by \eqref{feqC}.
Then 
it satisfies the equation  
\begin{align}
\notag
&\partial_t  f^{(h)}_{\bla, t}  = \mathscr{B}^{(h)}(t) f^{(h)}_{\bla, t} ,\\
\notag&\mathscr{B}^{(h)}(t)  f(\boeta) = \sum_{k \neq \ell} c_{k\ell}(t) \eta_k (1+ \eta_\ell) \left( f(\boeta^{k, \ell})-f(\boeta)\right).
\end{align}
\end{theorem}

\noindent As in the real case, the above theorem is independent of our choice of the canonical basis, see Remark \ref{rem:canon}. It therefore generalizes the class of observables
for the eigenvector moment flow from \cite[Theorem 3.1 (ii)]{BouYau2017}.

\subsection{Proof of Theorem \ref{thm:EMFC}.}\  \label{proofofEMFC}  
We naturally replace the definition (\ref{eqn:gR}) with
$$
g(\boeta)=\frac{1}{\mathcal{L}(\boeta)}\sum_{G\in\mathcal{G}_{\boeta}} P(G)
$$
and let $1\leq k<\ell\leq n$ be fixed for the rest of this subsection. We abbreviate $X=X_{k\ell}^{(h)}, \overline{X}=\overline{X}_{k\ell}^{(h)}$. With (\ref{LH}) the proof reduces to 
\begin{equation}\label{eqn:applicationC}
\frac{1}{2}(X\bar X+\bar X X) g(\boeta)=\eta_k(1+\eta_\ell) (g(\boeta^{k\ell})-g(\boeta))+\eta_\ell(1+\eta_k) (g(\boeta^{\ell k})-g(\boeta)).
\end{equation}
To calculate  $\frac{1}{2}(X\bar X+\bar X X)  P(G)$ for any $G\in\mathcal{G}_{\boeta}$ we first need the following definition.

\begin{definition} Let $\boeta$ and $k<\ell$ be fixed.
\begin{enumerate}[(i)]
\item $\mathcal{V}_i\subset\mathcal{V}_{\boeta}$ is the set of vertices of type $(i,a,\e)$, $1\leq a\leq \eta_i$, $\e\in\{b,w\}$. 
\item $\mathcal{V}^b_i\subset\mathcal{V}_i$ is the set of vertices of type $(i,a,b)$, $1\leq a\leq \eta_i$, and similarly for $\mathcal{V}^w_i$.
\item 
For any two sets, denote $A\cdot B=(A\times B)\cup (B\times A)$.
We define 
$$
\epsilon(v_1,v_2)=\left\{
\begin{array}{ll}
1&{\rm if}\ (v_1,v_2)\in(\mathcal{V}_k^b\cdot \mathcal{V}_k^w)\cup(\mathcal{V}_\ell^b\cdot \mathcal{V}_\ell^w),\\
-1&{\rm if}\ (v_1,v_2)\in(\mathcal{V}_k^b\cdot \mathcal{V}_\ell^b)\cup(\mathcal{V}_k^w\cdot \mathcal{V}_\ell^w),\\
0&{\rm otherwise}.
\end{array}
\right.
$$
\item Let $G\in\mathcal{G}_{\boeta}$ and $(v_1,v_2)\in (\mathcal{V}_k\cup\mathcal{V}_\ell)^2_*$.

Assume $(v_1,v_2)\in(\mathcal{V}_k^b\cdot\mathcal{V}_\ell^b)\cup(\mathcal{V}_k^w\cdot\mathcal{V}_\ell^w)$. Then we define $S_{v_1v_2}G=S_{v_2v_1}G\in\mathcal{G}_{\boeta}$ as the perfect matching obtained by transposition of $v_1$ and $v_2$. More precisely, let $\tau_{v_1v_2}$ be the permutation of $\mathcal{V}_{\boeta}$ transposing $v_1$ and $v_2$. Then 
$$
\mathcal{E}(S_{v_1v_2}G)=\{\{\tau_{v_1,v_2}(w_1),\tau_{v_1,v_2}(w_2)\}: \{w_1,w_2\}\in\mathcal{E}(G)\}.
$$

Assume  $(v_1,v_2)\in\mathcal{V}_k^b\cdot\mathcal{V}_k^w$, and write $v_1=(k,a_1,b)$, $v_2=(k,a_2,w)$ for example, where $1\leq a_1,a_2\leq \eta_k$. Then we define $S_{v_1v_2}G=S_{v_2v_2}G\in\mathcal{G}_{\boeta^{k\ell}}$ as the perfect matching obtained by a jump of $v_1$ and $v_2$ to $\ell$. 
More precisely, let $j_{v_1v_2}=j_{v_2v_1}$ be the following bijection from $\mathcal{V}_{\boeta}$ to $\mathcal{V}_{\boeta^{k\ell}}$: 
$j_{v_1v_2}(v_1)=(\ell,\eta_\ell+1,b)$, $j_{v_1v_2}(v_2)=(\ell,\eta_\ell+1,w)$, $j_{v_1v_2}((k,c,b))=(k,c-1,b)$ if $a_1<c$, $j_{v_1v_2}((k,c,w))=(k,c-1,w)$ if $a_1<c$ and $j_{v_1v_2}(w_1)=w_1$ in all other cases. 
Then
$$
\mathcal{E}(S_{v_1v_2}G)=\{\{j_{v_1v_2}(w_1),j_{v_1v_2}(w_2)\}: \{w_1,w_2\}\in\mathcal{E}(G)\}.
$$
A similar definition applies if $(v_1,v_2)\in\mathcal{V}_\ell^b\cdot\mathcal{V}_\ell^w$, the jump now being towards $k$.

Finally, if $(v_1,v_2)\not\in(\mathcal{V}_k^b\cdot\mathcal{V}_\ell^b)\cup(\mathcal{V}_k^w\cdot\mathcal{V}_\ell^w)\cup\left(\mathcal{V}_k^b\cdot\mathcal{V}_k^w\right)\cup \left(\mathcal{V}_\ell^b\cdot\mathcal{V}_\ell^w\right)$, we define $S_{v_1v_2}G=G$ (or any arbitrary function).
\end{enumerate}
\end{definition}

\begin{figure}[h]
\centering
\begin{subfigure}{.45\textwidth}
\centering
\begin{tikzpicture}[scale=0.4]
  
\normalcolor
\node at (1,0.3) {\normalcolor$i$};
\node at (4,0.05) {\normalcolor$k$};
\node at (6,0) {\normalcolor$\ell$};

\draw [ultra thick] (0.6,1) to[out=45,in=225] (1.4,2);
\draw [ultra thick] (0.6,2) to[out=55,in=120] (4.4,3);
\draw [ultra thick] (1.4,1) to[out=35,in=220] (3.6,3);
\draw [ultra thick] (4.4,1) to[out=100,in=80] (3.6,1);
\draw [ultra thick] (4.4,2) to[out=0,in=120] (5.6,1);
\draw [ultra thick] (3.6,2) to[out=65,in=100] (6.4,1);

\draw [dashed,->] (-0.5,1) -- (7.5,1);
\draw[fill] (0.6,1) circle [radius=0.2];
\draw[fill] (0.6,2) circle [radius=0.2];
\draw[fill=white] (1.4,1) circle [radius=0.2];
\draw[fill=white] (1.4,2) circle [radius=0.2];
\draw[fill] (3.6,1) circle [radius=0.2];
\draw[fill] (3.6,2) circle [radius=0.2];
\draw[fill] (3.6,3) circle [radius=0.2];
\draw[fill=white] (4.4,1) circle [radius=0.2];
\draw[fill=white] (4.4,2) circle [radius=0.2];
\draw[fill=white] (4.4,3) circle [radius=0.2];
\draw[fill] (5.6,1) circle [radius=0.2];
\draw[fill=white] (6.4,1) circle [radius=0.2];

\draw(5.6,1) circle [radius=0.35];
\node at (5.2,0.2) {\normalcolor$v_2$};

\draw(3.6,3) circle [radius=0.35];
\node at (2.7,2.9) {\normalcolor$v_1$};

\draw [ultra thick] (10.6,1) to[out=45,in=225] (11.4,2);
\draw [ultra thick] (10.6,2) to[out=55,in=120] (14.4,3);
\draw [ultra thick] (11.4,1) to[out=-40,in=-150] (15.6,1);
\draw [ultra thick] (14.4,1) to[out=100,in=80] (13.6,1);
\draw [ultra thick] (14.4,2) ..controls (13.6,0.9) and (12.6,2) .. (13.6,3);
\draw [ultra thick] (13.6,2) to[out=65,in=100] (16.4,1);

\draw [dashed,->] (9.5,1) -- (17.5,1);
\draw[fill] (10.6,1) circle [radius=0.2];
\draw[fill] (10.6,2) circle [radius=0.2];
\draw[fill=white] (11.4,1) circle [radius=0.2];
\draw[fill=white] (11.4,2) circle [radius=0.2];
\draw[fill] (13.6,1) circle [radius=0.2];
\draw[fill] (13.6,2) circle [radius=0.2];
\draw[fill] (13.6,3) circle [radius=0.2];
\draw[fill=white] (14.4,1) circle [radius=0.2];
\draw[fill=white] (14.4,2) circle [radius=0.2];
\draw[fill=white] (14.4,3) circle [radius=0.2];
\draw[fill] (15.6,1) circle [radius=0.2];
\draw[fill=white] (16.4,1) circle [radius=0.2];

\draw [ultra thick,->] (7,4) to[out=30,in=150] (10,4);

\node at (8.5,5.3) {\normalcolor$S_{v_1v_2}$};

\end{tikzpicture}
  \caption{The map $S_{v_1v_2}$ in case of a transposition.
 }
   \label{fig:sub3}
\end{subfigure}
\hspace{1cm}
\begin{subfigure}{.45\textwidth}
\centering
\begin{tikzpicture}[scale=0.4]
    \normalcolor

\node at (1,0.3) {\normalcolor$i$};
\node at (4,0.05) {\normalcolor$k$};
\node at (6,0.3) {\normalcolor$\ell$};

\draw [ultra thick] (0.6,1) to[out=45,in=225] (1.4,2);
\draw [ultra thick] (0.6,2) to[out=55,in=120] (4.4,3);
\draw [ultra thick] (1.4,1) to[out=35,in=220] (3.6,3);
\draw [ultra thick] (4.4,1) to[out=100,in=80] (3.6,1);
\draw [ultra thick] (4.4,2) to[out=0,in=120] (5.6,1);
\draw [ultra thick] (3.6,2) to[out=65,in=100] (6.4,1);

\draw [dashed,->] (-0.5,1) -- (7.5,1);
\draw[fill] (0.6,1) circle [radius=0.2];
\draw[fill] (0.6,2) circle [radius=0.2];
\draw[fill=white] (1.4,1) circle [radius=0.2];
\draw[fill=white] (1.4,2) circle [radius=0.2];
\draw[fill] (3.6,1) circle [radius=0.2];
\draw[fill] (3.6,2) circle [radius=0.2];
\draw[fill] (3.6,3) circle [radius=0.2];
\draw[fill=white] (4.4,1) circle [radius=0.2];
\draw[fill=white] (4.4,2) circle [radius=0.2];
\draw[fill=white] (4.4,3) circle [radius=0.2];
\draw[fill] (5.6,1) circle [radius=0.2];
\draw[fill=white] (6.4,1) circle [radius=0.2];

\draw [ultra thick] (10.6,1) to[out=45,in=225] (11.4,2);
\draw [ultra thick] (10.6,2) to[out=35,in=140] (14.4,2);
\draw [ultra thick] (11.4,1)  ..controls (9,-1.3) and (9,6) .. (15.6,2);
\draw [ultra thick] (14.4,1) to[out=30,in=150] (15.6,1);
\draw [ultra thick] (13.6,1) ..controls (15,0) and (18,0) .. (16.4,2);
\draw [ultra thick] (13.6,2) to[out=-90,in=90] (16.4,1);

\draw [dashed,->] (9.5,1) -- (17.5,1);
\draw[fill] (10.6,1) circle [radius=0.2];
\draw[fill] (10.6,2) circle [radius=0.2];
\draw[fill=white] (11.4,1) circle [radius=0.2];
\draw[fill=white] (11.4,2) circle [radius=0.2];
\draw[fill] (13.6,1) circle [radius=0.2];
\draw[fill] (13.6,2) circle [radius=0.2];
\draw[fill=white] (14.4,1) circle [radius=0.2];
\draw[fill=white] (14.4,2) circle [radius=0.2];
\draw[fill] (15.6,1) circle [radius=0.2];
\draw[fill=white] (16.4,1) circle [radius=0.2];
\draw[fill] (15.6,2) circle [radius=0.2];
\draw[fill=white] (16.4,2) circle [radius=0.2];

\draw [ultra thick,->] (7,4) to[out=30,in=150] (10,4);

\draw(4.4,1) circle [radius=0.35];
\node at (5,0.2) {\normalcolor$v_2$};

\draw(3.6,3) circle [radius=0.35];
\node at (2.7,2.9) {\normalcolor$v_1$};

\node at (8.5,5.3) {\normalcolor$S_{v_1v_2}$};
\end{tikzpicture}
\vspace{-0.8cm}
  \caption{The map $S_{v_1v_2}$ in case of a jump.
 }
   \label{fig:sub4}
\end{subfigure}
\end{figure}

Below is the main result for the proof of Theorem \ref{thm:EMFC}.

\begin{lemma}\label{lem:keyC}
For any $G\in\mathcal{G}_{\boeta}$, we have
\begin{equation}\label{eqn:keyC}
\frac{1}{2}(X\bar X+\bar XX)P(G)=\frac{1}{2}\sum_{(v_1,v_2)\in (\mathcal{V}_k\cup\mathcal{V}_\ell)^2_*}\epsilon(v_1,v_2)P(S_{v_1v_2}G)-(\eta_k+\eta_\ell)P(G).
\end{equation}
\end{lemma}

Assuming the above lemma we can complete the proof of Theorem \ref{thm:EMFC}. Let 
$$
h(\boeta)=\sum_{G\in\mathcal{G}_{\boeta}}P(G).
$$
Note that if $(v_1,v_2)\in(\mathcal{V}_k^b\cdot \mathcal{V}_\ell^b)\cup(\mathcal{V}_k^w\cdot \mathcal{V}_\ell^w)$, then $S_{v_1v_2}$ is a permutation of $\mathcal{G}_{\boeta}$. Moreover, if $(v_1,v_2)\in\mathcal{V}_k^b\cdot \mathcal{V}_k^w$ (resp. $\mathcal{V}_\ell^b\cdot \mathcal{V}_\ell^w$) then
 $S_{v_1v_2}$ is a bijection from $\mathcal{G}_{\boeta}$ to  $\mathcal{G}_{\boeta^{k\ell}}$ (resp. $\mathcal{G}_{\boeta^{\ell k}}$). Summing (\ref{eqn:keyC}) over all $G\in\mathcal{G}_{\boeta}$ therefore gives
\begin{align*}
\frac{1}{2}\left(X\bar X+\bar X X\right)h(\boeta)=&\frac{1}{2}\left(2\eta_k^2h(\boeta^{k\ell})+2\eta_\ell^2h(\boeta^{\ell k})-(2\eta_k\eta_\ell+2\eta_\ell\eta_k)h(\boeta)-2(\eta_k+\eta_\ell)h(\boeta)\right)\\
=&
\eta_k^2h(\boeta^{k\ell})+\eta_\ell^2h(\boeta^{\ell k})-(\eta_k(\eta_\ell+1)+\eta_\ell(\eta_k+1))h(\boeta).
\end{align*}
The above equation implies (\ref{eqn:applicationC}) after renormalization by $\mathcal{L}(\boeta)$. This concludes the proof of Theorem \ref{thm:EMFC}.

\begin{proof}[Proof of Lemma \ref{lem:keyC}]
Let $L=\frac{1}{2}\left(X\bar X+\bar X X\right)$. We have
\begin{equation}\label{eqn:LeibnizC}
L P(G)=\sum_{(e_1,e_2)\in \mathcal{E}(G)^2_*}Xp(e_1)\bar Xp(e_2)\prod_{e\in\mathcal{E}(G)\backslash\{e_1,e_2\}}p(e)
+
\sum_{e_1\in \mathcal{E}(G)}Lp(e_1)\prod_{e\in\mathcal{E}(G)\backslash\{e_1\}}p(e).
\end{equation}
We keep the notations (\ref{edg1}), (\ref{edg2}), (\ref{edg3}) for the single, double and transverse edges.
Remember that for any $v\in\mathcal{V}_{\boeta}$,  $e_v$ is the unique edge containing $v$ and $v'$ be the unique vertex such that $e_v=\{v,v'\}$. We still denote
\begin{align*}
\mathcal{V}_s&=\{v\in \mathcal{V}_k\cup\mathcal{V}_\ell:\{v,v'\}\in\mathcal{E}_s\},\\
\mathcal{V}_d&=\{v\in \mathcal{V}_k\cup\mathcal{V}_\ell:\{v,v'\}\in\mathcal{E}_d\},\\
\mathcal{V}_t&=\{v\in \mathcal{V}_k\cup\mathcal{V}_\ell:\{v,v'\}\in\mathcal{E}_t\},
\end{align*}
and $\mathcal{V}_{k,s}^b$ the single, black vertices in $\mathcal{V}_k$  (and similarly for $\mathcal{V}_{k,s}^w$, etc).
We will need the following elementary rules: if $i\neq \ell$ and $j\neq k$, $Xp_{ij}=0$, and
\begin{align}
&X p_{ik}=-p_{i\ell}, X p_{\ell j}=p_{kj}\label{trans1C}\\
&X p_{\ell k}=p_{kk}-p_{\ell\ell},\label{trans2C}\\
&X p_{kk}=-p_{k\ell}, X p_{\ell\ell}=p_{kk}\label{trans3C}.
\end{align}
We also obviously have $\bar X p=\overline{X \bar p}$. 
Equation (\ref{eqn:LeibnizC}) can be written as
$$
L P(G)={\rm(I)+(II)+(III)+(IV)+(V)+(VI)+(VII)+(VIII)+(IX)}
$$
where all terms are defined and calculated below. First,
$$
{\rm(I)}:=\sum_{(e_1,e_2)\in (\mathcal{E}_s)^2_*}Xp(e_1)\bar Xp(e_2)\prod_{e\in\mathcal{E}(G)\backslash\{e_1,e_2\}}p(e)
=
\sum_{(v_1,v_2)\in (\mathcal{V}_s)^2_*}Xp_{\{v_1,v_1'\}}\bar Xp_{\{v_2,v_2'\}}\prod_{e\in\mathcal{E}(G)\backslash\{e_{v_1},e_{v_2}\}}p(e).
$$
From (\ref{trans1C}),  $Xp_{\{v_1,v_1'\}}\bar Xp_{\{v_2,v_2'\}}=-p_{\{v_2,v_1'\}}p_{\{v_1,v_2'\}}$ if $(v_1,v_2)\in(\mathcal{V}_{\ell,s}^b\times \mathcal{V}_{k,s}^b)\cup(\mathcal{V}_{k,s}^w\times\mathcal{V}_{\ell,s}^w)$, and 
$Xp_{\{v_1,v_1'\}}\bar Xp_{\{v_2,v_2'\}}=p_{\{j_{v_1,v_2}(v_1),v_1'\}}p_{\{j_{v_1,v_2}(v_2),v_2'\}}$ if $(v_1,v_2)\in(\mathcal{V}_{k,s}^b\times \mathcal{V}_{k,s}^w)\cup(\mathcal{V}_{\ell,s}^b\times\mathcal{V}_{\ell,s}^w)$.
In all other cases, 
$Xp_{\{v_1,v_1'\}}\bar Xp_{\{v_2,v_2'\}}=0$. We therefore proved
\begin{equation}\label{eqn:IC}
{\rm(I)}=\sum_{(v_1,v_2)\in (\mathcal{V}_{\ell,s}^b\times \mathcal{V}_{\ell,s}^b)\cup(\mathcal{V}_{k,s}^w\times\mathcal{V}_{\ell,s}^w)\cup(\mathcal{V}_{k,s}^b\times  \mathcal{V}_{k,s}^w)\cup(\mathcal{V}_{\ell,s}^b\times\mathcal{V}_{\ell,s}^w)} \e(v_1,v_2)P(S_{v_1v_2}G)
=
\frac{1}{2}\sum_{(v_1,v_2)\in (\mathcal{V}_s)^2_*} \e(v_1,v_2)P(S_{v_1v_2}G).
\end{equation}
We now consider
$$
{\rm(II)}:=\sum_{(e_1,e_2)\in \mathcal{E}_s\cdot\mathcal{E}_d}Xp(e_1)\bar Xp(e_2)\prod_{e\in\mathcal{E}(G)\backslash\{e_1,e_2\}}p(e)
=
\frac{1}{2}\sum_{(v_1,v_2)\in \mathcal{V}_s\cdot \mathcal{V}_d}Xp_{\{v_1,v_1'\}}\bar Xp_{\{v_2,v_2'\}}\prod_{e\in\mathcal{E}(G)\backslash\{e_{v_1},e_{v_2}\}}p(e).
$$
We used that vertices on a double edge need to be weighted by a factor $1/2$.
From (\ref{trans1C}) and (\ref{trans3C}),  
\begin{align*}
Xp_{\{v_1,v_1'\}}\bar Xp_{\{v_2,v_2'\}}&=-p_{\{v_2,v_1'\}}p_{\{v_1,v_2'\}}\ {\rm if}\ (v_1,v_2)\in\left(\mathcal{V}_{k,d}^w\times \mathcal{V}_{\ell,s}^w\right)\cup(\mathcal{V}_{\ell,s}^b\times \mathcal{V}_{k,d}^b)
\cup\left(\mathcal{V}_{\ell,d}^b\times \mathcal{V}_{k,s}^b\right)\cup(\mathcal{V}_{k,s}^w\times \mathcal{V}_{\ell,d}^w),\\
Xp_{\{v_1,v_1'\}}\bar Xp_{\{v_2,v_2'\}}&=p_{\{j_{v_1v_2}(v_1),v_1'\}}p_{\{j_{v_1v_2}(v_2),v_2'\}}\ {\rm if}\ (v_1,v_2)\in(\mathcal{V}_{k,d}^w\times\mathcal{V}_{k,s}^b)\cup(\mathcal{V}_{\ell,d}^b\times\mathcal{V}_{\ell,s}^w)
\cup(\mathcal{V}_{k,s}^w\times\mathcal{V}_{k,d}^b)\cup(\mathcal{V}_{\ell,s}^b\times\mathcal{V}_{\ell,d}^w).
\end{align*}
We therefore have
\begin{equation}\label{eqn:IIC}
{\rm(II)}=\frac{1}{2}\sum_{(v_1,v_2)\in \mathcal{V}_s\cdot\mathcal{V}_d} \e(v_1,v_2)P(S_{v_1v_2}G).
\end{equation}
Concerning
$$
{\rm(III)}:=\sum_{(e_1,e_2)\in (\mathcal{E}_d)^2_*}Xp(e_1)\bar Xp(e_2)\prod_{e\in\mathcal{E}(G)\backslash\{e_1,e_2\}}p(e)
=
\frac{1}{4}\sum_{(v_1,v_2)\in (\mathcal{V}_d)^2_*:v_1\neq v_2'}Xp_{\{v_1,v_1'\}}\bar Xp_{\{v_2,v_2'\}}\prod_{e\in\mathcal{E}(G)\backslash\{e_{v_1},e_{v_2}\}}p(e),
$$
using (\ref{trans3C}) we have $Xp_{\{v_1,v_1'\}}\bar Xp_{\{v_2,v_2'\}}=-p_{\{v_2,v_1'\}}p_{\{v_1,v_2'\}}$ if $v_1$ and $v_2$ are in distinct $\mathcal{V}_{i}$'s and with the same color, 
$p_{\{j_{v_1v_2}(v_1),v_1'\}}p_{\{j_{v_1v_2}(v_2),v_2'\}}$ if they are in the same $\mathcal{V}_{i}$ with distinct colors. All together, we always have
$Xp_{\{v_1,v_1'\}}\bar Xp_{\{v_2,v_2'\}}=\e(v_1,v_2)P(S_{v_1v_2}G)+\e(v_1',v_2)P(S_{v_1'v_2}G)$.
We therefore proved
\begin{align}\label{eqn:IIIC}
{\rm(III)}&=\frac{1}{4}\sum_{(v_1,v_2)\in (\mathcal{V}_d)^2_*} \left(\e(v_1,v_2)P(S_{v_1v_2}G)+\e(v_1',v_2)P(S_{v_1'v_2}G)\right)-\frac{1}{2}\sum_{v\in\mathcal{V}_d} P(S_{vv'}G)\notag\\
&=\frac{1}{2}\sum_{(v_1,v_2)\in (\mathcal{V}_d)^2_*} \e(v_1,v_2)P(S_{v_1v_2}G)-\frac{1}{2}\sum_{v\in\mathcal{V}_d} P(S_{vv'}G).
\end{align}
Our next estimate is a diagonal term, namely
\begin{equation}\label{eqn:IVC}
{\rm(IV)}:=
\sum_{e_1\in \mathcal{E}_s}Lp(e_1)\prod_{e\in\mathcal{E}(G)\backslash\{e_1\}}p(e)
=
\sum_{v\in \mathcal{V}_s}Lp_{\{v,v'\}}\prod_{e\in\mathcal{E}(G)\backslash\{e_v\}}p(e)
=
-\frac{1}{2}\sum_{v\in \mathcal{V}_s}P(G)
\end{equation}
where we used  (\ref{trans1C}) twice to obtain $Lp_{\{v,v'\}}=-\frac{1}{2}p_{\{v,v'\}}$.

Another diagonal term is
$$
{\rm(V)}:=
\sum_{e_1\in \mathcal{E}_d}Lp(e_1)\prod_{e\in\mathcal{E}(G)\backslash\{e_1\}}p(e)
=
\frac{1}{2}\sum_{v\in \mathcal{V}_d}Lp_{\{v,v'\}}\prod_{e\in\mathcal{E}(G)\backslash\{e_1\}}p(e).
$$
Note that we have
$Lp_{\{v,v'\}}=p_{kk}-p_{\ell\ell}$ if $v\in\mathcal{V}_\ell$, $p_{\ell\ell}-p_{kk}$ otherwise. This proves
\begin{equation}\label{eqn:VC}
{\rm(V)}=
\frac{1}{2}\sum_{v\in \mathcal{V}_d}(P(S_{vv'}(G))-P(G)).
\end{equation}
We now consider cases where transverse edges appear:
\begin{align*}
&{\rm(VI)}:=\sum_{(e_1,e_2)\in \mathcal{E}_s\times\mathcal{E}_t\cup \mathcal{E}_t\times\mathcal{E}_s}Xp(e_1)\bar Xp(e_2)\prod_{e\in\mathcal{E}(G)\backslash\{e_1,e_2\}}p(e)\\
=&
\sum_{v_1\in \mathcal{V}_s,\{v_2,v_2'\}\in\mathcal{E}_t}\left(Xp_{\{v_1,v_1'\}}\bar Xp_{\{v_2,v_2'\}}+\bar Xp_{\{v_1,v_1'\}} Xp_{\{v_2,v_2'\}}\right)\prod_{e\in\mathcal{E}(G)\backslash\{e_{v_1},e_{v_2}\}}p(e).
\end{align*}
Up to transposing $v_2$ and $v_2'$, we can assume that $v_1$ and $v_2$ are in the same $\mathcal{V}_i$. With (\ref{trans1C}) and (\ref{trans2C}) a calculation gives 
$Xp_{\{v_1,v_1'\}}\bar Xp_{\{v_2,v_2'\}}+\bar Xp_{\{v_1,v_1'\}} Xp_{\{v_2,v_2'\}}=p_{j_{v_1v_2}(v_1)v_1'}p_{j_{v_1v_2}(v_2)v_2'}-p_{\tau_{v_1v_2'}(v_1)v_1'}p_{\tau_{v_1v_2'}(v_2')v_2}$. This yields
\begin{equation}\label{eqn:VIC}
{\rm(VI)}=
\sum_{(v_1,v_2)\in\mathcal{V}_s\times\mathcal{V}_t}\e(v_1,v_2)P(S_{v_1v_2}(G))
=
\frac{1}{2}\sum_{(v_1,v_2)\in\mathcal{V}_s\cdot\mathcal{V}_t}\e(v_1,v_2)P(S_{v_1v_2}(G)).
\end{equation}
We also consider
\begin{align*}
&{\rm(VII)}:=\sum_{(e_1,e_2)\in \mathcal{E}_d\times\mathcal{E}_t\cup \mathcal{E}_t\times\mathcal{E}_d}Xp(e_1)\bar Xp(e_2)\prod_{e\in\mathcal{E}(G)\backslash\{e_1,e_2\}}p(e)\\
=&
\sum_{v_1\in \mathcal{V}_d,\{v_2,v_2'\}\in\mathcal{E}_t}\left(Xp_{\{v_1,v_1'\}}\bar Xp_{\{v_1,v_2'\}}+\bar Xp_{\{v_1,v_1'\}} Xp_{\{v_1,v_2'\}}\right)\prod_{e\in\mathcal{E}(G)\backslash\{e_v,e_w\}}p(e).
\end{align*}
Without loss of generality we can assume $v_1$ and $v_2$ are in the same $\mathcal{V}_i$. Assume they also have a different color. Then (\ref{trans2C}) and (\ref{trans3C}) give $Xp_{\{v_1,v_1'\}}\bar Xp_{\{v_1,v_2'\}}+\bar Xp_{\{v_1,v_1'\}} Xp_{\{v_2,v_2'\}}=
p_{j_{v_1v_2}(v_1)v_1'}p_{j_{v_1v_2}(v_2)v_2'}-p_{\tau_{v_1v_2'}(v_1)v_1'}p_{\tau_{v_1v_2'}(v_2')v_2}$. If $v_1$ and $v_2$ have the same color, a similar equation holds, permuting $v_1$ and $v_1'$. This implies
\begin{equation}\label{eqn:VIIC}
{\rm(VII)}=
\sum_{(v_1,v_2)\in\mathcal{V}_d\times\mathcal{V}_t}\e(v_1,v_2)P(S_{v_1v_2}(G))=
\frac{1}{2}\sum_{(v_1,v_2)\in\mathcal{V}_d\cdot\mathcal{V}_t}\e(v_1,v_2)P(S_{v_1v_2}(G)).
\end{equation}
For two transverse edges, with (\ref{trans2C}) we first compute $\frac{1}{2}(X p_{k\ell}\bar X p_{k\ell}+\bar X p_{k\ell} X p_{k\ell})=0$, and indeed $\e(v_1,v_2)=0$ when $v_1,v_2$ are the same color on the same site, or different colors on different sites.  Moreover, $\frac{1}{2}(X p_{k\ell}\bar X p_{\ell k}+\bar X p_{k\ell} X p_{\ell k})=\frac{1}{2}(p_{kk}^2+p_{\ell\ell}^2-2p_{kk}p_{\ell\ell})$, so that in all cases we proved
\begin{equation}\label{eqn:VIIIC}
{\rm(VIII)}:=\sum_{(e_1,e_2)\in (\mathcal{E}_t)^2_*}Xp(e_1)\bar Xp(e_2)\prod_{e\in\mathcal{E}(G)\backslash\{e_1,e_2\}}p(e)
=
\frac{1}{2}\sum_{(v_1,v_2)\in(\mathcal{V}_t^2)_*}\e(v_1,v_2)P(S_{v_1v_2}(G).
\end{equation}
Finally, from (\ref{trans2C}) we have $L p_{k\ell}=-p_{k\ell}$, so that
\begin{equation}\label{eqn:IXC}
{\rm(IX)}:=
\sum_{e_1\in \mathcal{E}_t}Lp(e_1)\prod_{e\in\mathcal{E}(G)\backslash\{e_1\}}p(e)
=
-\frac{1}{2}\sum_{v\in \mathcal{V}_t}P(G)
\end{equation}
By summation of all equations (\ref{eqn:IC}), (\ref{eqn:IIC}), (\ref{eqn:IIIC}), (\ref{eqn:IVC}), (\ref{eqn:VC}), (\ref{eqn:VIC}), (\ref{eqn:VIIC}), (\ref{eqn:VIIIC}), (\ref{eqn:IXC}), the right hand sides of (\ref{eqn:keyC}) and (\ref{eqn:Leibniz})
are the same. This concludes the proof of Lemma \ref{lem:keyC}.
\end{proof}

  \end{document}